\documentclass[11pt]{article}
\usepackage{multicol}
\usepackage[colorlinks=true,linkcolor=black,citecolor=black,hyperfootnotes=false]{hyperref}
\usepackage{float}
\usepackage{amssymb}
\usepackage{stmaryrd}
\usepackage{mathalfa}
\usepackage{euscript}
\usepackage{extarrows}

\usepackage[normalem]{ulem}
\usepackage{amsmath}
\usepackage{amsthm}
\numberwithin{equation}{section}
\newtheorem{thm}{Theorem}

\newtheorem*{thmc*}{Coherence Theorem}
\newtheorem{lem}{Lemma}
\newtheorem{cor}{Corollary}
\newtheorem{rem}{Remark}
\newtheorem{nota}{Notation}
\newtheoremstyle{bre}  
  {\topsep}   
  {\topsep}   
  {\itshape}  
  {0pt}       
  {\bfseries} 
  {.}         
  {5pt plus 1pt minus 1pt}  
  {#1 #2*}          
\theoremstyle{bre}

\theoremstyle{definition}
\newtheorem{defn}{Definition}

\newtheorem{example}{\normalfont\scshape Example}
 
\newtheorem{fact}{\normalfont\itshape Fact}
\allowdisplaybreaks
\usepackage{tabularx}
\usepackage{xcolor}
\usepackage{framed}
\usepackage{tikz}
\usetikzlibrary{decorations.markings}
\usetikzlibrary{shapes.geometric}
\usepackage{booktabs}
\usepackage{euscript}
\usepackage[normalem]{ulem}
\usepackage{array}
\usepackage{collectbox}

\usepackage{footnote}
\usepackage{makecell}
\pgfdeclarelayer{edgelayer}
\pgfdeclarelayer{nodelayer}
\pgfsetlayers{edgelayer,nodelayer,main}
\usepackage{array}
\usepackage{collectbox}
\usetikzlibrary{decorations.pathmorphing, patterns,shapes}
\newcommand{\xupdownarrow}[1]{%
  {\left\Updownarrow\vbox to #1{}\right.\kern-\nulldelimiterspace}
}

\makeatletter
\newcommand{\mybox}{%
    \collectbox{%
        \setlength{\fboxsep}{5pt}%
        \fbox{\BOXCONTENT}%
    }%
}
\makeatother

\usepackage{multirow}

\newcommand \seq[2]{\shortstack{$#1$ \\ \mbox{}\\
                    \mbox{}\hrulefill\mbox{}\\ \mbox{}\\ $#2$}}
 \usetikzlibrary{calc}
\newcolumntype{P}[1]{>{\centering\arraybackslash}p{#1}}
\newcolumntype{M}[1]{>{\centering\arraybackslash}m{#1}}

\makeatletter
\providecommand{\leftsquigarrow}{%
  \mathrel{\mathpalette\reflect@squig\relax}%
}
\newcommand{\reflect@squig}[2]{%
  \reflectbox{$\m@th#1\rightsquigarrow$}%
}
\makeatother
\title{A formal language for cyclic operads} 
\author
{Pierre-Louis Curien$^{1}$ and Jovana Obradovi\' c$^{2}$ \\
\small{\em IRIF,  Universit\' e Paris Diderot and $\pi r^2$ team,  
 Inria, France}\\
{\small$^{1}$\url{curien@pps.univ-paris-diderot.fr}}\enspace\enspace\enspace
\small$^{2}$\url{jovana@pps.univ-paris-diderot.fr}}

\date{}

\begin{document} 

\maketitle 
\begin{abstract}
\noindent We propose a $\lambda$-calculus-style formal language, called the $\mu$-syntax, as a lightweight representation of the structure of cyclic operads. We illustrate the rewriting methods behind the formalism by giving a complete step-by-step proof of the equivalence between the unbiased and biased definitions of cyclic operads. 
\end{abstract}
\section*{Introduction}
In the spirit of recent years' movement in bringing closer mathematics and computer science communities 
through  formalisation of mathematics, this paper proposes a  $\lambda$-calculus-style formal language, called the $\mu$-syntax, as a lightweight representation of the  cyclic operad structure.  \\
\indent   The name and the language of the $\mu$-syntax formalism were motivated by another formal syntactical tool,  the $\mu\tilde{\mu}$-subsystem of the $\overline{\lambda}\mu\tilde{\mu}$-calculus, presented by  Curien and Herbelin in \cite{doc}. In their paper, programs are described by means of expressions called commands, of the form\vspace{-0.10cm} $$\langle\mu\beta.c_1\,|\,\tilde{\mu}x.c_2\rangle ,\vspace{-0.1cm}$$ which exhibit a computation as the result of an interaction between a term $\mu\beta.c_1$ and an evaluation context $\tilde{\mu}x.c_2$, together with a symmetric reduction system \vspace{-0.10cm}$$c_2[\mu\beta.c_1 / x]\longleftarrow\langle\mu\beta.c_1\,|\,\tilde{\mu}x.c_2\rangle\longrightarrow c_1[\tilde{\mu}x.c_2 / \beta] ,\vspace{-0.10cm}$$ accounting for the symmetry of calling mechanisms in programming languages.   In our syntactical approach, we follow this idea and view operadic composition as such a program, i.e. as an interaction between two operations $f$ and $g$, where $f$ provides an input $x$ (selected with $\tilde{\mu}$) for the output $\beta$ of $g$ (marked with ${\mu}$). By moving this concept to the {\em entries-only framework of cyclic operads} \cite[Definition 48]{mm}, in which an operation, instead of having inputs and an   output, now has only {\em entries}, and  can be composed with another operation along any of them, the input/output distinction of  the $\mu\tilde{\mu}$-subsystem goes away, leading to   the existence of a {\em single binding operator} $\mu$, whose purpose is to select the entries of two operations which are to be connected in this interaction.\\ \indent Concretely,  the pattern $\langle \mu x.\underline{\hspace{0.3cm}}\, |\, \mu y.\underline{\hspace{0.3cm}}\,\rangle$  encodes the partial composition operation $(-){{_{x}\circ_{y}}}(-)$. Hence, from the tree-wise perspective, $\langle \mu x.\underline{\hspace{0.3cm}}\, |\, \mu y.\underline{\hspace{0.3cm}}\,\rangle$ encodes   the unrooted tree obtained by grafting two unrooted trees along entries (or half-edges, or flags) $x$ and $y$. For those combinatorially oriented, this construction (and, in particular, the syntactic concept of  binding) can also be seen in terms of differentiation of species of Joyal \cite{joyal}, as a mapping $\partial S\cdot \partial S\rightarrow S$, where $\partial S$ is the derivative of the  species $S$ and $\cdot$ denotes the product of species. In fact, in \cite{mo},  cyclic operads are defined internally to the category of species by using precisely this mapping, and an equivalence with the   representation by means of individual composition operations $(-) {_x\circ}_y (-)$ (and, therefore, with $\langle \mu x.\underline{\hspace{0.3cm}}\, |\, \mu y.\underline{\hspace{0.3cm}}\,\rangle$) is set up. In addition to commands of the form $\langle \mu x.\underline{\hspace{0.3cm}}\, |\, \mu y.\underline{\hspace{0.3cm}}\,\rangle$, which describe {\em partial} grafting of two unrooted trees, the $\mu$-syntax features another kind of commands, whose shape is $(-)\{\mu x.\underline{\hspace{0.3cm}},\dots,\mu y.\underline{\hspace{0.3cm}}\}$, and which describe {\em simultaneous} grafting of unrooted trees. Such a command encodes the unrooted tree obtained by grafting to {\em all the entries} of  the corolla $(-)$ the unrooted trees within the brackets, along their respective entries bound by $\mu$. Therefore, the command $(-)\{\mu x.\underline{\hspace{0.3cm}},\dots,\mu y.\underline{\hspace{0.3cm}}\}$  is to the command  $\langle \mu x.\underline{\hspace{0.3cm}}\, |\, \mu y.\underline{\hspace{0.3cm}}\,\rangle$ what the original notion of simultaneous operadic composition of \cite{GILS} is to the notion of partial operadic composition of  \cite{mss},   in the framework of cyclic operads. The equations of the $\mu$-syntax identify different constructions on unrooted trees that should be regarded as being the same, and the  $\mu$-syntax in whole is easilly mapped to the algebraic formalism of cyclic operads.\\[0.1cm]
\indent The advantage of the $\mu$-syntax over the  usual ``mathematical'' definitions of cyclic operads is tangible from two perspectives. On one hand, if one lays down the two  usual ways of defining cyclic operads, the biased way (resulting in definitions via generators and relations \cite[Theorem 2.2]{Getzler:1994pn}, \cite[Definition 48]{mm}), and the unbiased way (leading to the definition via monads \cite[Definition 2.1]{Getzler:1994pn}), one would argue that these look quite
formidable. This is due to the underlying intricate combinatorial structure of unrooted trees. The commands of the $\mu$-syntax play the role of trees, but with the benefit of being rather simple in-line formulas. Accordingly, the equations of the $\mu$-syntax make a crisp representaion of the cumbersome  laws of the composition of cyclic operads. Summed up, the $\mu$-syntax {\em makes the long story short(er)}.\\
\indent On the other hand, in the spirit of Leibniz's {\em characteristica universalis} and {\em calculus ratiocinator}, the usefulness of the $\mu$-syntax arises when the   question about the completeness, rigour and formalisability of mathematical proofs is asked. This especially concerns  long and involved proofs, which are common in operad theory. Such a proof is, for example, the proof of the equivalence between the biased   and unbiased  definitions of cyclic operads, which is a well-known result (cf. \cite[Theorem 2.2]{Getzler:1994pn}, \cite[Section 5]{Kaufmann}, \cite[Section 4.2]{manin}). The above  requirements, typically asked for in computer science, reflect through out  syntactical proof of this equivalence, as follows. The {\em formalisability} property is met here by fixing a universal syntactic language in which the proof is presented. The   internal structural patterns of this language are convenient for describing in a step-by-step fashion the transitions involved in this proof.  In order to meet the {\em rigour}  requirement, all the involved structures are spelled out in detail. In particular, the correct treatment of the identities of the appropriate monad structure is given.
    Finally, as required by the  {\em completeness} property, the proof that the laws satisfied by an  algebra over the  monad indeed come down to the axioms from the biased definition, is explicitly given.  We  shall make a syntactic reformulation of the monad of unrooted trees figuring in the unbiased definition, which, together with  the $\mu$-syntax, makes a syntactic framework well-suited for  a complete step-by-step proof of the equivalence.\\[-0.75cm]
\paragraph{Layout.}   In Section 1, we recall the biased entries-only definition \cite[Definition 48]{mm} and the unbiased definition \cite[Definition 2.1]{Getzler:1994pn} of cyclic operads. For the latter definition, this involves a syntactic reformulation and a detailed description of   the monad of unrooted trees.
The section  finishes with the statement of the theorem that expresses the equivalence between the two definitions.   Section 2 will be devoted to the introduction and analysis of the $\mu$-syntax.  In Section 3, we employ the $\mu$-syntax in crafting the proof of the equivalence from Section 1.   \\[-0.75cm]
\paragraph{Notation and conventions.} {\em About cyclic operads}. This paper is about non-skeletal set-based cyclic operads. Non-skeletality means that the entries of   operations are labeled by arbitrary finite sets, rather than by natural numbers (as done in the skeletal approach). 
This is just a matter of convenience and a practice coming from  computer science:  we prefer the non-skeletal setting  because we prefer   formulas with ``named"  (rather than ``numbered") variables, and we chose to work in {\bf Set} (rather than in an arbitrary symmetric monoidal category) only to be able to (correctly)  speak about operadic operations in terms of elements. We  assume the existence of  operadic units.\\[0.1cm]
\indent {\em About finite sets and bijections.} Conforming to the computer science practice, in this paper  we assume that a sufficiently large universe of finite sets is fixed (denumerable is enough). Union will always be the {\em ordinary} union of {\em already} disjoint sets. For disjoint finite sets $X$ and $Y$, $X\cup Y$ shall stand for the  union of $X$ and $Y$. For a bijection $\sigma:X'\rightarrow X$ and  $Y\subseteq X$, we shall  denote with $\sigma|^{Y}$ the corestriction of $\sigma$ on $\sigma^{-1}(Y)$.  For $y\not\in X\cup X'$, we  denote with $\sigma_y$  the bijection $\sigma_{y}:X'\cup\{y\}\rightarrow X\cup\{y\}$, defined as $\sigma$ on $X'$, and such that $\sigma_y(y)=y$. If $\sigma(x')=x$,  we  denote with $\sigma^{y/x'}$ the bijection defined in the same way as $\sigma$, except that, instead of $x'$, it contains $y$ in its domain (the inverse image of $x$ now being $y$). If $\tau:Y'\rightarrow Y$ is a bijection such that $X'\cap Y'=X\cap Y=\emptyset$,  then $\sigma\cup\tau:X'\cup Y'\rightarrow X\cup Y$ denotes the bijection defined  as $\sigma$ on $X'$ and as $\tau$ on $Y'$. Finally, if $\kappa:X\backslash\{x\}\cup\{x'\}\rightarrow X$ is the identity on $X\backslash\{x\}$ and $\kappa(x')=x$, we say that $\kappa$ renames $x$ to $x'$ (notice the contravariant nature of this convention). \\[0.1cm]
\indent {\em About type-theoretical notions.} For a comprehensive account on the terminology of type theory and rewriting theory, whose basic notions we shall use in this paper, we refer to \cite{types} and \cite{rewriting}. We list here the essentials. \\
\indent We assume given an infinite set $V$ of {\em variables}, or {\em names} (countable is enough). 
We say that a variable $x$ is {\em fresh with respect to a set} $X$ if $x\not\in X$. The existence of $V$ assures that for any finite set, there exists a variable which is fresh with respect to that set.\\
\indent A {\em multi-sorted formal theory} is a formal theory for which variables, constant  symbols and function symbols, as well as all the terms built from them, have a property called {\em sort} or  {\em type}. Types serve  to   control  the  formation  of terms and to classify them.    A model   of  a multi-sorted formal theory, i.e. of a typed formal language, is a model in the usual sense, which additionally takes into account  sorts of the symbols of the signature of the theory. In other words, the domain of such a model  is a collection of sets $\{{\cal M}(s_i)\}_{i\in I}$, indexed by all sorts of the theory, and the interpretation function   maps  constant symbols  of sort   $s_i$ to the set  ${\cal M}(s_i)$, for all $i\in I$, and  function symbols of sort $(s_1,\dots,s_n;s)$ to functions of the form ${\cal M}(s_1)\times\cdots\times  {\cal M}(s_n)\rightarrow  {\cal M}(s)$.

An {\em abstract rewriting system} (a rewriting system for short) is a pair $(A,\rightarrow)$, where $A$ is  a set and $\rightarrow$ is a binary relation on $A$. The name is supposed to indicate that an element $(a,b)$ of $\rightarrow$ should be seen as a rewriting of $a$ into $b$. We write $a\rightarrow b$ to denote that $(a,b)\in \,\rightarrow$. An element $a\in A$ is a {\em normal form} for $\rightarrow$ if there does not exist $a'\in A$, such that $a\rightarrow a'$. We say that a rewriting system $(A,\rightarrow)$ is {\em terminating} if there does not exist an infinite sequence $a_1\rightarrow a_2\rightarrow\cdots\rightarrow  a_n\rightarrow \cdots $ of elements of $A$. We denote with $\xlongrightarrow{\ast}$ the reflexive and transitive closure of $\rightarrow$. A rewriting system $(A,\rightarrow)$ is {\em confluent} if, for any triple $(a,a_1,a_2)$ of elements of $A$, such that $a\xlongrightarrow{\ast}a_1$ and $a\xlongrightarrow{\ast}a_2$, there exists $a'\in A$, such that $a_1\xlongrightarrow{\ast}a'$ and $a_2\xlongrightarrow{\ast}a'$. A rewriting system $(A,\rightarrow)$ is {\em locally confluent} if, for any triple $(a,a_1,a_2)$ of elements of $A$, such that $a\rightarrow a_1$ and $a\rightarrow a_2$, there exists $a'\in A$, such that $a_1\xlongrightarrow{\ast}a'$ and $a_2\xlongrightarrow{\ast}a'$.
\begin{fact}\label{fact1}
If $(A,\rightarrow)$ is terminating, then it is {\em normalising}, i.e. for any $a\in A$, there exists a normal form $a'$, such that $a\xlongrightarrow{\ast}a'$.
\end{fact}
\begin{fact}\label{fact2}
If $(A,\rightarrow)$ is terminating and confluent, then for $a\in A$, there exists a {\em unique} normal form $a'$, such that $a\xlongrightarrow{\ast}a'$.
\end{fact}
\begin{fact}\label{fact3}
If $(A,\rightarrow)$ is terminating, then it is confluent if and only if it is locally confluent.
\end{fact}
In this paper, we shall examine certain {\em term rewriting systems}, i.e. abstract rewriting systems  $(A,\rightarrow)$, for which the set $A$ is the set of terms of some syntax, and the rewriting relation $\rightarrow$ is obtained by orienting some of the equations of the syntax.

 \section{Cyclic operads}\label{s1}
Operads encode  categories of algebras whose operations have multiple inputs and one output, such as associative  algebras, commutative algebras,  Lie algebras, etc. The interest in encoding more general algebraic structures was a part of the {\em renaissance of operads} in the early nineties of the last century, when various generalizations of operads came into existence. The formalism of cyclic operads was originally introduced by Getzler and Kapranov in \cite{Getzler:1994pn}. The enrichment of the operad structure determined by the definition of a cyclic operad  is provided by adding to the action of permuting the inputs of an operation an action of interchanging its output  with one of the inputs. This feature essentially makes the distinction between the inputs and the output no longer visible, which is adequately captured by unrooted trees as pasting schemes for operations of a cyclic operad. In other words, cyclic operads can be seen as  generalisations of operads for which an operation, instead of having inputs and an   output, now has only ``entries", and   can be composed with another operation along any of them.  As for the formal description of  composition of such operations, the {\em unbiased} and {\em biased} frameworks provide two ways to complete the characterisation of a cyclic operad.  
\subsection{Biased definition of cyclic operads}\label{uff}
 In the biased (entries-only) approach, the  definition of a cyclic operad is biased towards ``local" operadic compositions $_x\circ_y$,  in the sense that these are the only explicitly defined concepts. The various ways to derive a global operadic composition are then equated by the appropriate  axioms.   We revisit below Markl's definition  \cite[Definition 48]{mm},  for a particular case when the underlying functor is  ${\underline{\EuScript C}}:{\bf Bij}^{op}\rightarrow {\bf Set}$, and by adapting it further by also demanding operadic units.   In the sequel, for  $f\in{\underline{\EuScript C}}(X)$ and a bijection $\sigma:X'\rightarrow X$, we write $f^{\sigma}$ instead of ${\underline{\EuScript C}}(\sigma)(f)$.\vspace{-0.1cm}
 \begin{defn}\label{entriesonly}
A {\em cyclic operad} is a  functor ${{\EuScript C}}:{\bf Bij}^{op}\rightarrow {\bf Set}$, together with a distinguished element ${\it id}_{x,y}\in {{\EuScript C}}(\{x,y\})$ for each two-element set $\{x,y\}$, and a partial composition operation \vspace{-0.15cm}
$${{_{x}\circ_{y}}}:{{\EuScript C}}(X)\times {{\EuScript C}}(Y)\rightarrow {{\EuScript C}}(X\backslash\{x\}\cup Y\backslash\{y\}) , \vspace{-0.1cm}$$
defined for arbitrary non-empty finite sets $X$ and $Y$ and elements $x\in X$ and $y\in Y$, such that $X\backslash\{x\}\cap Y\backslash \{y\}=\emptyset .$
These data   satisfy the   axioms given below, wherein, for each of the axioms, we assume the  set disjointness that ensures that all the partial compositions involved are well-defined.  \\[0.15cm]
{\em Sequential associativity.} For $f\in {{\EuScript C}}(X)$, $g\in {{\EuScript C}}(Y)$, $h\in {{\EuScript C}}(Z)$,  $x\in X$, $y,u\in Y$ and  $z\in Z$, the following   equality holds:\\[0.15cm]
\indent \texttt{(A1)} $(f\, {_{x}\circ_{y}}\,\, g)\,\,{_{u}\circ_z}\, h = f\, {_{x}\circ_{y}}\,\, (g\, {_{u}\circ_z}\, h)$.\\[0.15cm]
{\em Commutativity.} For  $f\in{{\EuScript C}}(X)$, $g\in {{\EuScript C}}(Y)$, $x\in X$ and $y\in Y$, the following equality holds:\\[0.15cm]
\indent \texttt{(CO)} $f\, {_x\circ_y} \,\, g=g\, {_y\circ_x} \,\, f$.\\[0.15cm]
{\em Equivariance.} For bijections $\sigma_1:X'\rightarrow X$, $\sigma_2:Y'\rightarrow Y$ and $\sigma=\sigma_1|^{X\backslash\{x\}}\cup \sigma_2|^{Y\backslash\{y\}}$, and $f\in{{\EuScript C}}(X)$ and $g\in {{\EuScript C}}(Y)$, the following equality holds:\\[0.15cm] 
\indent \texttt{(EQ)} $f^{\sigma_1}\,\,{_{{ {\sigma_1^{-1}}(x)}}\circ_{\sigma_2^{-1}(y)}}\,\, g^{\sigma_2}=(f {_x\circ_y} \,\, g)^{\sigma}$.\\[0.15cm]
{\em Right Unitality.} For $f\in\EuScript{C}(X)$, $x\in X$ and a bijection $\sigma$ that renames $x$ to $z$, the following two equalities hold:\\[0.15cm]
\indent \texttt{(U1)} $f\,\,{_x\circ_y}\,\, {\it id}_{y,z}=f^{\sigma}$.\\[0.15cm]
Moreover, the unit elements are preserved under the action of ${{\EuScript C}}(\sigma)$, i.e.\\[0.15cm] 
\indent \texttt{(U3)} ${id_{x,y}}^{\sigma}=id_{u,v}$, \\[0.15cm]
for any two two-element sets $\{x,y\}$ and $\{u,v\}$, and a bijection $\sigma:\{u,v\}\rightarrow\{x,y\}$.\\[0.1cm]
\indent For $f\in {{\EuScript C}}(X)$, the elements of the set $X$ are called the {\em entries} of $f$.\hfill$\square$
\end{defn}
Note that  we impose a slightly weaker condition on the sets $X$ and $Y$ and elements $x\in X$ and $y\in Y$  involved in partial composition than in \cite[Definition 48]{mm}: instead of requiring $X$ and $Y$ to be disjoint, as Markl does, we allow the possibility that they intersect, provided that their intersection is a subset of $\{x,y\}$. This also means that we allow the possibility that $x= y$. Nevertheless, the characterizations of  Definition \ref{entriesonly} and \cite[Definition 48]{mm}, with units added, are equivalent. As for the units, here is a notational remark.
\begin{nota}
It is understood that ${\it id}_{x,y}={\it id}_{y,x}$. We reserve the notation ${\it id}_{\{x,y\}}$ for the identity bijection on the two-element set $\{x,y\}$.
\end{nota}

\indent The lemma below gives  basic properties of the partial composition operation.
\begin{lem}\label{par} The partial composition operation from Definition \ref{entriesonly} satisfies the following laws.\\[0.15cm]
{\em Parallel associativity}. For $f\in {{\EuScript C}}(X)$, $g\in {{\EuScript C}}(Y)$, $h\in {{\EuScript C}}(Z)$,  $x,u\in X$, $y\in Y$ and $z\in Z$, the following   equality holds:\\[0.15cm]
\indent {\em\texttt{(A2)}} $(f\, {_{x}\circ_{y}}\,\, g)\,\,{_{u}\circ_z}\, h =(f\, {_{u}\circ_{z}}\,\, h)\,\,{_{x}\circ_y}\, g$.\\[0.15cm]
{\em Left unitality}. For $f\in\EuScript{C}(X)$, $x\in X$ and a bijection $\sigma$ that renames $x$ to $z$, the following   equality holds:\\[0.15cm]
\indent  {\em\texttt{(U2)}} ${\it id}_{y,z}\,\,{_y\circ_x}\,\,f =f^{\sigma}$.
\end{lem}
\begin{proof}
For \texttt{(A2)}, combine \texttt{(A1)} and \texttt{(CO)}. For \texttt{(U2)}, combine \texttt{(U1)} and \texttt{(CO)}.
\end{proof}
\indent Definition \ref{entriesonly} naturally incorporates the notion of {\em simultaneous composition}, as a sequence of partial compositions of the form as in the law \texttt{(A2)} from Lemma \ref{par}, that is, in which the entry involved in the next instance of a composition always comes from  $f\in{{\EuScript C}}(X)$ and which, moreover, ends when {\em all} the entries of $f\in{{\EuScript C}}(X)$ are exhausted.  In order to avoid writing explicitly such sequences, we introduce the following notation. For $f\in{{\EuScript C}}(X)$,  let\vspace{-0.15cm} $$\varphi:x\mapsto (Y_x,g_x,\underline{x})\vspace{-0.15cm}$$ be an assignment that associates to each $x\in X$ a finite set $Y_x$, an operation $g_x\in{{\EuScript C}}(Y_x)$ and an element $\underline{x}\in Y_x$, in such a way that\vspace{-0.1cm} $$\bigcap_{x\in X}Y_x\backslash\{\underline{x}\}=\emptyset.\vspace{-0.1cm}$$ Let, moreover, $\sigma:X'\rightarrow X$ be an arbitrary bijection such that for all $x\in X$, \vspace{-0.1cm}$$X'\backslash\{\sigma^{-1}(x)\}\cap Y_{x}\backslash\{\underline{x}\}=\emptyset .\vspace{-0.1cm}$$ Under these assumptions, the composite assignment\vspace{-0.1cm} $$\varphi\circ\sigma:x'\mapsto (Y_{\sigma(x')},g_{\sigma(x')},\underline{\sigma(x')}),\vspace{-0.1cm}$$ defined for all $x'\in X'$, together with $f^{\sigma}\in {\EuScript C(X')}$, determines the composition\vspace{-0.1cm} $$((f^{\sigma} \, {_{x'}\circ_{\underline{\sigma(x')}}}\, g_x) \, {_{y'}\circ_{\underline{\sigma(y')}}}\, g_y) \, {_{z'}\circ_{\underline{\sigma(z')}}}\, g_z \cdots ,\vspace{-0.1cm} $$ consisting of a sequence of partial compositions indexed by the entries of $f^{\sigma}$.
We will use the abbreviation $f^{\sigma}(\varphi\circ\sigma)$ to denote such a composition. Thanks to   \texttt{(A2)},  $f^{\sigma}(\varphi\circ\sigma)$ does not depend on the order in which the partial compositions were carried out. We finally set \begin{equation}\label{simultaneous}
f(\varphi)=f^{\sigma}(\varphi\circ\sigma),
\end{equation} and refer to $f(\varphi)$ as {\em the simultaneous composition determined by $f$ and $\varphi$}. That $f(\varphi)$  does not depend on the choice of $\sigma$ is a consequence of   \texttt{(EQ)}.\\
\indent Notice that without the renaming role of $\sigma$,  $f(\varphi)$ is not necessarily well-defined. For example, $f(\varphi)=(f \, {_{x}\circ_{\underline{x}}}\, g_x) \, {_{y}\circ_{\underline{y}}}\, g_y$, where $f\in{{\EuScript C}}(\{x,y\})$, $g_x\in{{\EuScript C}}(\{\overline{x},y\})$ and $g_y\in{{\EuScript C}}(\{\overline{y},v\})$,  is not well-defined, although $\varphi$ satisfies the required disjointness condition.\\
\indent In relation  to the above construction, the statements of the following lemma are  easy consequences of the axioms from  Definition \ref{entriesonly}.\begin{lem}\label{geneq} The simultaneous composition $f(\varphi)$ has the following properties.\begin{itemize}
\item[a)] Let $\psi: Z\rightarrow\bigcup_{x\in X}(Y_x\backslash\{\underline{x}\})$ be a bijection such that for all $x\in X$, $\underline{x}\not\in \psi^{-1}(Y_x\backslash\{\underline{x}\})$. Denote with $\psi_{\underline x}$ the   extension on $Y_x$ of the bijection $\psi|^{Y_x\backslash\{\underline{x}\}}$, which is identity on $\underline{x}$, and let
 $\varphi_{\psi}$ be  defined as ${\varphi}_{\psi}:x\mapsto (g_x^{\psi_{\underline{x}}},\underline{x})$, for all $x\in X$. Then $f(\varphi)^{\psi}=f(\varphi_{\psi}) .$
\item[b)] Let $\psi:y\mapsto(h_y,\underline{y})$ be an assignment that associates to each $y\in\bigcup_{x\in X}(Y_x\backslash\{\underline{x}\})$ an operation $h_y\in{{\EuScript C}}(Z_y)$ and  $\underline{y}\in Z_y$, in such a way that $f(\varphi)(\psi)$ is defined. If $\varphi_{\psi}$ is the assigment defined as $\varphi_{\psi}:x\mapsto(g_x^{\psi_{\underline x}},\underline{x})$, where $\psi_{\underline x}$ denotes the extension on $Y_x$ of the assignment $\psi|_{Y_x\backslash\{\underline{x}\}}$, which is identity on $\underline{x}$, then $f(\varphi)(\psi)=f(\varphi_{\psi}) .$\end{itemize}
\end{lem}
\indent The generators-and-relations nature of Definition \ref{entriesonly} allows us to easily formalise  cyclic operads   as models of the  multi-sorted equational theory which we now introduce. \\
\indent    The signature  of this theory is determined by taking as sorts all finite sets, while, having denoted with $s$ the sort of a constant symbol  and with $(s_1,\dots,s_n;s)$ the sort of an $n$-ary function symbol, as constant symbols we take the collection consisting of $${\it id}_{x,y}: \{x,y\}$$ and, as function symbols, we take the collection consisting of 
 
 $$\sigma  :  (Y;X)  \mbox{ (of arity $1$) \quad and\quad }   {_{x}\circ_{y}} : (X,Y ; X\backslash\{x\} \cup Y\backslash\{y\})  \mbox{ (of arity $2$)},$$ where $x,y\in V$ and $\sigma$ ranges over all bijections of finite sets. Here, $V$ is the infinite set of variables (i.e. names) whose existence we postulated in the Introduction.
 
Fixing a collection of {\em sorted variables}, or {\em parameters} $P$, and denoting with $P(X)$ the collection of parameters whose sort is $X$, the terms of the theory are built in the usual way:
\begin{center}
\mybox{
$s,t::= a\enspace |\enspace {\it id}_{x,y}\enspace |\enspace s\,{_{x}\circ_{y}}\, t\enspace |\enspace t^{\sigma}  $
}
\end{center} whereas the assignment of sorts to terms is done by the following  rules:  
\begin{center}
\mybox{ 
$\displaystyle\frac{a\in { P}(X)}{a:X}$ \enspace\enspace\enspace\enspace\enspace\enspace $\displaystyle\frac{}{{\it id}_{x,y}:\{x,y\}}$ \enspace\enspace \enspace\enspace\enspace\enspace $\displaystyle\frac{s:X\quad t:Y}{s\, {_{x}\circ_{y}}\,\, t: X\backslash\{x\}\cup Y\backslash\{y\}}$ \enspace\enspace\enspace\enspace\enspace\enspace $\displaystyle\frac{t:X\quad \sigma:(Y; X)}{t^{\sigma}:Y}$ }
\end{center} where $x$ and $y$ are distinct variables in the second rule, while, in the third rule,   $x\in X$, $y\in Y$ and $X\backslash\{x\}\cap Y\backslash\{y\}=\emptyset$. The equations of the theory are derived from the axioms of Definition \ref{entriesonly}, and there are two additional equations, namely \begin{equation}\label{additional}{\it id}_{x,y}^{\sigma}=id_{u,v}\quad\quad\mbox{and}\quad\quad (t^{\sigma})^{\tau}=t^{\sigma\circ\tau},\end{equation} where, in the first equation, $\sigma:(\{u,v\};\{x,y\})$.  

\begin{defn}\label{d2}  A cyclic operad   is a  model of the equational theory from above.
 \hfill$\square$
\end{defn}
That this definition indeed describes the same structure as does Definition \ref{entriesonly} is clear from the requirements that models of multi-sorted   theories fulfill.   The domain of such a model is a collection of sets ${\EuScript C}(X)$, arising by interpreting  all sorts $X$, and  the interpretation  of the remaining of the signature in this universe exhibits the cyclic operad structure in the obvious way.   Observe that the equations \eqref{additional} ensure that the assignment ${\EuScript C}:{\bf Bij}^{\it op}\rightarrow {\bf Set}$, induced   by the model,  is  functorial.\\[0.1cm]
\indent Let $\underline{\EuScript C}:{\bf Bij}^{\it op}\rightarrow {\bf Set}$ be a functor and  let \begin{equation}\label{pc}P_{\underline{\EuScript C}}=\{a\in\underline{\EuScript C}(X)\,|\, X \mbox{ is a finite set}\} \end{equation} be the collection of {\em parameters of} $\underline{\EuScript C}$. Observe that $P_{\underline{\EuScript C}}$ can be considered as a collection of sorted variables for the equational theory introduced above. In this regard,
we call the  syntax of terms built over $P_{\underline{\EuScript C}}$   the {\em combinator syntax generated by} $\underline{\EuScript C}$ and we refer to   terms as {\em combinators}. We shall denote the set of all combinators induced by  $\underline{\EuScript C}$ by $\tt{cTerm}_{\underline{\EuScript C}}$, and, for a finite set $X$, ${\tt{cTerm}}_{\underline{\EuScript C}}(X)$ will be used to denote  the set of all combinators of type $X$.\\[0.1cm] 
\indent In connection with Definition \ref{d2},  if ${\EuScript C}$ is a cyclic operad (and, hence,   a model of the equational theory from above), and writing $\underline{\EuScript C}$ for the underlying functor of ${\EuScript C}$, we shall denote with  $[\rule{.4em}{.4pt}]_{{\EuScript C}}: $ \texttt{cTerm}$_{\underline{\EuScript C}}\rightarrow {{\EuScript C}}$  the induced  interpretation of the combinator syntax.

\subsection{Unbiased definition of cyclic operads}
Cyclic operads were originally introduced in  unbiased manner in \cite[Definition 2.1]{Getzler:1994pn}, as   {\em algebras over a monad of unrooted trees}. In the operadic literature,   incorporated in the structure of cyclic operads and similar definitions, one can find two formalisms of unrooted trees: in \cite[Definition 2.1]{Getzler:1994pn}, the usual  formalism of trees  with   ``indivisible" edges is used, while in \cite{ezra},  \cite{kock}, \cite{Kaufmann}, trees with half-edges (or flags), due to \cite{gra}, are used in the context of modular operads and Feynman categories.  The operations  decorating the  nodes of an unrooted tree are ``composed in one shot'' through the structure morphism of the algebra. 
In this part, we syntactically reformulate  \cite[Definition 2.1]{Getzler:1994pn}. The adaptations we make also include translating it  to the non-skeletal setting, and reconstructing it within a formalism of unrooted trees that incorporates edges as pairs of half-edges, due to \cite{gra}. As it will be clear in Section \ref{mu-syntax}, the formal language of unrooted trees that we present here is crafted in a way which reflects closely the formal language of the $\mu$-syntax.   
\subsubsection{Graphs and unrooted trees}\label{trees}
 Let ${\underline{\EuScript C}}: {\bf Bij}^{op}\rightarrow {\bf Set}$ be a functor and let $P_{\underline{\EuScript C}}$ be as in \eqref{pc}. The syntax of unrooted trees generated by $P_{{\underline{\EuScript C}}}$ is obtained as follows.  An {\em ordinary corolla} is a term $$a(x,y,z,\dots) ,$$ where $a\in {{\underline{\EuScript C}}}(X)$ and $X=\{x,y,z,\dots\}$.  We refer to $a$ as the {\em head symbol} of   $a(x_1,\dots,x_n)$.  We call the elements of $X$ the {\em free variables}  of $a(x,y,z,\dots)$, and we write  $FV(a)=X$  to denote this set. Whenever the set of free variables is irrelevant, we shall refer to an ordinary corolla only by its head symbol. In addition to ordinary corollas, we define {\em special corollas} to be terms of the shape  $$(x,y),$$ i.e. terms which do not have a parameter as a head symbol and which consist only of two distinct variables $x,y\in V$. For a special corolla $(x,y)$, we define $FV((x,y))=\{x,y\}$. 
\begin{rem}
In both ordinary  and special corollas,  the order of  appearance of free variables  in the terms is irrelevant. In other words, we consider equal the terms, say, $a(x,y,z)$ and $a(z,x,y)$, as well as $(x,y)$ and $(y,x)$.
\end{rem}
\indent  A {\em graph} ${\EuScript V}$ is a non-empty, finite set  of corollas with mutually disjoint free variables, together with an involution $\sigma$ on the set \vspace{-0.2cm} $$V({\EuScript V})=\bigcup_{i=1}^{k}FV(a_i)\cup\bigcup_{j=1}^{p}FV((u_j,v_j))  \vspace{-0.2cm} $$  of all variables occuring in ${\EuScript V}$. We write \vspace{-0.1cm}  $${\EuScript V}=\{a_1(x_1,\dots ,x_n), \dots, a_k(y_1,\dots y_m),\dots, (u_1,v_1),\dots,(u_p,v_p);\sigma\} .\vspace{-0.1cm} $$ 

\noindent We  denote  with ${\it Cor}({\EuScript V})$  the set of all corollas of ${\EuScript V}$, and we shall refer to an ordinary corolla by its parameter and denote special corollas with $s_1,s_2$, etc. The set of {\em edges} ${\it Edge}({\EuScript V})$ of   ${\EuScript V}$ consists of pairs $(x,y)$ of variables such that $\sigma(x)=y$  (and, therefore, also $\sigma(y)=x$).  Finally,  we  refer to the fixpoints of $\sigma$ as the {\em free variables of ${\EuScript V}$}, the set of which we shall denote with $FV({\EuScript V})$. 
\begin{rem} 
The set of  variables  of a graph in our formalism corresponds to the set of {\em flags} in the formalism of {\em\cite{gra}} and {\em \cite{ezra}}, i.e. to the set of {\em half-edges}   in the formalism of {\em\cite{Kaufmann}}. All these formalisms of graphs are inherent to operad theory. In graph theory in general, one  does not usually encounter graphs with half edges: graphs  typically feature ``indivisible'' edges.
\end{rem}
  Here is an example.
\begin{example}\label{ex1} The graph  $\{a(x_1,x_2,x_3,x_4,x_5), b(y_1,y_2,y_3,y_4\};\sigma\},$ where $\sigma=(x_4\,\, y_3)(x_5\,\, y_4)$, should be depicted as
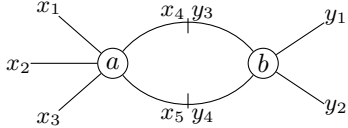
\vspace{-0.1cm}
\begin{center}
\begin{tikzpicture}
 \node (f) [circle,fill=none,draw=black,minimum size=4mm,inner sep=0.1mm]  at (-1,0) {\small $a$};
\node (g) [circle,fill=none,draw=black,minimum size=4mm,inner sep=0.1mm]  at (1,0) {\small $b$};
\node (a) [label={[xshift=-0.05cm, yshift=-0.29cm,]{\footnotesize $x_1$}},circle,fill=none,draw=none,minimum size=2mm,inner sep=0mm]  at (-1.8,0.7) {};
\node (b) [label={[xshift=-0.05cm,yshift=-0.37cm,]{\footnotesize $x_2$}},circle,fill=none,draw=none,minimum size=2mm,inner sep=0mm]  at (-2.2,0) {};
\node (c) [label={[xshift=-0.05cm, yshift=-0.37cm,]{\footnotesize $x_3$}},circle,fill=none,draw=none,minimum size=2mm,inner sep=0mm]  at (-1.8,-0.7) {};
\node (d) [label={[xshift=0.07cm, yshift=-0.33cm,]{\footnotesize $y_1$}},circle,fill=none,draw=none,minimum size=2mm,inner sep=0mm]  at (1.9,0.6) {};
\node (e) [label={[xshift=0.07cm, yshift=-0.33cm,]{\footnotesize $y_2$}},circle,fill=none,draw=none,minimum size=2mm,inner sep=0mm]  at (1.9,-0.6) {};
\node (i) [label={[xshift=-0.2cm, yshift=-0.15cm,]{\footnotesize $x_4$}},label={[xshift=0.2cm, yshift=-0.17cm,]{\footnotesize $y_3$}},circle,fill=none,draw=none,minimum size=0mm,inner sep=0mm]  at (0,0.6) {};
\node (j) [label={[xshift=-0.2cm, yshift=-0.42cm,]{\footnotesize $x_5$}},label={[xshift=0.2cm, yshift=-0.45
cm,]{\footnotesize $y_4$}},circle,fill=none,draw=none,minimum size=2mm,inner sep=0mm]  at (0,-0.6) {};
\path[out=-50,in=230] (f) edge (g);
\path[out=50,in=130] (f) edge (g);
\draw (f)--(a);
\draw (f)--(b);
\draw (f)--(c);
\draw (g)--(d);
\draw (g)--(e);
\draw (0,0.6)--(0,0.4);
\draw (0,-0.6)--(0,-0.4);
\end{tikzpicture}
\end{center}
\vspace{-0.1cm}
\noindent This graph has two corollas, $a(x_1,x_2,x_3,x_4,x_5)$ and $b(y_1,y_2,y_3,y_4)$, two edges, $(x_4,y_3)$ and $(x_5,y_4)$, and five free variables, $x_1,x_2,x_3,y_1,y_2$.\hfill$\square$\end{example}
Graphs do not need to be connected.  {\em Connected graphs} are distinguished by the following recursive definition:
\begin{itemize} 
\item[$\diamond$] for any finite set $X$ and any $a\in \underline{\EuScript C}(X)$, $\{a(x_1,\dots,x_n);{\it id}_{X}\}$  is  connected,
\item[$\diamond$] for any two-element set $\{x,y\}$, $\{(x,y);{\it id}_{\{x,y\}}\}$ is  connected,
\item[$\diamond$] if graphs ${\EuScript V}_1$ and ${\EuScript V}_2$, with involutions $\sigma_1$ and $\sigma_2$, respectively, are connected, and if   $V({\EuScript V}_1)\cap V({\EuScript V}_2)=\emptyset$,  then, for any $x\in FV({\EuScript V}_1)$ and $y\in FV({\EuScript V}_2)$, the graph  ${\EuScript V}$, determined by ${\it Cor}({\EuScript V})={\it Cor}({\EuScript V}_1)\cup {\it Cor}({\EuScript V}_2)$ and the involution $\sigma$ on $V({\EuScript V})$, defined by $$    \sigma(v) = \left\{\begin{array}{ll}
        \sigma_1(v), & \text{if } v\in V({\EuScript V}_1)\backslash \{x\}\\
         \sigma_2(v), & \text{if } v\in V({\EuScript V}_2)\backslash\{y\}\\
       y, & \text{if } v=x
        \end{array}  \right.$$
is conneted. 
\end{itemize}

\indent The {\em set of  subgraphs of a graph} ${\EuScript V}$ ({\em with involution} ${\sigma}$) is obtained by the following recursive definiton:
\begin{itemize}
\item[$\diamond$] if $a(x_1,\dots,x_n)\in {\it Cor}({\EuScript V})$, then $\{a(x_1,\dots,x_n);{\it id}_{X}\}$, where $X=\{x_1,\dots,x_n\}$, is a subgraph of ${\EuScript V}$,
\item[$\diamond$] if $(x,y)\in  {\it Cor}({\EuScript V})$, then $\{(x,y);{\it id}_{\{x,y\}}\}$ is a subgraph of ${\EuScript V}$,
\item[$\diamond$] if graphs ${\EuScript V}_1$ and ${\EuScript V}_2$, with involutions $\sigma_1$ and $\sigma_2$, respectively, are subgraphs of ${\EuScript V}$, and if there exist  $x\in FV({\EuScript V}_1)$ and $y\in FV({\EuScript V}_2)$, such that $\sigma(x)=y$, then the graph ${\cal W}$, determined by ${\it Cor}({\cal W})={\it Cor}({\EuScript V}_1)\cup {\it Cor}({\EuScript V}_2)$ and the involution  $\tau$ on $V({\cal W})$, defined by  
$$
    \tau(v) = \left\{\begin{array}{ll}
        \sigma_1(v), & \text{if } v\in V({\EuScript V}_1)\backslash \{x\}\\
         \sigma_2(v), & \text{if } v\in V({\EuScript V}_2)\backslash\{y\}\\
       y, & \text{if } v=x
        \end{array}  \right.
 \quad\mbox{or}\quad     \tau(v) = \left\{\begin{array}{ll}
        \sigma_1(v), & \text{if } v\in V({\EuScript V}_1)\\
         \sigma_2(v), & \text{if } v\in V({\EuScript V}_2)
        \end{array} \right. $$
 
is a subgraph of ${\EuScript V}$, and
\item[$\diamond$] if graphs ${\EuScript V}_1$ and ${\EuScript V}_2$, with involutions $\sigma_1$ and $\sigma_2$, respectively, are subgraphs of ${\EuScript V}$, and if there does not exist  $x\in FV({\EuScript V}_1)$ and $y\in FV({\EuScript V}_2)$, such that $\sigma(x)=y$, then the graph ${\cal W}$, determined by ${\it Cor}({\cal W})={\it Cor}({\EuScript V}_1)\cup {\it Cor}({\EuScript V}_2)$ and the involution  $\tau$ on $V({\cal W})$, defined by   $$\tau(v) = \left\{\begin{array}{ll}
        \sigma_1(v), & \text{if } v\in V({\EuScript V}_1)\\
         \sigma_2(v), & \text{if } v\in V({\EuScript V}_2)
        \end{array} \right. $$
is a subgraph of ${\EuScript V}$.
\end{itemize}

Observe that, just as graphs do not have to be connected, so do not subgraphs of an arbitrary graph.\\[0.1cm]
\indent   Starting from this notion of graph,   an {\em extended unrooted tree}  is defined as a {\em  connected graph without loops, 
multiple edges and cycles}.  As these  requirements are standard in the terminology of graphs, we omit their formal definition and illustrate them with an example instead. 
\begin{example}\label{ex2}
The   graph from {\textsc{Example}} \ref{ex1}  is not an extended unrooted tree, since it has two edges between corollas $a$ and $b$.\\[0.1cm]
\indent The graph  $\{a(x_1,x_2,x_3), b(y_1,y_2,y_3);\sigma\}$, where $\sigma=(x_3\,\, y_3)(y_1\,\, y_2)$, is  not an extended unrooted tree either, since the edge $(y_1, y_2)$ connects the corolla $b$ with itself, i.e. it is a loop:
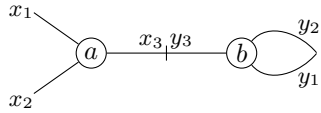
\vspace{-0.1cm}
\begin{center}
\begin{tikzpicture}
 \node (f) [circle,fill=none,draw=black,minimum size=4mm,inner sep=0.1mm]  at (-1,0) {\small $a$};
\node (g) [circle,fill=none,draw=black,minimum size=4mm,inner sep=0.1mm]  at (1,0) {\small $b$};
\node (a) [label={[xshift=-0.08cm, yshift=-0.38cm,]{\footnotesize $x_1$}},circle,fill=none,draw=none,minimum size=2mm,inner sep=0mm]  at (-1.85,0.6) {};
\node (c) [label={[xshift=-0.08cm, yshift=-0.38cm,]{\footnotesize $x_2$}},circle,fill=none,draw=none,minimum size=2mm,inner sep=0mm]  at (-1.85,-0.6) {};
\node (d) [label={[xshift=0cm, yshift=-0.8cm,]{\footnotesize $y_2$}},circle,fill=none,draw=none,minimum size=2mm,inner sep=0mm]  at (1.9,0.8) {};
\node (e) [label={[xshift=0cm, yshift=0.1cm,]{\footnotesize $y_1$}},circle,fill=none,draw=none,minimum size=2mm,inner sep=0mm]  at (1.9,-0.8) {};
\node (k) [circle,fill=none,draw=none,minimum size=0mm,inner sep=0mm]  at (2,0) {};
\node (i) [label={[xshift=-0.2cm, yshift=-0.17cm,]{\footnotesize $x_3$}},label={[xshift=0.2cm, yshift=-0.18cm,]{\footnotesize $y_3$}},circle,fill=none,draw=none,minimum size=0mm,inner sep=0mm]  at (0,0.1) {};
\path[out=-50,in=230] (g) edge (k);
\path[out=50,in=130] (g) edge (k);
\draw (f)--(g);
\draw (f)--(a);
\draw (f)--(c);
\draw (0,0.1)--(0,-0.1);
\end{tikzpicture}
\end{center}

\indent   The graph $\{a(x_1,x_2,x_3,x_4,x_5),b(y_1,y_2,y_3,y_4),c(z_1,z_2,z_3);\sigma\}$,
with $\sigma=(x_4\,\, y_2)(y_1\,\, z_2)\linebreak (z_3\,\, x_5)$, is another example of a graph which is  not an extended unrooted tree, this time because of the presence of a cycle  that connects  its three corollas:
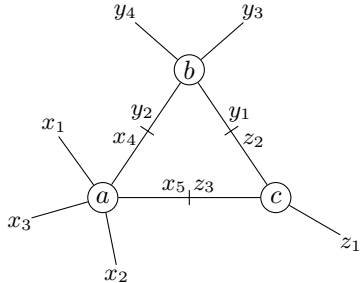
\vspace{-0.1cm}
\begin{center}
\begin{tikzpicture}
 \node (f) [circle,fill=none,draw=black,minimum size=4mm,inner sep=0.1mm]  at (-1.15,0) {\small $a$};
\node (g) [circle,fill=none,draw=black,minimum size=4mm,inner sep=0.1mm]  at (0,1.7) {\small $b$};
\node (h) [circle,fill=none,draw=black,minimum size=4mm,inner sep=0.1mm]  at (1.15,0) {\small $c$};
\node (a) [label={[xshift=0cm, yshift=-0.27cm]{\footnotesize $x_1$}},circle,fill=none,draw=none,minimum size=2mm,inner sep=0mm]  at (-1.8,0.9) {};
\node (b) [label={[xshift=-0.05cm, yshift=-0.37cm,]{\footnotesize $x_3$}},circle,fill=none,draw=none,minimum size=2mm,inner sep=0mm]  at (-2.2,-0.3) {};
\node (c) [label={[xshift=-0.015cm, yshift=-0.37cm,]{\footnotesize $x_2$}},circle,fill=none,draw=none,minimum size=2mm,inner sep=0mm]  at (-0.95,-1) {};
\node (d) [label={[xshift=-0.05cm, yshift=-0.28cm,]{\footnotesize $y_4$}},circle,fill=none,draw=none,minimum size=2mm,inner sep=0mm]  at (-0.8,2.4) {};
\node (e) [label={[xshift=0.05cm, yshift=-0.28cm,]{\footnotesize $y_3$}},circle,fill=none,draw=none,minimum size=2mm,inner sep=0mm]  at (0.8,2.4) {};
\node (i) [label={[xshift=-0.2cm, yshift=-0.17cm,]{\footnotesize $x_5$}},label={[xshift=0.2cm, yshift=-0.17cm,]{\footnotesize $z_3$}},circle,fill=none,draw=none,minimum size=0mm,inner sep=0mm]  at (0,0.1) {};
\node (i1) [label={[xshift=-0.23cm, yshift=-0.1cm,]{\footnotesize $x_4$}},label={[xshift=0cm, yshift=0.24cm,]{\footnotesize $y_2$}},circle,fill=none,draw=none,minimum size=0mm,inner sep=0mm]  at (-0.625,0.65) {};
\node (i2) [label={[xshift=0.05cm, yshift=0.22cm,]{\footnotesize $y_1$}},label={[xshift=0.23cm, yshift=-0.12cm,]{\footnotesize $z_2$}},circle,fill=none,draw=none,minimum size=0mm,inner sep=0mm]  at (0.625,0.65) {};
\node (k) [label={[xshift=0.05cm, yshift=-0.4cm,]{\footnotesize $z_1$}},circle,fill=none,draw=none,minimum size=2mm,inner sep=0mm]  at (2.1,-0.55) {};
\draw (h)--(k);
\draw (f)--(g);
\draw (f)--(h);
\draw (g)--(h);
\draw (f)--(a);
\draw (f)--(b);
\draw (f)--(c);
\draw (g)--(d);
\draw (g)--(e);
\draw (0,0.1)--(0,-0.1);
\draw (0.65,0.95)--(0.47,0.83);
\draw (-0.65,0.95)--(-0.47,0.83);
\end{tikzpicture}
\end{center}

\indent Finally, we get an example of a graph which is  an extended unrooted tree   by  changing the involution $\sigma$ of the previous graph to, say, $\sigma'=(x_4\,\, y_2)(y_1\,\, z_2)$, producing in this way the extended unrooted tree with graphical representation 
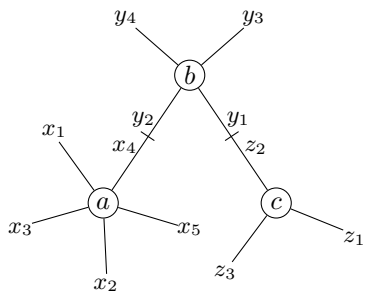
\vspace{-0.1cm}
\begin{center}
\begin{tikzpicture}
 \node (f) [circle,fill=none,draw=black,minimum size=4mm,inner sep=0.1mm]  at (-1.15,0) {\small $a$};
\node (g) [circle,fill=none,draw=black,minimum size=4mm,inner sep=0.1mm]  at (0,1.7) {\small $b$};
\node (h) [circle,fill=none,draw=black,minimum size=4mm,inner sep=0.1mm]  at (1.15,0) {\small $c$};
\node (a) [label={[xshift=0cm, yshift=-0.27cm]{\footnotesize $x_1$}},circle,fill=none,draw=none,minimum size=2mm,inner sep=0mm]  at (-1.8,0.9) {};
\node (b) [label={[xshift=-0.05cm, yshift=-0.37cm,]{\footnotesize $x_3$}},circle,fill=none,draw=none,minimum size=2mm,inner sep=0mm]  at (-2.2,-0.3) {};
\node (c) [label={[xshift=-0.015cm, yshift=-0.37cm,]{\footnotesize $x_2$}},circle,fill=none,draw=none,minimum size=2mm,inner sep=0mm]  at (-1.1,-1.05) {};
\node (d) [label={[xshift=-0.05cm, yshift=-0.28cm,]{\footnotesize $y_4$}},circle,fill=none,draw=none,minimum size=2mm,inner sep=0mm]  at (-0.8,2.4) {};
\node (e) [label={[xshift=0.05cm, yshift=-0.28cm,]{\footnotesize $y_3$}},circle,fill=none,draw=none,minimum size=2mm,inner sep=0mm]  at (0.8,2.4) {};
\node (i) [label={[xshift=0.15cm, yshift=-0.27cm,]{\footnotesize $x_5$}},circle,fill=none,draw=none,minimum size=0mm,inner sep=0mm]  at (-0.15,-0.3) {};
\node (i1) [label={[xshift=-0.24cm, yshift=-0.13cm,]{\footnotesize $x_4$}},label={[xshift=0.01cm, yshift=0.22cm,]{\footnotesize $y_2$}},circle,fill=none,draw=none,minimum size=0mm,inner sep=0mm]  at (-0.625,0.65) {};
\node (i2) [label={[xshift=0.02cm, yshift=0.22cm,]{\footnotesize $y_1$}},label={[xshift=0.25cm, yshift=-0.14cm,]{\footnotesize $z_2$}},circle,fill=none,draw=none,minimum size=0mm,inner sep=0mm]  at (0.625,0.65) {};
\node (k) [label={[xshift=0.05cm, yshift=-0.4cm,]{\footnotesize $z_1$}},circle,fill=none,draw=none,minimum size=2mm,inner sep=0mm]  at (2.14,-0.4) {};
\node (l) [label={[xshift=-0.03cm, yshift=-0.38cm,]{\footnotesize $z_3$}},circle,fill=none,draw=none,minimum size=2mm,inner sep=0mm]  at (0.5,-0.87) {};
\draw (h)--(k);
\draw (h)--(l);
\draw (f)--(g);
\draw (f)--(i);
\draw (g)--(h);
\draw (f)--(a);
\draw (f)--(b);
\draw (f)--(c);
\draw (g)--(d);
\draw (g)--(e);
\draw (0.65,0.95)--(0.47,0.83);
\draw (-0.65,0.95)--(-0.47,0.83);
\end{tikzpicture}
\end{center}
\vspace{-0.4cm}
\hfill$\square$
\end{example}

In moving from graphs to trees, we will  additionally differentiate the classes of extended unrooted trees with respect to the shape of corollas they contain. Let ${\EuScript T}$ be a connected  graph  with no loops, multiple edges and cycles. \\[-0.5cm]
\begin{itemize}
\item If ${\it Cor}({\EuScript T})$ consists only of ordinary corollas, then ${\EuScript T}$ is an {\em ordinary unrooted tree}.\\[-0.5cm]
\item If ${\it Cor}({\EuScript T})$ is a singleton with a special corolla, then ${\EuScript T}$ is an {\em exceptional unrooted tree}.\\[-0.5cm]
\item An {\em unrooted tree} is either an ordinary unrooted tree or an exceptional unrooted tree.\\[-0.5cm]
\end{itemize}
\begin{example}
The last graph in {\textsc{Example}} \ref{ex2} is an ordinary unrooted tree. \\[0.1cm]
\indent The graph $\{(x,y); {\it id}_{\{x,y\}}\}$ is an exceptional unrooted tree. We depict it as 
\begin{center}
\begin{tikzpicture}
\node (i) [label={[xshift=0cm, yshift=-0.2cm,]{\footnotesize $x$}},circle,fill=none,draw=none,minimum size=0mm,inner sep=0mm]  at (-0.15,0.3) {};
\node (j) [label={[xshift=0cm, yshift=-0.32cm,]{\footnotesize $y$}},circle,fill=none,draw=none,minimum size=0mm,inner sep=0mm]  at (-0.15,-0.3) {};
 \draw (0,0.5)--(0,-0.5);
\draw (0.1,0)--(-0.1,0);
\end{tikzpicture}
\end{center}

\indent The graph $\{a(x_1,x_2,x_3),b(y_1,y_2),(z_1,z_2);\sigma\}$, where $\sigma=(x_3\,y_2)(y_1\, z_1)$,  depicted as
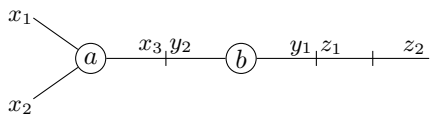
\begin{center}
\begin{tikzpicture}
 \node (f) [circle,fill=none,draw=black,minimum size=4mm,inner sep=0.1mm]  at (-1,0) {\small $a$};
\node (g) [circle,fill=none,draw=black,minimum size=4mm,inner sep=0.1mm]  at (1,0) {\small $b$};
\node (a) [label={[xshift=-0.08cm, yshift=-0.38cm,]{\footnotesize $x_1$}},circle,fill=none,draw=none,minimum size=2mm,inner sep=0mm]  at (-1.85,0.6) {};
\node (c) [label={[xshift=-0.08cm, yshift=-0.38cm,]{\footnotesize $x_2$}},circle,fill=none,draw=none,minimum size=2mm,inner sep=0mm]  at (-1.85,-0.6) {};
\node (d) [label={[xshift=-0.1cm, yshift=-1cm,]{\footnotesize $y_1$}},circle,fill=none,draw=none,minimum size=2mm,inner sep=0mm]  at (1.9,0.8) {};
\node (k) [circle,fill=none,draw=none,minimum size=0mm,inner sep=0mm]  at (2,0) {};
\node (i) [label={[xshift=-0.2cm, yshift=-0.17cm,]{\footnotesize $x_3$}},label={[xshift=0.2cm, yshift=-0.18cm,]{\footnotesize $y_2$}},circle,fill=none,draw=none,minimum size=0mm,inner sep=0mm]  at (0,0.1) {};
\node (z_1) [label={[xshift=-0.3cm, yshift=-0.18cm,]{\footnotesize $z_1$}},circle,fill=none,draw=none,minimum size=0mm,inner sep=0mm]  at (2.5,0.1) {};
\node (z_2) [label={[xshift=0.3cm, yshift=-0.18cm,]{\footnotesize $z_2$}},circle,fill=none,draw=none,minimum size=0mm,inner sep=0mm]  at (3,0.1) {};
\draw (2,0.1)--(2,-0.1);
\draw (2.75,0.1)--(2.75,-0.1);
\draw (f)--(g)--(3.5,0);
\draw (f)--(a);
\draw (f)--(c);
\draw (0,0.1)--(0,-0.1);
\end{tikzpicture}
\end{center}
is an extended unrooted tree. It is neither ordinary, nor exceptional unrooted tree.\hfill $\square$
\end{example}
\begin{rem}
Exceptional unrooted trees in our formalism correspond  to   {\em trivial graphs} in the formalism of {\em\cite{kock}}. In the operadic context, these are  {\em trivial trees} of {\em\cite{mss}}. Moreover, the graphs of {\em\cite{kock}}  are our extended unrooted trees, while the graphs of {\em\cite{gra}} are our ordinary unrooted trees.
\end{rem}
\begin{rem}
Every unrooted tree is an extended unrooted tree. On the other hand, every unrooted tree containing at least one ordinary and one special corolla (or at least  two special corollas) is an extended unrooted tree, but not an unrooted tree (in the narrow sense).
\end{rem}

\indent A {\em subtree} of an (extended) unrooted tree ${\EuScript T}$ is a connected, non-empty subgraph of ${\EuScript T}$.  We say that a subtree ${\EuScript S}$ of ${\EuScript T}$ is {\em  proper}   if ${\it Cor}({\EuScript  S})\neq {\it Cor}({\EuScript T})$.\\[0.1cm]
\indent A {\em decomposition of an (extended) unrooted tree}  ${\EuScript T}$ ({\em with involution} $\sigma$) is a set of subtrees of ${\EuScript T}$ defined recursively as follows: \label{marker}
\begin{itemize}
\item[$\diamond$]  $\{{\EuScript T}\}$ is a decomposition of ${\EuScript T}$,
\item[$\diamond$] if  ${\EuScript T}_1$ and ${\EuScript T}_2$ are subtrees of  ${\EuScript T}$ with involutions  $\sigma_1$ and $\sigma_2$, respectively, such that  ${\it Cor}({\EuScript T}_1)\cap {\it Cor}({\EuScript T}_2)=\emptyset$, ${\it Cor}({\EuScript T})={\it Cor}({\EuScript T}_1)\cup {\it Cor}({\EuScript T}_2)$ and there exist $x\in {\it FV}({\EuScript T}_1)$ and $y\in {\it FV}({\EuScript T}_2)$ such that   
\[
    \sigma(v) = \left\{\begin{array}{ll}
        \sigma_1(v), & \text{if } v\in V({\EuScript T}_1)\backslash\{x\}\\
        \sigma_2(v), & \text{if } v\in V({\EuScript T}_2)\backslash\{y\}\\
        y, & \text{if } v=x,
        \end{array}  \right.
  \]
and if $\{{\EuScript T}_{11},\dots ,{\EuScript T}_{1n}\}$ is a decomposition of ${\EuScript T}_1$ and  $\{{\EuScript T}_{21},\dots ,{\EuScript T}_{2m}\}$ is a decomposition of ${\EuScript T}_2$, then $\{{\EuScript T}_{11},\dots ,{\EuScript T}_{1n},{\EuScript T}_{21},\dots ,{\EuScript T}_{2m}\}$ is a decomposition of ${\EuScript T}$.
\end{itemize}

\indent We now define $\alpha$-equivalence on extended unrooted trees. Suppose first that\vspace{-0.1cm} $${\EuScript T}=\{a(x_1,{\dots},x_n),{\dots} ;\sigma\} \vspace{-0.1cm} $$ is an ordinary unrooted tree, with  $a\in{\underline{\EuScript C}}(X)$,  $x_i\in FV(a)\backslash FV({\EuScript T})$ and $\sigma(x_i)=y_j$. Let $\tau:X'\rightarrow X$ be a bijection that renames $x_i$ to $z$, where $z$ is  fresh with respect to  $V({\EuScript T})\backslash\{x_i\}$. The $\alpha$-equivalence (for ordinary unrooted trees) is the smallest equivalence relation generated by equalities\vspace{-0.1cm} $$(a(x_1,\dots,x_{i-1},x_i,x_{i+1},\dots,x_n),{\dots} ;\sigma)=_{\alpha}(a^{\tau}(x_1,{\dots},x_{i-1},z,x_{i+1},\dots,x_n),{\dots};\sigma'),\vspace{-0.1cm} $$ where  $\sigma'=\sigma$ on $V({\EuScript T})\backslash\{x_i,y_j\}$ and $\sigma'(z)=y_j$.  This definition generalises in a natural way to extended unrooted trees: to the set of generators from above we add the clauses \vspace{-0.1cm} $$\{(x,y),\dots ; \sigma\}=_{\alpha}\{(x,z),\dots ;\sigma'\} ,\vspace{-0.1cm}$$ where, for some variable $x_i$, $\sigma(y)=x_i$ (i.e. $y$ is not a free variable of the tree on the left), $z$ is fresh in the same sense as earlier, and $\sigma'$ is the obvious modification of $\sigma$.  In simple terms, we consider $\alpha$-equivalent any two trees such that we can obtain one from another only by renaming variables which are not fixed points of the corresponding involutions. \\[0.1cm]
\indent We shall denote with $[{\EuScript T}]_{\alpha}$  the $\alpha$-equivalence class determined by an (extended) unrooted tree ${\EuScript T}$. Finally,  we shall denote with  \texttt{T}$_{{\underline{\EuScript C}}}(X)$ (resp.  \texttt{eT}$_{{\underline{\EuScript C}}}(X)$)   the set of all $\alpha$-equivalence classes of unrooted trees (resp. extended unrooted trees) whose parameters belong to $P_{{\underline{\EuScript C}}}$ and whose free variables are given by the set $X$. If $X$ is a two-element set, this definition includes the possibility that an unrooted tree has $0$ parameters, in which case   the corresponding equivalence class is determined by the appropriate exceptional unrooted tree. We shall write  \texttt{T}$_{{\underline{\EuScript C}}}$ (resp. \texttt{eT}$_{{\underline{\EuScript C}}}$) for the collection of all unrooted trees (resp. extended unrooted trees) generated by $P_{{\underline{\EuScript C}}}$.
\subsubsection{The monad of unrooted trees}\label{monad1}
The monad of unrooted trees is the monad $({\EuScript M},\mu,\eta)$ on the functor category ${\bf Set}^{{\bf Bij}^{op}}$, defined as follows. The endofunctor ${\EuScript M}$ is defined by $${\EuScript M}(\underline{\EuScript C})(X)={\tt{T}}_{\underline{\EuScript C}}(X).$$
\indent The component ${\eta_{\underline{\EuScript C}}}_X:\underline{\EuScript C}(X)\rightarrow {\EuScript M}({\underline{\EuScript C}})(X)$ of the monad unit associates to $a\in\underline{\EuScript C}(X)$ the isomorphism class of the unrooted tree $\{a(x_1,\dots,x_n),{\it id}_X\}$, where $X=\{x_1,\dots,x_n\}$.\\[0.1cm]
\indent The action of the monad multiplication, typically (and incompletely) described  as ``flattening'' 
in the literature (which is acceptable only if one forgets about units), deserves more attention.\\
\indent In order to obtain its complete description, we first build a rewriting system on ${\tt{eT}}_{{\underline{\EuScript C}}}$. The rewriting relation $\rightarrow$ on classes of ${\tt{eT}}_{{\underline{\EuScript C}}}$ is canonically induced by the  reflexive and transitive closure of the union of the following   reductions, defined on their representatives:  \vspace{-0.1cm} $$(a(x_1,\dots,x_{i-1},x_i,x_{i+1},\dots,x_n),(y,z),\dots;\sigma)\rightarrow (a^{\tau}(x_1,\dots,x_{i-1},z,x_{i+1},\dots,x_n),\dots;\sigma')\, ,\vspace{-0.1cm}$$ where $\sigma(x_i)=y$, $\tau$ renames $x_i$ to $z$, and $\sigma'$ is the obvious restriction of $\sigma$, and $$((x,y),(u,v),\dots ;\sigma)\rightarrow ((x,v),\dots ;\sigma')\, ,\vspace{-0.1cm}$$ where $\sigma(y)=u$, and $\sigma'$ is  again the obvious restriction of $\sigma$.
\begin{lem}\label{conf}
The rewriting system $(${\em\texttt{eT}}$_{{\underline{\EuScript C}}},\rightarrow)$ is confluent and terminating.
\end{lem} 
\begin{proof}
The termination of the system is obvious: in an arbitrary reduction sequence, each subsequent tree has one special corolla less, and the sequence finishes either when all of them are exhausted (in the case when the initial tree has at least one ordinary corolla), or when there is only one special corolla left (in the case when the initial tree consists only of special corollas). Due to the connectedness of  unrooted trees, all  special corollas (except one in the latter case) will indeed be exhausted. Clearly, the   normal forms are precisely the unrooted trees of \texttt{T}$_{{\underline{\EuScript C}}}$. \\
\indent  If ${\EuScript T}_1$ and ${\EuScript T}_2$ are reduced from ${\EuScript T}\in $ \texttt{eT}$_{{\underline{\EuScript C}}}$  in one step, and if  $s_1$ and $s_2$ are the special corollas   involved in the respective reductions, the local confluence is proved by case analysis, with respect to whether $s_1$ and $s_2$ are equal or not. By Fact \ref{fact3}, this establishes confluence. 
\end{proof}
By Lemma \ref{conf}, an arbitrary normal form ${\it nf}({\EuScript T})$ of an extended unrooted tree ${\EuScript T}$, with respect to $\rightarrow$, determines a unique $\alpha$-equivalence class  $[{\it nf}({\EuScript T})]_{\alpha}$ in \texttt{T}$_{{\underline{\EuScript C}}}$. It is easily seen that, for every finite set $X$, this assignment gives rise to the function ${\it nf}_X:{\tt{eT}}_{\underline{\EuScript C}}(X)\rightarrow {\tt{T}}_{\underline{\EuScript C}}(X)$, determined by \begin{equation}\label{assi}{\it nf}_X:[{\EuScript T}]_{\alpha}\mapsto [{\it nf}({\EuScript T})]_{\alpha} .\end{equation} 
\indent We now formally define the  flattening (which is still not the monad multiplication) on ${\EuScript M}{\EuScript M}({\underline{\EuScript C}})(X)$. Observe that  the isomorphism classes of \begin{center} ${\EuScript M}{\EuScript M}({\underline{\EuScript C}})(X)={\EuScript M}($\texttt{T}$_{{\underline{\EuScript C}}})(X)=$\texttt{T}$_{\mbox{\texttt{T}}_{{\underline{\EuScript C}}}}(X)$\end{center} are determined by unrooted trees whose parameters are $\alpha$-equivalence classes of unrooted trees themselves (with parameters from $P_{{\underline{\EuScript C}}}$), and whose set of free variables is $X$. We call these  trees {\em two-level   trees}. Syntactically, a   two-level tree of \texttt{T}$_{\mbox{\texttt{T}}_{{\underline{\EuScript C}}}}$ can be either \\[-0.5cm]
\begin{itemize}
\item an exceptional unrooted tree $\{(x,y);id_{\{x,y\}}\},$  in which case we trivially have $0$ parameters  coming from $P_{{\underline{\EuScript C}}}$, or \\[-0.5cm]
\item an ordinary unrooted tree \vspace{-0.1cm} {\small $$\{[\{a(x_1,x_2,\dots),b(y_1,\dots),\dots ;\sigma_1\}]_{\alpha}(x_1,x_2,y_1,\dots),\dots ,[\{(z_1,z_2);id_{\{z_1,z_2\}}\}]_{\alpha}(z_1,z_2),\dots;\sigma\} , $$} 

\vspace{-0.65cm}
whose parameters can be $\alpha$-equivalence classes of both ordinary and exceptional unrooted trees of \texttt{T}$_{\underline{\EuScript C}}$.
\end{itemize} 
\begin{nota}\label{conv}
Let ${\EuScript T}$ be a two-level unrooted tree. Suppose that, for $1\leq i\leq n$, $[{\EuScript T}_i]_{\alpha}\in   {\tt{T}}_{\underline{\EuScript C}}(Y_i)$ are the parameters of ${\EuScript T}$ and let $C_i$ be their corresponding corollas. We then  have  $FV(C_i)=FV({\EuScript T}_i)=Y_i$. The fact that the set of free variables of each corolla is recorded by the data of the corresponding parameter allows us to shorten the notation  by writing  ${\EuScript T}_i$ without listing explicitly the elements of $FV({\EuScript T}_i)$. For example, for the tree from the latter case above, we shall write 
 $$\{[\{a(x_1,x_2,\dots),b(y_1,\dots),\dots ;\sigma_1\}]_{\alpha},\dots ,[\{(z_1,z_2);id_{\{z_1,z_2\}}\}]_{\alpha},\dots;\sigma\}.$$  We shall extend this abbreviation  to trees of $ {\tt{eT}}_{{\tt{eT}}_{\underline{\EuScript C}}}$, and when  the form of the parameters of a two-level tree  is irrelevant, we shall write  $\{[{\EuScript T}_1]_{\alpha},\dots,[{\EuScript T}_n]_{\alpha}, s_1,\dots,s_m;\sigma\}$, where $s_i$ are special corollas.
\end{nota}

\indent The flattening of  two-level unrooted trees is a familly of functions $${\it flat}_X:{\tt{T}}_{{\tt{T}}_{\underline{\EuScript C}}}(X)\rightarrow   {{\tt{eT}}_{\underline{\EuScript C}}}(X),$$ indexed by finite sets, defined by the following two clauses:\\[-0.5cm]
\begin{itemize}
\item ${\it flat}_{\{x,y\}}([\{(x,y);{\it id}_{\{x,y\}}\}]_{\alpha})=[\{(x,y);{\it id}_{\{x,y\}}\}]_{\alpha},$ and\\[-0.5cm]
\item if ${\EuScript T}=\{[\{a(x_1,x_2,\dots),b(y_1,\dots),\dots ;\sigma_1\}]_{\alpha},\dots ,[\{(z_1,z_2);id_{\{z_1,z_2\}}\}]_{\alpha},\dots;\sigma\}$, then   $${\it flat}_X([{\EuScript T}]_{\alpha})=[\{a(x_1,x_2,\dots),b(y_1,\dots),\dots ,(z_1,z_2),\dots;\underline{\sigma}\}]_{\alpha},$$
where, having denoted with ${\EuScript T}_i$, $1\leq i\leq n$, the corollas of ${\EuScript T}$, and with $\sigma_i$ the corresponding involutions,\vspace{-0.1cm} $${\underline{\sigma}}(x)=
\left\{ 
	\begin{array}{ll}
		\sigma(x) & \mbox{if }  x\in \bigcup_{i=1}^{n}FV({\EuScript T}_i)\\[0.1cm]
		\sigma_i(x) & \mbox{if } x\in V({\EuScript T}_i)\backslash FV({\EuScript T}_i)\, .
	\end{array}
\right.\vspace{-0.1cm}$$
\end{itemize}
Observe that ${\it flat}_X([{\EuScript T}]_{\alpha})$ is an $\alpha$-equivalence class of an extended unrooted tree whenever ${\EuScript T}$ contains a corolla that is an exceptional unrooted tree. These are the cases that make a gap between the flattening function and the action of the monad multiplication (which always results in an ordinary unrooted tree). In the same style as we presented the functions ${\it nf}_X$ by \eqref{assi}, 
in what follows, we shall often denote the class ${\it flat}_X([{\EuScript T}]_{\alpha})$ simply by $[{\it flat}({\EuScript T})]_{\alpha}$.  \\[0.1cm]
\indent The complete characterisation of the monad multiplication ${\mu_{\underline{\EuScript C}}}:{\tt{T}}_{{\tt{T}}_{\underline{\EuScript C}}}\rightarrow  {{\tt{T}}_{\underline{\EuScript C}}}$ is defined by $${\mu_{\underline{\EuScript C}}}_X={\it nf}_X\circ {\it flat}_X.$$
Therefore,   for $[{\EuScript T}]_{\alpha}\in {\tt{T}}_{{\tt{T}}_{\underline{\EuScript C}}}(X)$, we have $${\mu_{\underline{\EuScript C}}}_X:[{\EuScript T}]_{\alpha}\mapsto [{\it nf}({\it flat}({\EuScript T}))]_{\alpha}.$$
Hense,  in the presence of units, this action is indeed more than just ``flattening". \\[0.1cm]
\indent We now prepare the grounds for the proof that $({\EuScript M},\mu,\eta)$ is indeed a monad. \\[0.1cm]
\indent The domain of   flattening is extended in a natural way to ${{\EuScript M}}'{{\EuScript M}}'({\underline{\EuScript C}})$, where ${{\EuScript M}}'({\underline{\EuScript C}})(X)=$ \texttt{eT}$_{\underline{\EuScript C}}(X)$. The clause that needs to be added to encompass \texttt{eT}$_{\mbox{\texttt{eT}}_{{\underline{\EuScript C}}}}(X)$ concerns  two-level trees of the form
$$\{[\{a(x_1,x_2,\dots),b(y_1,\dots),(z_1,z_2)\dots ;\sigma_1\}]_{\alpha},\dots ,[\{(u_1,u_2);id_{\{u_1,u_2\}}\}]_{\alpha},\dots,(v_1,v_2),\dots;\sigma\} ,$$ i.e. extended unrooted trees whose set of corollas allows special corollas and  classes of extended  unrooted trees. Let us denote with ${\EuScript T}$ the above tree, and let ${\it Cor}_{s}({\EuScript T})$ be the set of its special corollas. The flattening of $[{\EuScript T}]_{\alpha}$ is defined simply as $${\it flat}_X([{\EuScript T}]_{\alpha})=[\{a(x_1,x_2,\dots),b(y_1,\dots),(z_1,z_2),\dots,(u_1,u_2),\dots,(v_1,v_2),\dots;\underline{\sigma}\}]_{\alpha} ,$$ with $\underline{\sigma}$ being  defined exactly like before for the variables coming from ${\it Cor}({\EuScript T})\backslash {\it Cor}_s({\EuScript T})$, while we set $\underline{\sigma}(x)=\sigma(x)$ for all variables $x\in\bigcup_{s\in {\it Cor}_{s}({\EuScript T})}FV(s)$.\\[0.1cm]
\indent For  $[{\EuScript T}]_{\alpha}$ and $[{\EuScript T}']_{\alpha}$ from \texttt{eT}$_{\mbox{\texttt{eT}}_{{\underline{\EuScript C}}}}$, the following two   lemmas give conditions that ensure that ${\it flat}([{\EuScript T}]_{\alpha})\rightarrow {\it flat}([{\EuScript T}']_{\alpha})$, in the instance $($\texttt{eT}$_{\mbox{\texttt{eT}}_{{\underline{\EuScript C}}}},\rightarrow)$  of the rewriting system defined earlier. 
\begin{lem}\label{fl1}
For $[{\EuScript T}_1]_{\alpha},[{\EuScript T}_2]_{\alpha}\in{\tt{eT}}_{{\tt{eT}}_{\underline{\EuScript C}}}(X)$, if $[{\EuScript T}_1]_{\alpha}\rightarrow [{\EuScript T}_2]_{\alpha}$   in $({{\tt{eT}}}_{{\tt{eT}}_{\underline{\EuScript C}}},\rightarrow)$,   then ${\it flat}_X([{\EuScript T}_1]_{\alpha})\rightarrow {\it flat}_X([{\EuScript T}_2]_{\alpha})$  in $({{\tt{eT}}_{\underline{\EuScript C}}},\rightarrow)$.
\end{lem}

\begin{lem}\label{fl2} For $[\{[{\EuScript T}_1]_{\alpha},\dots,[{\EuScript T}_n]_{\alpha} ,s_1,\dots,s_m;\sigma\}]_{\alpha}\in{\tt{eT}}_{{\tt{eT}}_{\underline{\EuScript C}}}(X)$ and $1\leq j\leq n$, if $[{\EuScript T_j}]_{\alpha}\rightarrow [{\EuScript T}_j']_{\alpha}$ in  $({{\tt{eT}}_{\underline{\EuScript C}}},\rightarrow)$, then   {\small $${\it flat}_X([\{[{\EuScript T}_1]_{\alpha},\dots,[{\EuScript T}_j]_{\alpha},\dots,[{\EuScript T}_n]_{\alpha}  ,s_1,\dots,s_m;\sigma\}]_{\alpha})\rightarrow {\it flat}_X([\{[{\EuScript T}_1]_{\alpha},\dots,[{\EuScript T}'_j]_{\alpha},\dots,[{\EuScript T}_n]_{\alpha} ,s_1,\dots,s_m;\sigma\}]_{\alpha})$$}

\vspace{-0.54cm}
\noindent in  $({{\tt{eT}}_{\underline{\EuScript C}}},\rightarrow)$
\end{lem}
Relying on   Lemma \ref{fl1} and Lemma \ref{fl2}, we  obtain the following two equivalent characterisations of the monad multiplication.
\begin{lem}\label{thmhm}
For $[{\EuScript T}]_{\alpha}=[\{[{\EuScript T}_1]_{\alpha},\dots,[{\EuScript T}_n]_{\alpha} ,s_1,\dots,s_m;\sigma\}]_{\alpha}\in{\tt{eT}}_{{\tt{eT}}_{\underline{\EuScript C}}}(X)$ the following  claims hold:
\begin{itemize}
\item[1.] ${\it nf}(flat({\EuScript T}))=_{\alpha}{\it nf}(flat({\it nf}({\EuScript T}))),$
\item[2.] ${\it nf}(flat({\EuScript T}))=_{\alpha}{\it nf}(flat(\{[{\it nf}({\EuScript T}_1)]_{\alpha},\dots,[{\it nf}({\EuScript T}_n)]_{\alpha} ,s_1,\dots,s_m ;\sigma\})).$
\end{itemize}
\end{lem}
\begin{proof}
By the termination of $($\texttt{eT}$_{\mbox{\texttt{eT}}_{{\underline{\EuScript C}}}},\rightarrow)$, we have  ${\EuScript T}\rightarrow {\it nf}(\EuScript T)$, and then, by Lemma \ref{fl1} and the termination of $($\texttt{eT}$_{{\underline{\EuScript C}}},\rightarrow)$, we know that,  in $($\texttt{eT}$_{{\underline{\EuScript C}}},\rightarrow)$, \vspace{-0.1cm}$$flat(\EuScript T)\rightarrow flat({\it nf}(\EuScript T))\rightarrow {\it nf}(flat({\it nf}(\EuScript T)))  \vspace{-0.1cm}.$$ On the other hand, by the termination of $($\texttt{eT}$_{{\underline{\EuScript C}}},\rightarrow)$, we  also have that $flat(\EuScript T)\rightarrow {\it nf}(flat(\EuScript T))$. Therefore, the first claim follows by the confluence of $($\texttt{eT}$_{{\underline{\EuScript C}}},\rightarrow)$.\\[0.1cm]
\indent As for the second claim, by the termination of $({\tt{eT}}_{\underline{\EuScript C}},\rightarrow)$, we have  ${{\EuScript T}_i}\rightarrow {\it nf}({\EuScript T}_i)$, for all $i\in I$. Hence, by Lemma \ref{fl2}, and then again by the termination of $({\tt{eT}}_{\underline{\EuScript C}},\rightarrow)$, we get that $$
\begin{array}{rcl}
{\it flat}({\EuScript T})&\rightarrow &{\it flat}(\{[{\it nf}({\EuScript T}_1)]_{\alpha},\dots,[{\it nf}({\EuScript T}_n)]_{\alpha},s_1,\dots,s_m;\sigma\})\\
&\rightarrow &{\it nf}({\it flat}(\{{\it nf}({\EuScript T}_1),\dots,{\it nf}({\EuScript T}_n),s_1,\dots,s_m;\sigma\}))
\end{array}$$ is a reduction sequence of  $($\texttt{eT}$_{{\underline{\EuScript C}}},\rightarrow)$. The conclusion follows as in the previous claim.
\end{proof}
On the other hand, by the very definition of flattening on extended unrooted trees, we have the following property.
\begin{lem}\label{assoc_flat} For ${\EuScript T}=\{{\EuScript T}_1,\dots,{\EuScript T}_n;\sigma\}\in$ {\em\texttt{T}}$_{\mbox{\em\texttt{eT}}_{\mbox{\em\texttt{eT}}_{{\underline{\EuScript C}}}}}$ the following equality holds: \vspace{-0.15cm}
$${\it flat}({\it flat}({\EuScript T}))={\it flat}(\{[{\it flat}({\EuScript T}_1)]_{\alpha},\dots,[{\it flat}({\EuScript T}_n)]_{\alpha};\sigma\}).$$ 
\end{lem}
We  now finally verify the laws of the monad $({\EuScript M},\mu,\eta)$.
\begin{lem}
For natural transformations $\mu:{\EuScript M}{\EuScript M}\rightarrow {\EuScript M}$ and $\eta : 1\rightarrow{\EuScript M}$, the following diagrams commute for every functor ${\underline{\EuScript C}}: {\bf Bij}^{op}\rightarrow {\bf Set}$  and a finite set $X$:
\begin{center}
\begin{tikzpicture}
\node (A)  at (0,1) {\small ${\EuScript M}{\EuScript M}{\EuScript M}({\underline{\EuScript C}})(X)$};
\node (B) at (3.5,1) {\small ${\EuScript M}{\EuScript M}({\underline{\EuScript C}})(X)$};
\node (C) at (0,-1) {\small ${\EuScript M}{\EuScript M}({\underline{\EuScript C}})(X)$};
\node (D) at (3.5,-1) {\small ${\EuScript M}({\underline{\EuScript C}})(X)$};
\path[->,font=\footnotesize]
(A) edge node[above]{${\EuScript M}\mu_{{{\underline{\EuScript C}}}_{X}}$} (B)
(A) edge node[left]{$\mu{\EuScript M}_{{{\underline{\EuScript C}}}_X}$} (C)
(B) edge node[right]{$\mu_{{{\underline{\EuScript C}}}_X}$} (D)
(C) edge node[above]{$\mu_{{{\underline{\EuScript C}}}_X}$} (D);
\end{tikzpicture}
\enspace  \begin{tikzpicture}
\node (A)  at (0,1) {\small ${\EuScript M}({{\underline{\EuScript C}}})(X)$};
\node (B) at (3,1) {\small ${\EuScript M}{\EuScript M}({{\underline{\EuScript C}}})(X)$};
\node (C) at (1.5,-1) {\small ${\EuScript M({{\underline{\EuScript C}}})(X)}$};
\path[->,font=\footnotesize]
(A) edge node[above]{${\EuScript M}\eta_{{{\underline{\EuScript C}}}_X}$} (B)
(A) edge node[left]{$id_{{{\underline{\EuScript C}}}_X}$} (C)
(B) edge node[right]{$\mu_{{{\underline{\EuScript C}}}_X}$} (C);
\end{tikzpicture}\enspace  \begin{tikzpicture}
\node (A)  at (0,1) {\small ${\EuScript M}({{\underline{\EuScript C}}})(X)$};
\node (B) at (3,1) {\small ${\EuScript M}{\EuScript M}({{\underline{\EuScript C}}})(X)$};
\node (C) at (1.5,-1) {\small ${\EuScript M}({{\underline{\EuScript C}}})(X)$};
\path[->,font=\footnotesize]
(A) edge node[above]{$\eta{\EuScript M}_{{{\underline{\EuScript C}}}_X}$} (B)
(A) edge node[left]{$id_{{{\underline{\EuScript C}}}_X}$} (C)
(B) edge node[right]{$\mu_{{{\underline{\EuScript C}}}_X}$} (C);
\end{tikzpicture}
\end{center}
\end{lem}
\begin{proof} We begin with the left diagram. Chasing the  associativity of multiplication includes treating several cases, according to the shape of the unrooted tree of \begin{center}${\EuScript M}{\EuScript M}{\EuScript M}({{\underline{\EuScript C}}})(X)=$\texttt{T}$_{\mbox{\texttt{T}}_{\mbox{\texttt{T}}_{{\underline{\EuScript C}}}}}(X)$\end{center} that we start from. The most interesting is the one starting from (a class determined by) an ordinary unrooted tree with corollas given by ordinary unrooted trees built upon \texttt{T}$_{{\underline{\EuScript C}}}$ and we prove the associativity only for this case. Let, therefore, ${\EuScript T}=\{{\EuScript T}_1,\dots ,{\EuScript T}_n;\sigma\}$.\\
\indent By chasing the diagram to the right-down, the action of ${\EuScript M}\mu_{{{\underline{\EuScript C}}}_X}$ corresponds to corolla-per-corolla flattening of ${\EuScript T}$, followed by taking the respective normal forms. Then  $\mu$ flattens additionally the resulting tree and reduces it to a normal form. These actions make the following sequence of steps: 
$$\begin{array}{rcl}
{\EuScript T}&\mapsto &\{[{\it flat}({\EuScript T}_1)]_{\alpha},\dots,[{\it flat}({\EuScript T}_n)]_{\alpha};\sigma\}\\
&\mapsto&\{[{\it nf}({\it flat}({\EuScript T}_1))]_{\alpha},\dots,[{\it nf}({\it flat}({\EuScript T}_n))]_{\alpha};\sigma\}\\
&\mapsto& {\it flat}(\{[{\it nf}({\it flat}({\EuScript T}_1))]_{\alpha},\dots,[{\it nf}({\it flat}({\EuScript T}_n))]_{\alpha};\sigma\})\\
&\mapsto& {\it nf}({\it flat}(\{[{\it nf}({\it flat}({\EuScript T}_1))]_{\alpha},\dots,[{\it nf}({\it flat}({\EuScript T}_n))]_{\alpha};\sigma\}))=R.
\end{array}$$
\indent The action $\mu{\EuScript M}_{{{\underline{\EuScript C}}}_{X}}$ on the left-down side of the diagram corresponds to the action of $\mu$ on the tree ${\EuScript T}$ itself, which flattens it and reduces it to a normal form. Followed by $\mu$ again, this gives us the following sequence:
$$\begin{array}{rcl}
{\EuScript T}&\mapsto &{\it flat}({\EuScript T})\\
&\mapsto& {\it nf}({\it flat}({\EuScript T}))\\
&\mapsto& {\it flat}({\it nf}({\it flat}({\EuScript T})))\\
&\mapsto& {\it nf}({\it flat}({\it nf}({\it flat}({\EuScript T}))))=L.
\end{array}$$
\indent  Let $R'={\it nf}({\it flat}(\{[{\it flat}({\cal T}_1)]_{\alpha},\dots,[{\it flat}({\cal T}_n)]_{\alpha};\sigma\}))$ and $L'={\it nf}({\it flat}({\it flat}({\cal T})))$. By  Lemma \ref{thmhm}, we have that $R=R'$ and $L=L'$, and, by Lemma \ref{assoc_flat}, we have  $R'=L'$.\\[0.1cm]
\indent  We now verify the unit laws  for the case when $[{\EuScript T}]_{\alpha}\in {\EuScript M}({{\underline{\EuScript C}}})(X)$ is determined by an ordinary unrooted tree. Let, therefore,  ${\EuScript T}=\{a_1(x_1,\dots,x_k),\dots,a_n(y_1,\dots,y_r);\sigma\}.$ \\
\indent  By going to the right-down in the first unit diagram (i.e. the diagram in the middle), the action of ${\EuScript M}_{\eta_{{{\underline{\EuScript C}}}_X}}$ turns each corolla  $a_i$ into a single-corolla unrooted tree ${\EuScript T}_i$, leading to a two-level unrooted tree, which is then  flattened and reduced to a normal form by $\mu$. Therefore, the right-down side sequence is as follows: 
$$\begin{array}{rcl}
{\EuScript T}&\mapsto & \{[\{a_1(x_1,\dots,x_k),{\it id}\}]_{\alpha},\dots,[\{a_n(y_1,\dots,y_r);{\it id}\}]_{\alpha};\sigma\} \\[0.1cm]
&\mapsto & \{a_1(x_1,\dots,x_k),\dots,a_n(y_1,\dots,y_r);\underline{\sigma}\} \\[0.1cm]
&\mapsto & \{a_1(x_1,\dots,x_k),\dots ,a_n(y_1,\dots,y_r);\underline{\sigma}'\}
\end{array} $$
  the resulting tree being  exactly ${\EuScript T}$, since $$\underline{\sigma}'(x)=\underline{\sigma}(x)=\left\{ 
	\begin{array}{ll}
		\sigma(x) & \mbox{if }  x\in  \bigcup_{i=1}^{n}FV({\EuScript T}_i)\\[0.1cm]
		x & \mbox{if } x\in V({\EuScript T}_i)\backslash FV({\EuScript T}_i)  
	\end{array}
\right.=\left\{ 
	\begin{array}{ll}
		\sigma(x) & \mbox{if }  x\in  V({\EuScript T})\\[0.1cm]
		x & \mbox{if } x\in V({\EuScript T}_i)\backslash FV({\EuScript T}_i)  
	\end{array}
\right.=\sigma(x),$$ wherein the last equality holds since $V({\EuScript T}_i)\backslash FV({\EuScript T}_i)=\emptyset$, for all $1\leq i\leq n$.\\
\indent By chasing the second unit diagram to the right, ${\EuScript T}$ will first be turned, by the action of $\eta{\EuScript M}_{{{\underline{\EuScript C}}}_X}$, into a single-corolla two-level tree, which will then be flattened and reduced to a normal form by the action of $\mu$. Therefore, we have the sequence
$$\begin{array}{rcl}
{\EuScript T}&\mapsto & \{[\{a_1(x_1,\dots,x_k),\dots, a_n(y_1,\dots,y_r);\sigma\}]_{\alpha},{\it id}_X\} \\[0.1cm]
&\mapsto & \{a_1(x_1,\dots,x_k),\dots,a_n(y_1,\dots,y_r);\underline{{\it id}_X}\} \\[0.1cm]
&\mapsto & \{a_1(x_1,\dots,x_k),\dots,a_n(y_1,\dots,y_r);\underline{{\it id}_X}'\}
\end{array} $$
For the resulting involution $\underline{{\it id}_X}'$ we have\vspace{-0.1cm} $$\underline{{\it id}_X}'(x)=\underline{{\it id}_X}(x)=\left\{ 
	\begin{array}{ll}
		x & \mbox{if }  x\in  FV({\EuScript T})\\[0.1cm]
		\sigma(x) & \mbox{if } x\in V({\EuScript T})\backslash FV({\EuScript T})  
	\end{array}
\right.=\sigma(x).\vspace{-0.1cm}$$
 Therefore, the resulting tree is exactly ${\EuScript T}$.
\end{proof}
\indent Finally, here is the original definition \cite[Definition 2.1]{Getzler:1994pn} of a cyclic operad, recasted in the new syntactic framework.
\begin{defn}\label{unbiased}
A cyclic operad is an algebra over the monad $({\EuScript M},\mu,\eta)$.
\end{defn}
And, under these syntactic glasses, here is the well-known result about the equivalence of the biased and unbiased definitions.
\begin{thm}\label{th1}
A functor ${\underline{\EuScript C}}:{\bf Bij}^{\it op}\rightarrow {\bf Set}$ is endowed with a cyclic operad structure (as described by Definition \ref{d2}) if and only if it is  endowed  with a structure morphism of an ${\EuScript M}$-algebra.
\end{thm}
Before we introduce the $\mu$-syntax in the following section, and ultimatelly prove Theorem \ref{th1}, we indicate  the biased   structure ``hiding'' in the monad approach we just made. As we shall see, the exceptional unrooted trees will be used as pasting schemes of identities of cyclic operads. 
 \subsubsection{The free cyclic operad structure implicit in  $({\EuScript M},\mu,\eta)$}\label{free-forgetiful} A way to specify the free cyclic operad  over  ${\underline{\EuScript C}}$ is given implicitly in \S \ref{monad1}. The functor $F({\underline{\EuScript C}}):{\bf Bij}^{op}\rightarrow {\bf Set}$, underlying free cyclic operad structure, is  defined by $F({\underline{\EuScript C}})(X)={\tt{T}}_{{\underline{\EuScript C}}}(X)$. In the unbiased approach, the monad from \S \ref{monad1} indeed arose from the adjunction $F\vdash U$, where $U$ is the obvious forgetful functor. Before we introduce the rest of the free cyclic operad structure, we fix some notation.
\begin{nota} For an unrooted tree ${\EuScript T}$, a finite set $V$ and a bijection $\vartheta:V\rightarrow V({\EuScript T})$, we shall denote with ${\EuScript T}^{\vartheta}$ the unrooted tree obtained from ${\EuScript T}$ by renaming its variables in a way dictated by $\vartheta$ and adapting  its corollas accordingly.  More precisely, if $a\in {\it Cor}({\EuScript T})$ is an ordiary corolla,  ${\EuScript T}^{\vartheta}$ will, instead of $a$, contain the corolla  $a^{\vartheta|^{FV(a)}}$, and, if $(x,y)\in {\it Cor}({\EuScript T})$ is a special corolla, ${\EuScript T}^{\vartheta}$ will, instead of $(x,y)$,  contain the corolla  $(\vartheta^{-1}(x),\vartheta^{-1}(y))$.
The involution $\sigma^{\vartheta}$ of ${\EuScript T}^{\vartheta}$ is defined as $\sigma^{\vartheta}(v)=\vartheta^{-1}(\sigma(\vartheta(v)))$, for $v\in V$. 
\end{nota}
\indent For a bijection $\kappa : X'\rightarrow X$, the image $[{\EuScript T}]_{\alpha}^{\kappa}$ of $[{\EuScript T}]_{\alpha}\in$ \texttt{T}$_{{\underline{\EuScript C}}}(X)$ under  \texttt{T}$_{{\underline{\EuScript C}}}(\kappa):$ \texttt{T}$_{{\underline{\EuScript C}}}(X)\rightarrow$ \texttt{T}$_{{\underline{\EuScript C}}}(X')$
is the equivalence class $[{\EuScript T}^{{\kappa}\cup {\varepsilon}}]_{\alpha}$, where ${\varepsilon}:V\rightarrow V({\EuScript T})\backslash X$ is an arbitrary bijection, such that $X'\cap V=\emptyset$. \\[0.1cm]
\indent Let $X$ and $Y$ be non-empty finite sets such that for some $x\in X$ and $y\in Y$ we have $X\backslash \{x\}\cap Y\backslash \{y\}=\emptyset$, and let $[{\EuScript T}_1]_{\alpha}\in $ \texttt{T}$_{{\underline{\EuScript C}}}(X)$, $[{\EuScript T}_2]_{\alpha}\in$ \texttt{T}$_{{\underline{\EuScript C}}}(Y)$. The partial composition operation \vspace{-0.1cm} $$\,{_{x}\bullet_y}\, : \mbox{\texttt{T}}_{{\underline{\EuScript C}}}(X)\times\,\mbox{\texttt{T}}_{\underline{\EuScript C}}(Y)\rightarrow \mbox{ \texttt{T}}_{\underline{\EuScript C}}(X\backslash \{x\}\cup Y\backslash \{y\}) \vspace{-0.1cm}$$ is given as \vspace{-0.1cm}$$\label{ss}[{\EuScript T}_1]_{\alpha}\,{_{x}\bullet_y}\,[{\EuScript T}_2]_{\alpha}=[{\it nf}({\EuScript T})]_{\alpha},$$
where  ${\it Cor}({\EuScript T})$ is obtained by taking the union of the sets of corollas of ${\EuScript T}_1$ and ${\EuScript T}_2$, after having previously adapted them in a way that makes this union disjoint with respect to the variables occuring in it. More precisely, if $\vartheta_1:V_1\rightarrow (V({\EuScript T}_1)\backslash X)\cup\{x\}$ and $\vartheta_2:V_2\rightarrow (V({\EuScript T}_2)\backslash Y)\cup\{y\}$ are bijections such that $V_1\cap V_2=\emptyset$,  then \vspace{-0.1cm}
$${\it Cor}({\EuScript T})=\{C^{(\vartheta_1\cup {\it id}_{X\backslash\{x\}})|^{FV(C)}} \,|\, C\in {\it Cor}({\EuScript T}_1)\}\cup\{D^{(\vartheta_2\cup {\it id}_{Y\backslash\{y\}})|^{FV(D)}} \,|\, D\in {\it Cor}({\EuScript T}_2)\}.\vspace{-0.1cm}$$ 
 If  $\sigma_i$ is the involution  of ${\EuScript T}_i$, $i=1,2$, the involution $\sigma$ of ${\EuScript T}$ is defined as follows:
$${\sigma}(v)=\left\{ 
	\begin{array}{ll}
		\vartheta_1^{-1}(\sigma_1(\vartheta_1(v))) & \mbox{if }  v\in  V_1\backslash\vartheta_1^{-1}(x)\\[0.1cm]
		\vartheta_2^{-1}(y)					& \mbox{if }  v=\vartheta_1^{-1}(x)\\[0.1cm]
		\vartheta_2^{-1}(\sigma_2(\vartheta_2(v))) & \mbox{if }  v\in  \vartheta_2^{-1}(y)\\[0.1cm]
		\vartheta_1^{-1}(x) & \mbox{if }  v=\vartheta_2^{-1}(y)\\[0.1cm]
		v & \mbox{if } v\in X\backslash\{x\}\cup Y\backslash\{y\}\,.\\[0.1cm]
	\end{array}
\right.\vspace{-0.1cm}$$ 
\indent For an arbitrary two-element set $\{y,z\}$, we set ${\it id}_{y,z}=[\{(y,z);id_{\{y,z\}}\}]_{\alpha}$.
\section{$\mu$-syntax} \label{mu-syntax}
Backed up with the graphical intuition of the  biased cyclic operad structure on classes of unrooted trees described in \S \ref{free-forgetiful}, in this section we introduce the $\mu$-syntax. 
\subsection{The language and the equations}\label{lang}
The language of the $\mu$-syntax is built over the collection of parameters $P_{{\underline{\EuScript C}}}$ (see \eqref{pc}) and the set of variables $V$. Unlike the combinator syntax \texttt{cTerm}$_{{\underline{\EuScript C}}}$ from \S \ref{uff}, which has only one kind of expressions, the $\mu$-syntax features two different kinds of typed expressions
\begin{center}
\mybox{
\begin{tabular}{c c c }
\textsc{commands} &  & \textsc{terms} \\[0.05cm]
$c ::=\langle s\,|\,t \rangle  \enspace | \enspace \underline{a}\{t_{1},\dots,t_n\}$ &  & $s,t ::= x \enspace | \enspace \mu x.c$ 
\end{tabular}}
\end{center} 
where $a\in P_{{\underline{\EuScript C}}}$ and $x\in V$, 
whose respective typing judgments we denote with $c:X$ and $X\,|\,s$, where $X$ ranges over finite sets. In expressions $c:X$ and $X\,|\,s$, the set $X$ is the type of the command $c$ and of the term $s$, respectively, and the  backward typing judgment $X\,|\,s$ is used merely to further distinguish the representation of terms and commands. 

The assignment of types to commands and terms is done by the following rules:
\vspace{-0.1cm}\begin{center}\mybox{\begin{tabular}{c} 
$\displaystyle\frac{}{\{x\}\,|\,x}$ \,\enspace $\displaystyle\frac{a\in{\EuScript{C}}(\{x_1,\dots,x_n\})\,\enspace \; Y_{i}\,|\,t_{i} \mbox{ \small for all $i\in\{1,\dots,n\}$}}{\underline{a}\{t_{1},\dots,t_n\}:\bigcup_{i=1}^{n} Y_{i}}\,\enspace \; 
\displaystyle\frac{X\,|\, s\enspace\; Y\,|\, t}{\langle s\, |\, t\, \rangle :X\cup Y}$ \enspace  $\displaystyle\frac{c:X \enspace x\in X}{X\backslash \{x\} \,|\, \mu x.c} $
\end{tabular}}\end{center}
where, in the second rule, the sets $Y_{i}$ are   pairwise disjoint, as are $X$ and $Y$ in the third rule. 

Intuitively, commands mimick operations of the free cyclic operad over the functor ${\underline{\EuScript C}}$, and, thereby, a judgement  $c:X$ should be thought of as describing an unrooted tree whose free variables are precisely  the elements of $X$. On the other hand, terms represent operations with one selected entry and the role of the set $X$ in a judgement $X\,|\, s$ is to label all entries {\em except} the selected one. From the tree-wise perspective, this is represented by an unrooted tree whose set of free variables is $X\cup\{x\}$, where $x$ is precisely the variable bound by  $\mu$ (i.e. the variable placed immediatelly on the right of the symbol $\mu$).
\begin{nota}
We shall sometimes denote the commands introduced by the second typing rule as $\underline{a}\{t_x\,|\,x\in X\}$ (for $a\in{\underline{\EuScript C}}(X)$), or  as $\underline{a}\{\sigma\}$, where $\sigma$ assigns to every $x\in X$ a term $t_x$.  The order of appearance of the $t_{x}$'s in $\underline{a}\{t_{x}\,|\,x\in X\}$ is irrelevant. 
Whenever we  use the notation, say $\underline{a}\{t,s\}$, for $a\in\underline{\EuScript{C}}(\{x,y\})$,  it will be  clear from the context whether we mean  $\underline{a}\{t,s\}=\underline{a}\{\sigma\}$, with  $\sigma(x)=t$ and $\sigma(y)=s$, or with $\sigma$ defined in the other way around. \end{nota}
\indent The way commands are constructed is motivated by  the action of the simultaneous and partial   grafting of unrooted trees, formally defined through the composition operation $_x\bullet_y$ from \S \ref{free-forgetiful}. The command $\underline{a}\{t_x\,|\,x\in X\}$, introduced by the second rule, should be imagined as the simultaneous grafting of the corolla $a$  and the ``surrounding'' trees $t_x$, one for each free variable $x$ of $a$, along the variables bound by $\mu$ in each $t_x$. In the special case when, for some $x\in X$, the corresponding term $t_x$ is a variable, say $u$, this process of grafting reduces to the renaming of the variable $x$ of the corolla $a$ to $u$. Therefore, if all terms corresponding to the elements of $X$   are variables from the set, say, $V=\{u,v,w,...\}$, then the corresponding command is $\underline{a}\{u,v,w,\dots\}$ and  it describes the unrooted tree $\{a^{\sigma}(u,v,w,\dots);{\it id}_V\}$, where $\sigma:V\rightarrow X$ is an arbitrary bijection. The command $\langle s\,|\,t \rangle$  describes the  grafting of unrooted trees represented by the terms $s$ and $t$ along their  variables bound by $\mu$. Therefore, the pattern $\langle \mu x.\underline{\hspace{0.3cm}}\, | \mu y.\underline{\hspace{0.3cm}}\,\rangle$ corresponds  to the  composition   $(-){{_{x}\bullet_{y}}}(-)$ on classes of unrooted trees.  \\[0.1cm]
\indent The equations of the $\mu$-syntax are  \begin{center}\mybox{\begin{tabular}{lcllllllcl}
$\langle s\,|\, t\rangle=\langle t\,|\, s\rangle$ &  & {\small{\texttt{(MU1)}}} &&&&& $\mu x.c=\mu y.c[y/x]$ & & {\small\texttt{(MU3)}} \\[0.25cm]
$\langle\mu x.c\,|\, s\rangle=c[s/x]$ &  & {\small\texttt{(MU2)}} &&&&&  $\underline{a}\{t_x\,|\,x\in X\}=\underline{a^{\sigma}}\{t_{\sigma(y)}\,|\, y\in Y\}$ & & {\small\texttt{(MU4)}} 
\end{tabular}}\end{center}
where,  in {\small\texttt{(MU2)}}, $c[s/x]$ denotes the command $c$ in which the unique occurrence of the variable $x$ has been replaced by the term $s$,
in {\small\texttt{(MU3)}} $y$ has to be fresh with respect to all variables of $c$ except $x$, and in {\small\texttt{(MU4)}} $\sigma:Y\rightarrow X$ is an arbitrary bijection. \\
\indent The  equation {\small\texttt{(MU1)}}  stipulates the symmetry of grafting of unrooted trees, i.e.  the   commutativity of  composition operations  $_x\bullet_y$.\\
\indent The equations {\small\texttt{(MU3)}} and {\small\texttt{(MU4)}} are $\alpha$-conversions. Intuitively,  $\alpha$-conversion   tells that the name of the entry selected for the composition does not matter, which reflects the equivariance  of composition operations  $_x\bullet_y$. In    more simple  terms,  $\alpha$-conversion tells that the function $f(x)$   is the same as the function $f(y)$. \\
\indent The substitution $c[s/x]$, figuring in the  equation  {\small\texttt{(MU2)}}  (as well as the substitution $c[y/x]$ from {\small\texttt{(MU3)}}), must be performed in  the {\em capture-avoiding} manner. This means that the variables which were originally ``free'' (i.e. {\em not bound} by $\mu$)  in $c$ can not become ``captured'' (i.e. {\em bound} by $\mu$) after the substitution is made. This is achieved by renaming,  prior to the substitution, all the bound variables in $c$ and $s$, so that they are all turned mutually distinct, and then performing the appropriate substitution. For example, $$\mu x.\underline{a}\{x,y\}[x/y]\neq \mu x.\underline{a}\{x,x\},\quad \mbox{ but } \quad \mu x.\underline{a}\{x,y\}[x/y]= \mu z.\underline{a^{\sigma}}\{z,y\}[x/y]=\mu z.\underline{a^{\sigma}}\{z,x\},$$ where $\sigma$ renames $x$ to $z$.

The equation {\small\texttt{(MU2)}}) is quite evidently reminescent of the $\beta$-reduction of  $\lambda$-calculus, when considered as a rewriting rule $\langle\mu x.c\,|\, s\rangle\rightarrow c[s/x]$,
and it essentially captures the same idea of function application as $\lambda$-calculus. The intuition becomes more tangible from the point of view of trees: the commands $\langle\mu x.c\,|\, s\rangle$ and $c[s/x]$, equated with  {\small\texttt{(MU2)}},  describe two ways to build (by means of grafting) the same unrooted tree. Here is an example.
\begin{example} \label{ex4}
Consider the unrooted tree \vspace{-0.1cm}$${\EuScript T}=\{a(x_1,x_2,x_3,x_4),b(y_1,y_2,y_3,y_4),c(z_1,z_2);\sigma\},\vspace{-0.1cm}$$ where $\sigma=(x_3\enspace y_1)(x_4\enspace z_1)$. One way to build ${\EuScript T}$ is to graft along $x_4$ and $z_1$ unrooted trees ${\EuScript T}_1=\{a(x_1,x_2,x_3,x_4), b(y_1,y_2,y_3,y_4);\sigma_1\}$, where $\sigma_1=(x_3\enspace y_1)$, and ${\EuScript T}_2=\{c(z_1,z_2);{\it id}_{\{z_1,z_2\}}\}$, singled out with dashed lines in the left picture below:
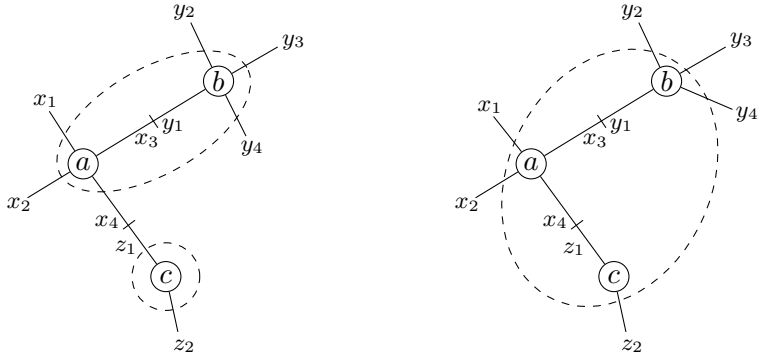
\vspace{-0.25cm}
\begin{center}
\begin{tikzpicture}
\node (f) [circle,fill=none,draw=black,minimum size=4mm,inner sep=0.1mm]  at (-1,0) {$a$};
\node (g) [circle, fill=none,draw=black,minimum size=4mm,inner sep=0.1mm]  at (0.8,1.1) {$b$};
\node (h) [circle, minimum size=4mm,fill=none,draw=black,inner sep=0.1mm]  at (0.1,-1.5) {$c$};
\node (2) [label={[xshift=-0.1cm, yshift=-0.275cm]\footnotesize{$x_1$}},circle, minimum size=1mm,fill=none,draw=none]  at (-1.4,0.7) {};
\node (s) [label={[xshift=0.045cm, yshift=-0.41cm]\footnotesize{$z_2$}},circle,  minimum size=1mm,fill=none,draw=none]  at (0.3,-2.42) {};
\node (4) [label={[xshift=0.05cm, yshift=-0.4
cm]\footnotesize{$x_2$}},circle, minimum size=1mm,fill=none,draw=none]  at (-1.9,-0.55) {};
\node (b) [label={[xshift=0cm, yshift=-0.45cm]\footnotesize{$y_2$}},circle,  minimum size=1mm,fill=none,draw=none]  at (0.35,2.05) {};
\node (c) [label={[xshift=0.05cm, yshift=-0.45cm]\footnotesize{$y_3$}},circle,  minimum size=1mm,fill=none,draw=none]  at (1.75,1.65) {};
\node (d) [label={[xshift=0cm, yshift=-0.41
cm]\footnotesize{$y_4$}},circle, minimum size=1mm,fill=none,draw=none]  at (1.25,0.2) {};
\node (x) [label={[xshift=-0.15cm, yshift=-0.6cm]\footnotesize{$x_3$}},circle,  minimum size=1mm,fill=none,draw=none]  at (0,0.5) {};
\node (y) [label={[xshift=0.2cm, yshift=-0.42cm]\footnotesize{$y_1$}},circle, minimum size=1mm,fill=none,draw=none]  at (0,0.5) {};
\node (x) [label={[xshift=-0.25cm, yshift=-0.28cm]\footnotesize{$x_4$}},circle, minimum size=1mm,fill=none,draw=none]  at (-0.43,-0.9) {};
\node (x) [label={[xshift=0.1cm, yshift=-0.67cm]\footnotesize{$z_1$}},circle, minimum size=1mm,fill=none,draw=none]  at (-0.53,-0.85) {};
\draw (-0.12,0.65)--(-0.006,0.48);
\draw (f)--(g);
\draw (f)--(h);
\draw (h)--(s);
\draw (-0.47,-0.87)--(-0.31,-0.75);
\draw (f)--(-1.45,0.7) ;\draw (f)--(4);
\draw (g)--(b);\draw (g)--(c);\draw (g)--(d);
\draw[rotate=30,dashed] (0.23,0.52) ellipse (1.43cm and 0.7cm);
\draw[dashed] (0.1,-1.5) circle (0.45cm);
\end{tikzpicture}\enspace\enspace\enspace\enspace\enspace\enspace\enspace\enspace\enspace
\begin{tikzpicture}
\node (f) [circle,fill=none,draw=black,minimum size=4mm,inner sep=0.1mm]  at (-1,0) {$a$};
\node (g) [circle, fill=none,draw=black,minimum size=4mm,inner sep=0.1mm]  at (0.8,1.1) {$b$};
\node (h) [circle, minimum size=4mm,fill=none,draw=black,inner sep=0.1mm]  at (0.1,-1.5) {$c$};
\node (2) [label={[xshift=-0.15cm, yshift=-0.35cm]\footnotesize{$x_1$}},circle, minimum size=1mm,fill=none,draw=none]  at (-1.4,0.7) {};
\node (s) [label={[xshift=0.045cm, yshift=-0.41cm]\footnotesize{$z_2$}},circle,  minimum size=1mm,fill=none,draw=none]  at (0.3,-2.42) {};
\node (4) [label={[xshift=0.05cm, yshift=-0.4
cm]\footnotesize{$x_2$}},circle, minimum size=1mm,fill=none,draw=none]  at (-1.9,-0.55) {};
\node (b) [label={[xshift=0cm, yshift=-0.45cm]\footnotesize{$y_2$}},circle,  minimum size=1mm,fill=none,draw=none]  at (0.35,2.05) {};
\node (c) [label={[xshift=0.05cm, yshift=-0.43cm]\footnotesize{$y_3$}},circle,  minimum size=1mm,fill=none,draw=none]  at (1.75,1.65) {};
\node (d) [label={[xshift=0.025cm, yshift=-0.41
cm]\footnotesize{$y_4$}},circle, minimum size=1mm,fill=none,draw=none]  at (1.85,0.65) {};
\node (x) [label={[xshift=-0.15cm, yshift=-0.6cm]\footnotesize{$x_3$}},circle,  minimum size=1mm,fill=none,draw=none]  at (0,0.5) {};
\node (y) [label={[xshift=0.2cm, yshift=-0.42cm]\footnotesize{$y_1$}},circle, minimum size=1mm,fill=none,draw=none]  at (0,0.5) {};
\node (x) [label={[xshift=-0.25cm, yshift=-0.28cm]\footnotesize{$x_4$}},circle, minimum size=1mm,fill=none,draw=none]  at (-0.43,-0.9) {};
\node (x) [label={[xshift=0.1cm, yshift=-0.69cm]\footnotesize{$z_1$}},circle, minimum size=1mm,fill=none,draw=none]  at (-0.53,-0.85) {};
\draw (-0.12,0.65)--(-0.006,0.48);
\draw (f)--(g);
\draw (f)--(h);
\draw (h)--(s);
\draw (-0.47,-0.87)--(-0.31,-0.75);
\draw (f)--(-1.5,0.63) ;\draw (f)--(4);
\draw (g)--(b);\draw (g)--(c);\draw (g)--(d);
\draw[rotate=65,dashed] (-0.15,-0.12) ellipse (1.78cm and 1.35cm);

\end{tikzpicture}
\end{center}
\vspace{-0.2cm}
The unrooted tree ${\EuScript T}_1$ (in the upper part of the left picture) can itself be seen as a grafting, namely the  simultaneous grafting of the corolla $a$ and its surrounding trees: in this case this involves explicit grafting only with the corolla $b$ (along the free variables $x_3$ and $y_1$). This way of constructing ${\EuScript T}$  is described by  the command  \begin{equation}\label{lhs}\langle\mu x_4.\underline{a}\{x_1,x_2,\mu y_1.\underline{b}\{y_1,y_2,y_3,y_4\},x_4\}\,|\, \mu z_1.\underline{c}\{z_1,z_2\}\rangle\,  \tag{*}\end{equation} that witnesses the fact that ${\EuScript T}_1$ and ${\EuScript T}_2$ are connected along their selected free variables $x_4$ and $z_1$, respectively:  $x_4$ and $z_1$ are bound with $\mu$ in the terms corresponding to these two trees. The subterm $\underline{a}\{x_1,x_2,\mu y_1.\underline{b}\{y_1,y_2,y_3,y_4\},x_4\}$ on the left-hand side is the command that accounts for the simultaneous grafting of the corolla $a$ and its surrounding trees, while $\underline{c}\{z_1,z_2\}$ on the right-hand side stands for the corolla $c$. On the other hand, we could have chosen to build the tree ${\EuScript T}$  simply by making the simultaneous grafting of the corolla $a$ and its surrounding trees, as indicated on the picture on the right. This way of building ${\EuScript T}$ is described with the command $\underline{a}\{x_1,x_2,\mu y_1.\underline{b}\{y_1,y_2,y_3,y_4\},\mu z_1.\underline{c}\{z_1,z_2\}\}$, which is, up to substitution, exactly the command   $$\underline{a}\{x_1,x_2, \mu y_1.\underline{b}\{y_1,y_2,y_3,y_4\},x_4\}[\mu z_1.\underline{c}\{z_1,z_2\}/x_4]$$  to which \eqref{lhs} reduces by applying the rewriting rule $\langle\mu x.c\,|\, s\rangle\rightarrow c[s/x]$. \end{example}

\indent We shall denote with $\mu$\texttt{Exp}$_{{\underline{\EuScript C}}}$ the set of all expressions of the $\mu$-syntax induced by $P_{{\underline{\EuScript C}}}$, and we shall use $\mu$\texttt{Term}$_{{\underline{\EuScript C}}}$ and $\mu$\texttt{Comm}$_{{\underline{\EuScript C}}}$ to denote the subsets of  terms and commands of $\mu$\texttt{Exp}$_{{\underline{\EuScript C}}}$,  respectively. As in the case of the combinator syntax, the set of expressions (resp. terms and commands) of   type $X$ will be denoted by $\mu$\texttt{Exp}$_{{\underline{\EuScript C}}}(X)$ (resp. $\mu$\texttt{Term}$_{{\underline{\EuScript C}}}(X)$ and $\mu$\texttt{Comm}$_{{\underline{\EuScript C}}}(X)$).
\subsection{$\mu$-syntax as a rewriting system} \label{mu-rewriting}
Let  $\leadsto$ be the rewriting relation defined on $\mu$\texttt{Exp}$_{{\underline{\EuScript C}}}$ as (the reflexive and transitive closure of) the union of  rewriting rules 
$$ \langle s\,|\, t\rangle \leadsto \langle t\,|\, s\rangle\quad  \mbox{ and }\quad   \langle\mu x.c\,|\, s\rangle\leadsto c[s/x]   
$$
obtained by orienting from left to right the  equations {\small\texttt{(MU1)}} and {\small\texttt{(MU2)}}, respectively, which is, moreover, congruent with respect to {\small\texttt{(MU3)}}, {\small\texttt{(MU4)}} and substitution\footnote{ Since the precautionary renaming which ensures that the substitution $c[s/x]$ is done in the capture-free manner is {\em non-deterministinc}, the rewriting relation $\leadsto$ is formally defined on the equivalence classes of expressions of the $\mu$-syntax with respect to {\small\texttt{(MU3)}} and {\small\texttt{(MU4)}},   just as   the usual rewriting systems in $\lambda$-calculus are actually defined on $\alpha$-conversion classes.}.  \\
\indent The non-confluence  of the rewriting system $(\mu$\texttt{Exp}$_{{\underline{\EuScript C}}},\rightsquigarrow)$ shows up immediately: for the  reductions $$c_2[\mu x.c_1/y]\leftsquigarrow\langle\mu x.c_1\,|\,\mu y.c_2\rangle\rightsquigarrow c_1[\mu y.c_2/x]$$  arising due to {\small\texttt{(MU1)}} (which makes the whole reduction system symmetric), we do not have a way to exhibit a command $c$, such that $c_2[\mu x.c_1/y]\leadsto c$ and $c_1[\mu y.c_2/x]\leadsto c$. Nevertheless, all three commands above describe the same unrooted tree. \\
\indent However, modulo the trivial commuting conversion,  this rewriting system is terminating:  the number of $\mu$-binders in an expression is strictly decreasing at each    reduction step of the form $\langle\mu x.c\,|\, s\rangle\rightsquigarrow c[s/x]$ (which makes it impossible to have an infinite sequence of reductions of this kind).  It is straightforward to prove that the set 
$\mu$\texttt{Exp}$_{{\underline{\EuScript C}}}^{{\it nf}}=\mu$\texttt{Comm}$_{{\underline{\EuScript C}}}^{{\it nf}}\cup\mu$\texttt{Term}$_{{\underline{\EuScript C}}}^{{\it nf}}$ of normal forms
is generated by the  following rules:  
\begin{center}\mybox{\begin{tabular}{c}  
$\displaystyle\frac{}{x\in \mu\mbox{\texttt{Term}}_{{\underline{\EuScript C}}}^{{\it nf}}}$ \enspace \enspace\enspace $\displaystyle\frac{a\in{\underline{\EuScript{C}}}(X)\enspace \enspace\enspace  t_x\in \mu\mbox{\texttt{Term}}_{{\underline{\EuScript C}}}^{{\it nf}} \mbox{ \small for all $x\in X$}}{{\underline{a}}\{t_x\,|\,x\in X\}\in \mu\mbox{\texttt{Comm}}_{{\underline{\EuScript C}}}^{{\it nf}}}$\enspace \enspace\enspace 
$\displaystyle\frac{c\in \mu\mbox{\texttt{Comm}}_{{\underline{\EuScript C}}}^{{\it nf}}}{\mu x.c\in \mu\mbox{\texttt{Term}}_{{\underline{\EuScript C}}}^{{\it nf}}}$ 
\end{tabular}}\end{center}

 In the next example, we examine the shape of normal forms in relation with unrooted trees. 
\begin{example} Let ${\EuScript T}$ be the unrooted tree from \textsc{Example} \ref{ex4}. Here is the list of commands in normal form that describe ${\EuScript T}$:  $$\underline{a}\{x_1,x_2,\mu y_1.\underline{b}\{y_1,y_2,y_3,y_4\},\mu z_1.\underline{c}\{z_1,z_2\}\} ,\vspace{-0.1cm}$$
 $$\underline{b}\{\mu x_3.\underline{a}\{x_1,x_2,x_3,\mu z_1.\underline{c}\{z_1,z_2\}\},y_2,y_3,y_4\},\quad\underline{b}\{\mu x_3.\underline{c}\{\mu x_4.\underline{a}\{x_1,x_2,x_3,x_4\},z_2\},y_2,y_3,y_4\} ,$$ $$\underline{c}\{\mu x_4.\underline{a}\{\mu y_1.\underline{b}\{y_1,y_2,y_3,y_4\},x_2,x_3,x_4\},z_2\},\quad\underline{c}\{\mu x_4.\underline{b}\{\mu x_3.\underline{a}\{x_1,x_2,x_3,x_4\},y_2,y_3,y_4\},z_2 \} .$$  Each of the   commands records  the free variables and corollas of ${\EuScript T}$:  free  variables are the variables not bound with $\mu$ ($x_1$, $x_2$, $y_2$, $y_3$, $y_4$ and $z_2$), and the corollas correspond to the underlined parameters ($a$, $b$ and $c$). The variables involved in edges of ${\EuScript T}$ ($x_3$, $y_1$, $x_4$ and $z_1$) can also be recovered from the list, as the variables  bound with $\mu$.  For example, in the first command we see  that $y_1$ and $z_1$ are explicitly bound by $\mu$, while for $x_3$ and $x_4$ we could say that  they are implicitly bound, given that they are   replaced with a non-variable term.
\end{example}
\indent In general, the set  $\mu${\texttt{Comm}}$_{{\underline{\EuScript C}}}^{{\it nf}}$     describes decompositions of unrooted trees of the following kind:  pick a corolla $a$  of a tree, and then proceed recursively so in all the connected components of the graph resulting from the removal of $a$ . (We provide in \S \ref{pru} an algorithmic computation of these  connected components). \\
\indent Amusingly, one can show that, if  {\small\texttt{(MU1)}} gets oriented in the other way around,  the  normal forms of the resulting rewriting system will be in one-to-one correspondence with the combinators of Section \ref{s1}, and thus describe decompositions of unrooted trees of the following kind: pick an edge $e$ of the tree, and then proceed  recursively so in  the two connected components of the graph resulting from the removal of $e$.\\
\indent These two extremes substantiate our informal explanation of the $\mu$-syntax as a mix of partial composition and simultaneous composition styles.

\subsection{The interpretation of the $\mu$-syntax in an arbitrary  cyclic operad }\label{muint}
We next consider the semantic aspect of the $\mu$-syntax relative to unrooted trees that we intuitively brought up in \S \ref{lang} and \S \ref{mu-rewriting}, by  defining an interpretation function of the  $\mu$-syntax into a cyclic operad, as defined in Definition \ref{d2}. We ascribe meaning to the  $\mu$-syntax by first translating it to the combinator syntax from Section \ref{s1}.\\[0.1cm]
\indent   The translation function \vspace{-0.1cm}\begin{center}$[[\rule{.4em}{.4pt}]]:$ $\mu$\texttt{Exp}$_{{\underline{\EuScript C}}}\rightarrow$ \texttt{cTerm}$_{{\underline{\EuScript C}}}$ \end{center}\vspace{-0.1cm} is defined recursively as follows, wherein the assignment of a combinator to a term $t\in\mu$\texttt{Term}$_{{\underline{\EuScript C}}}$ is  indexed by a variable that is fresh relative to $t$:
\begin{itemize}
\item $[[x]]_{y}={\it id}_{x,y}$,\\[-0.5cm]
\item if, for each $x\in X$,   $[[t_x]]_{\overline{x}}$ is a translation of the term $t_x$, then \vspace{-0.1cm}$$[[\underline{a}\{t_x\,|\, x\in X\}]]=a(\varphi) ,\vspace{-0.1cm} $$where  $a(\varphi)$ denotes the combinator corresponding to the simultaneous composition  determined by $a\in{\EuScript{C}}(X)$ and $\varphi : x\mapsto ([[t_x]]_{\overline{x}},\overline{x})$ (see \eqref{simultaneous}),\\[-0.5cm]
\item $[[\mu x.c]]_{y}=[[c[y/x]\,]]$, and\\[-0.5cm]
\item $[[\langle s\,|\, t\rangle]]=[[s]]_{x}\,{_{x}\circ_{y}} \,[[t]]_{y}$ .
\end{itemize}

\indent In order to show that $[[\rule{.4em}{.4pt}]]$ is well-defined, we introduce the following notational conventions. For a command $c:X$ (resp.  term $X\,|\,t$) and a bijection $\sigma:X'\rightarrow X$, we define$$c^{\sigma}:=c[\dots , \sigma^{-1}(x)/x ,\dots]\enspace (\mbox{resp.}\enspace t^{\sigma}:=t[\dots , \sigma^{-1}(x)/x ,\dots])$$ as a simultaneous substitution (renaming) of the variables from the set $X$ (guided by $\sigma$). One of the basic properties of the introduced substitution is  the equality $(\mu a.c)^{\sigma}=\mu a.c^{\sigma_a}$ (for the definition of $\sigma_a$, see the paragraph Notation and conventions in the Introduction).   \\[0.1cm]
\indent The way  $c^{\sigma}$ is defined  indicates that its translation should be the combinator  $[[c]]^{\sigma}:X'$. The following lemma ensures that this is exactly the case. In its statement,  $[[\rule{.4em}{.4pt}]]_{X}$   denotes the restriction of  $[[\rule{.4em}{.4pt}]]$ on   $\mu$\texttt{Exp}$_{{\underline{\EuScript C}}}(X)$.  
\begin{lem}\label{co} For a  bijection $\sigma:X'\rightarrow X$,   $t\in\mu{\tt{Term}}_{{\underline{\EuScript C}}}(X)$ and $c\in\mu{\tt{Comm}}_{{\underline{\EuScript C}}}(X)$, the following two equalities hold: $$[[t^{\sigma}]]_y=[[t]]_y^{{\sigma}_y} \quad\quad \mbox{and} \quad\quad [[c^{\sigma}]]=[[c]]^{\sigma}.$$
\end{lem}
To verify that the definition of $[[\rule{.4em}{.4pt}]]$ is valid, we shall also need the following result.
\begin{lem}[Substitution lemma]\label{sl} Let $X\cap Y=\emptyset$, $t\in \mu${\em\texttt{Term}}$_{{\underline{\EuScript C}}}(Y)$ and $x\in X$. Then, for $s\in\mu{\tt{Term}}_{{\underline{\EuScript C}}}(X)$ and $c\in\mu{\tt{Comm}}_{{\underline{\EuScript C}}}(X)$, the following two equalities hold: $$[[s[t/x]]]_u=[[s]]_u\,{_{x}\circ_{v}} \,[[t]]_{v}\quad\quad \mbox{and}\quad\quad [[c[t/x]]]=[[c]]\,{_{x}\circ_{v}} \,[[t]]_{v}.$$ 
\end{lem}
\begin{proof}
\indent If $t$ is a variable, say $y$, then,  by {\tt{(U1)}} and {\tt{(EQ)}}, we get\vspace{-0.1cm} $$[[s[y/x]]]_u=[[s^{{\it id}_X^{y/x}}]]_u={[[s]]_u}^{{\it id}_X^{y/x}}=[[s]]_u\,{_{x}\circ_{v}}\,{\it id}_{v,y}=[[s]]_u\,{_{x}\circ_{v}}\,[[y]]_v\, ,\vspace{-0.1cm}$$ and, analogously,\vspace{-0.1cm} $$[[c[y/x]]]=[[c^{{\it id}_X^{y/x}}]]=[[c]]^{{\it id}_X^{y/x}}=[[c]]\,{_{x}\circ_{z}}\, {\it id}_{z,y}= [[c]]\,{_{x}\circ_{z}}\,[[y]]_z .\vspace{-0.1cm}$$
\indent If $t= \mu y.c_1$, we proceed by induction on the structure of $s$, i.e. $c$.
\begin{itemize}
\item If $s=x$, then, again by {\tt{(U1)}} and {\tt{(EQ)}}, we get
{\small $$\begin{array}{rclclclcl}
[[x[\mu y.c_1/x]]]_{u}&=&[[\mu y.c_1]]_{u}&=&[[c_1[u/y]]]&=&[[c_1^{id_Y^{u/y}}]]\\[0.2cm]
&=&[[c_1]]^{id_Y^{u/y}}&=&[[c_1]]\,{_{y}\circ_{x}}\,{\it id}_{x,u}&=&[[c_1]]\,{_{y}\circ_{x}}\,[[x]]_{u}&=&[[c_1[u/y]]]\,{_{u}\circ_{x}}\,[[x]]_{u} .
\end{array} \vspace{-0.1cm}$$}
\item Next, assume that $c:X\cup\{z\}$ satisfies the equality and let $s=\mu z.c$. Denote $U=X\backslash\{x\}\cup\{z\}\cup Y$. By by {\tt{(U1)}} and {\tt{(EQ)}}, we have\vspace{-0.1cm}
{\small $$\begin{array}{rclclcl}
[[\mu z.c[\mu y.c_1/x]]]_{u}&=&[[\mu z.(c[\mu y.c_1/x])]]_{u}&=&[[c[\mu y.c_1/x][u/z]]] &=&[[c[\mu y.c_1/x]^{{\it id}_U^{u/z}}]]\\[0.2cm]
 &=&[[c[\mu y.c_1/x]]]^{{\it id}_U^{u/z}} &=&([[c]]\,{_{x}\circ_{y}}\,[[c_1]])^{{\it id}_U^{u/z}}&=& [[c]]^{{\it id}_U^{u/z}} \,{_{x}\circ_{y}}\,[[c_1]]\\[0.2cm]
&=& [[c[u/z]]]\,{_{x}\circ_{v}}\,[[c_1[v/y]]]&=& [[\mu v.c]]_{u}\,{_{x}\circ_{v}}\,[[c_1[v/y]]] .
\end{array}\vspace{-0.1cm}$$}
\item Let $X=X_1\cup X_2$ and suppose that $c = \langle t_1\,|\, t_2\rangle$, where $X_1\,|\, t_1$ and $X_2\,|\, t_2$ satisfy the claim. Without loss of generality, we can assume  that $x\in X_2$. By {\tt{(A1)}}, we   have

\vspace{-0.4cm}
{\small $$\begin{array}{rclcl}
[[\langle t_1\,|\, t_2\rangle [\mu y.c_1/x]]] &=&[[\langle t_1\,|\, t_2[\mu y.c_1/x]\rangle]]&=& [[t_1]]_u\,{_{u}\circ_{v}}\,[[t_2[\mu y.c_1/x]]]_v\\[0.2cm]
&=& [[t_1]]_u\,{_{u}\circ_{v}}\, ([[t_2]]_v \,{_{x}\circ_{w}}\,[[\mu y.c_1]]_w)&=& ([[t_1]]_u\,{_{u}\circ_{u}}\, [[t_2]]_u) \,{_{x}\circ_{w}}\,[[\mu y.c_1]]_w\\[0.2cm]
&=& [[\langle t_1\,|\, t_2\rangle ]] \,{_{x}\circ_{v}}\, [[\mu y.c_1]]_v .
\end{array}\vspace{-0.1cm}$$ }
\item Finally, let $X=\bigcup_{z\in Z}Y_z$ and suppose that  $c=\underline{a}\{t_z\,|\, z\in Z\}$, where for all $z\in Z$, $Y_z\,|\, t_z$ satisfy the claim. Suppose, moreover, that for $u\in Z$, $x\in Y_u$. Then, on one hand, we have
$$
[[\underline{a}\{t_z\,|\, z\in Z\} [\mu y.c_1/x]]]=[[\underline{a}\{\{t_z\,|\, z\in Z\backslash\{u\}\}\cup \{t_u[\mu y.c_1/x]\}\}]]= a(\varphi) ,
$$
where $\varphi: z\mapsto ([[t_z]]_{\overline{z}},\overline{z})$, for all $z\in Z\backslash\{u\}$, and $\varphi: u\mapsto ([[t_u[\mu y.c_1/x]]]_{\overline{u}},\overline{u})$. On the other hand, $$[[\underline{a}\{t_z\,|\, z\in Z\}]]\,{_{x}\circ_{v}}\, [[\mu y.c_1]]_v= a(\psi_1)\,{_{x}\circ_{v}}\, [[\mu y.c_1]]_v , $$
where $\psi_1:z\mapsto ([[t_z]]_{\overline{z}},\overline{z})$, for all $z\in Z$. By Lemma \ref{geneq}, $$a(\psi_1)\,{_{x}\circ_{v}}\, [[\mu y.c_1]]_v=a(\psi_2) ,$$ where $\psi_2=\psi_1$ on $Z\backslash\{a\}$, and $\psi_2:u\mapsto ([[t_u]]_{\overline{u}}\,{_{x}\circ_{v}}\,[[\mu y.c_1]]_v ,\overline{u})$. Hence, we need to prove that \vspace{-0.1cm}$$[[t_u[\mu y.c_1/x]]]_{\overline{u}}=[[t_u]]_{\overline{u}}\,{_{x}\circ_{v}}\,[[\mu y.c_1]]_v ,\vspace{-0.1cm}$$ 
but this equality is exactly the induction hypothesis for the term $t_u$.  
\vspace{-0.45cm}
\end{itemize}
\end{proof}
Let $=_{\mu}$ (resp. $=$) be the smallest equivalence relation on $\mu${\texttt{Exp}}$_{{\underline{\EuScript C}}}$ (resp. {\texttt{cTerm}}$_{{\underline{\EuScript C}}}$) generated by the equations of $\mu$-syntax (resp. by the equations of Definition \ref{entriesonly}). 
\begin{thm}\label{wd}
The translation function {\em $[[\rule{.4em}{.4pt}]]:$ $\mu$\texttt{Exp}$_{{\underline{\EuScript C}}}\rightarrow$ \texttt{cTerm}$_{{\underline{\EuScript C}}}$} is well-defined, i.e., it induces a map from $\mu${\em\texttt{Exp}}$_{{\underline{\EuScript C}}}/_{=_{\mu}}$ to {\em\texttt{cTerm}}$_{{\underline{\EuScript C}}}/_{=}$. Moreover, the induced map is a bijection.
\end{thm}
\begin{proof}
The equation {\small\texttt{(MU1)}} is valid in the world of combinators, as it gets translated to ${\small\texttt{(CO)}}$. As for {\small\texttt{(MU2)}}, for a command $c:X$, by Lemma \ref{sl}, we get:
{\small  $$ [[\langle\mu x.c\,|\,t\rangle]]=[[\mu x.c]]_u\,{_{u}\circ_{v}} \,[[t]]_v
=[[c[u/x]]] \,{_{u}\circ_{v}} \, [[t]]_v
=[[c]]^{{\it id}_X^{u/x}} \,{_{u}\circ_{v}} \, [[t]]_v
=[[c]]\,{_{x}\circ_{v}} \, [[t]]_v
= [[c[t/x]]].$$ }
For {\small\texttt{(MU3)}} and {\small\texttt{(MU4)}}, we have\vspace{-0.1cm}
 $$[[\mu x.c]]_{u}=[[c[u/x]]]=[[c[y/x][u/y]]]=[[\mu y.c[y/x]]]_u\vspace{-0.1cm}$$ and
 $$[[\underline{a}^{\sigma}\{t_{\sigma(y)}\,|\, y\in Y\}]]=a^{\sigma}(\varphi')=a^{\sigma}(\varphi\circ\sigma)=a(\varphi)=[[\underline{a}\{t_x\,|\, x\in X\}]], $$ where $\varphi':y\mapsto([[t_{\sigma(y)}]]_{\overline{\sigma(y)}},\overline{\sigma(y)})$ and $\varphi:\sigma(y) \mapsto([[t_{\sigma(y)}]]_{\overline{\sigma(y)}},\overline{\sigma(y)})$, respectively.  The inverse translation   is obtained  via  correspondence $(-) {_{x}\circ_{y}} (-)\enspace\mapsto \enspace \langle\mu x.\, \rule{.5em}{.4pt}\, |\, \mu y.\,\rule{.5em}{.4pt}\, \rangle .$
\end{proof}
We  define the interpretation of the $\mu$-syntax in  an arbitrary  cyclic operad  ${\EuScript C}$, as the composition $$[\,[[\rule{.4em}{.4pt}]]\,]_{{\EuScript C}}: \mu {\tt{Exp}}_{{\underline{\EuScript C}}}\rightarrow{{\EuScript C}},$$ where the interpretation $[\rule{.4em}{.4pt}]_{{\EuScript C}}: {\tt{cTerm}}_{{\underline{\EuScript C}}}\rightarrow {{\EuScript C}}$ arises as explained after Definition \ref{d2}.

\subsection{$\mu$-syntax does the job!}
The theorem below  puts the $\mu$-syntax in line with already established frameworks for defining a cyclic operad. 
\begin{thm}\label{equiv}
The quotient set of  the commands of the $\mu$-syntax relative to the relation $=_{\mu}$, is in one-to-one correspondence with the one of unrooted trees relative to the $\alpha$-conversion. In other words, for every finite set $X$, there exists a bijection \begin{center}$\Phi_X:\mu${\em\texttt{Comm}}$_{{\underline{\EuScript C}}}(X)_{/_{=_{\mu}}}\rightarrow \enspace${\em\texttt{T}}$_{{\underline{\EuScript C}}}(X)$.\end{center} 
\end{thm}
The proof of Theorem \ref{equiv}  goes through a new equality $='$ on $\mu${\texttt{Comm}}$^{\it nf}_{\underline{\EuScript C}}(X)$, as well as suitably tailored decompositions of unrooted trees, necessary for establishing the injectivity of $\Phi_X$.   We first describe these decompositions and the  equality $='$ and  then prove the theorem.
\subsubsection{``Pruning" of unrooted trees}\label{pru}
We  describe an algorithm that takes an ordinary unrooted tree ${\EuScript T}$,   a corolla $a\in {\it Cor}({\EuScript T})$ and 
a variable $v\in FV(a)\backslash FV({\EuScript T})$, and returns a proper subtree ${\EuScript T}_v$ of ${\EuScript T}$, the subtree ``plucked" from $a$ at the junction of $v$ and $\sigma(v)$, where $\sigma$ is the involution of ${\EuScript T}$.  In the sequel, for an arbitrary corolla $b\in {\it Cor}({\EuScript T})$ and $w\in FV(b)\backslash FV({\EuScript T})$, $S_w(b)$ will denote the corolla adjacent to $b$ along the edge $(w,\sigma(w))$, if such a corolla exists. \\
\indent We first specify how to generate the set ${\it Cor}({\EuScript T}_v)^{+}$  of pairs of a corolla of  ${\EuScript T}_v$  and one of its free variables,  by the following formal rules:\begin{center} \mybox{\begin{tabular}{ccc}
$\seq{}{(S_v(a),\sigma(v))\in {\it Cor}({\EuScript T}_{v})^{+}}$ &\enspace &
$\seq{(b,u)\in {\it Cor}({\EuScript T}_{v})^{+} \quad w\in FV(b)\backslash (FV({\EuScript T})\cup\{u\})}
{(S_w(b),\sigma(x))\in {\it Cor}({\EuScript T}_v)^{+}}$\end{tabular}}\end{center}
This formal system has the following properties.
\begin{rem}\label{pairs}  
  Each  element $(S_w(b),\sigma(w))\in {\it Cor}({\EuScript T}_v)^{+}$  is such that $S_w(b)$ is adjacent to $b$ in ${\EuScript T}$. For each $(b,u)\in {\it Cor}({\EuScript T}_v)^{+}$, we have  $b\neq a$.
\end{rem}
We obtain the set of corollas of ${\EuScript T}_v$ by erasing from the elements of  ${\it Cor}({\EuScript T}_v)^{+}$ the data about the distinguished free variables, i.e. we define\vspace{-0.1cm} $${\it Cor}({\EuScript T}_v)=\{b\,|\,(b,u)\in {\it Cor}({\EuScript T}_v)^{+} \mbox{ for some } u\in FV(b)\}.\vspace{-0.1cm}$$ The involution $\sigma_{{\EuScript T}_v}$ of ${\EuScript T}_v$ is defined as \vspace{-0.1cm}$${\sigma}_{{\EuScript T}_v}(z)=\left\{ 
	\begin{array}{ll}
		\sigma(z) & \mbox{if }  z\in  \bigl(\bigcup_{b\in Cor({\EuScript T}_v)}FV(b)\bigr)\backslash \sigma(v)\\[0.1cm]
		z & \mbox{if }  z=\sigma(v)\,.\\[0.1cm]
	\end{array}
\right.\vspace{-0.1cm}$$
\indent We shall denote the algorithm with ${\EuScript P}$, and the result ${\EuScript P}({\EuScript T}, a,v)$ of instatiating ${\EuScript P}$ on a tree ${\EuScript T}$, a corolla $a\in{\it Cor}({\EuScript T})$, and a variable $v\in FV(a)\backslash FV({\EuScript T})$ will often be denoted  as ${\EuScript T}_v$, as we have just done above. The following claim guarantees that ${\EuScript P}$ is correct.
\begin{lem}\label{ppp} For   un unrooted tree ${\EuScript T}$,   $a\in {\it Cor}({\EuScript T})$ and   $v\in FV(a)\backslash FV({\EuScript T})$,
${\EuScript T}_v$ is a proper subtree of ${\EuScript T}$.
\end{lem}
\begin{proof} 
By the construction, we have that ${\it Cor}({\EuScript T}_v)\subseteq{\it Cor}({\EuScript T})$ and that ${\EuScript T}_v$ is connected. By Remark \ref{pairs}, it follows that ${\it Cor}({\EuScript T}_v)$ is a proper subset of ${\it Cor}({\EuScript T})$. Finally, since $\sigma_{{\EuScript T}_v}=\sigma$ on $V({\EuScript T}_v)\backslash FV({\EuScript T}_v)$,  we can  conclude that ${\EuScript T}_v$ is  indeed a subtree of ${\EuScript T}$.
\end{proof}
\begin{cor}\label{decomposition}
For  an unrooted tree ${\cal T}$ and  $a\in{\it Cor}({\cal T})$,    the set of unrooted trees ${\EuScript P}({\cal T},a)$, defined by $${\EuScript P}({\cal T},a)=\{\{a(y_1,\dots,y_n);{\it id}_{Y}\}\}\cup \{{\cal T}_v\,|\, v\in Y\backslash FV({\cal T})\},$$ where $Y=\{y_1,\dots,y_n\}$, is a decomposition of ${\cal T}$.
\end{cor}
\begin{proof} The proof goes by induction on the cardinality  of $Y\backslash FV({\cal T})$. 
\end{proof}
\begin{lem}\label{decomp} L Let ${\cal T}$ be an unrooted tree and let $a\in{\it Cor}({\cal T})$. Suppose that $FV({\cal T})=X$ and  ${\it FV}(a)=Y$, where $Y=\{y_1,\dots,y_n\}$. Let $I=\{i_1,\dots i_k\}=\{i\in\{1,\dots,n\}\,|\, y_i\in FV(a)\backslash X\}$.
Then, if ${\EuScript P}({\cal T},a)=\{\{a(y_1,\dots,y_n);{\it id}_{Y}\}\}\cup \{{\cal T}_{y_i}\,|\, i\in I\}$, we have that $$[{\cal T}]_{\alpha}=(([\{a(y_1,\dots,y_n);{\it id}_Y\}]_{\alpha}\,{_{y_{i_1}}\!\!\bullet_{\sigma(y_{i_1})}} \,[{\cal T}_{y_{i_1}}]_{\alpha})\cdots) \,{_{y_{i_k}}\!\!\bullet_{\sigma(y_{i_k})}} \,[{\cal T}_{y_{i_k}}]_{\alpha}\,.$$
\end{lem}
\begin{proof} By induction on the size of ${\EuScript T}$. The claim holds trivially if $a$ is the only corolla of ${\EuScript T}$.  \\
\indent Suppose that ${\EuScript T}$ has $k$ corollas, $k\geq 2$, and that the claim holds for all proper subtrees of ${\EuScript T}$ that contain the corolla $a$.  Since there exists at least one corolla other than $a$ in ${\EuScript T}$,  there exists $1\leq j\leq n$ such that $y_j\in FV(a)\backslash X$. Let ${\EuScript T}'$ be the unrooted tree whose set of corollas is $${\it Cor}({\EuScript T}')=\{a(y_1,\dots,y_n)\}\cup \{{\it Cor}({\EuScript T}_{y_i})\,|\, i\in I\backslash\{j\}\}$$ and whose involution $\sigma'$ is defined as $${\sigma}'(y)=\left\{ 
	\begin{array}{ll}
		\sigma(y) & \mbox{if }  y\in  FV(a)\backslash \{y_j\}\cup\bigcup_{b\in{\it Cor}({\EuScript T}')\backslash \{a\}}FV(b)  \\[0.1cm]
		y & \mbox{if }  y=x_j \,.\\[0.1cm]
	\end{array}
\right.$$
Clearly, ${\EuScript T}'$ is a proper subtree of ${\EuScript T}$, and, by the induction hypothesis, we have $$[{\cal T}']_{\alpha}=(([\{a(y_1,\dots,y_n);{\it id}_Y\}]_{\alpha}\,{_{y_{i_1}}\bullet_{\sigma(y_{i_1})}} \,[{\cal T}_{y_{i_1}}]_{\alpha})\cdots) \,{_{y_{i_k}}\bullet_{\sigma(y_{i_k})}} \,[{\cal T}_{y_{i_k}}]_{\alpha},$$ where $i_1,\dots,i_k\in I\backslash\{j\}$. The claim holds since $[{\EuScript T}]_{\alpha}=[{\EuScript T}']_{\alpha}\,{_{y_{j}}\bullet_{\sigma(y_{j})}} \, [{\EuScript T}_{y_j}]_{\alpha}$.
\end{proof}
\begin{lem}\label{kuku}
 If an unrooted tree ${\EuScript T}$ has at least two corollas, then there exists $c\in{\it Cor}({\EuScript T})$, such that $ FV(c)\backslash FV({\EuScript T})$ is a singleton.
\end{lem}
\begin{proof}
Suppose that $FV({\EuScript T})=X$ and let $\sigma$ be the involution of ${\EuScript T}$. We proceed by induction on the number $n$ of corollas of ${\EuScript T}$.
 For the base case, suppose that  ${\it Cor}({\EuScript T})=\{a,b\}$. Then  there exist  $x\in FV(a)$ and $y\in FV(b)$  such that $\sigma(x)=y$, while all other variables of ${\EuScript T}$ are fixpoints of $\sigma$. Hence,   $FV(a)\backslash FV({\EuScript T})=\{x\}$ and $FV(b)\backslash FV({\EuScript T})=\{y\}$, i.e. $a$ and $b$ both satisfy the claim.\\
\indent Assume now that ${\EuScript T}$ has $n$ corollas, where $n> 2$. Let $a\in{\it Cor}({\EuScript T})$, ${\it FV}(a)=Y$, be such that there exists $v\in Y\backslash X$. If $v$ is the unique such variable we are done. If not, let $\{C;{\it id}_Y\}\cup \{{\EuScript T}_u\,|\, u\in Y\backslash X\}$ be the decomposition of ${\EuScript T}$ obtained by applying   ${\EuScript P}$ on $a$.   Then, if ${\it Cor}({\EuScript T}_v)=\{S_v(a)\}$,  by the definition of ${\EuScript P}$, we know that   $FV(S_v(a))\backslash X=\{\sigma(v)\}$. Therefore, since ${\it Cor}({\EuScript T}_v)\subseteq {\it Cor}({\EuScript T})$, $S_v(a)$ is a corolla that satisfies the claim. On the other hand, if ${\EuScript T}_v$ contains more than one corolla, by the induction hypothesis on ${\EuScript T}_v$, we get  $b\in {\it Cor}({\EuScript T}_v)$ such that $FV(b)\backslash FV({\EuScript T}_v)=\{w\}$. Since $FV(b)\backslash X\subseteq FV(b)\backslash FV({\EuScript T}_v)$, we know that either $FV(b)\backslash X=\{w\}$, or $FV(b)\backslash X=\emptyset$. The latter is impossible  because   $b$ would be the only corolla of ${\EuScript T}$.
\end{proof}
Let ${\EuScript T}$ and $c$ be as in the previous lemma, and let $FV(c)\backslash FV({\EuScript T})=\{v\}$. We shall denote with ${\EuScript T}_{/c}$ the unrooted tree such that ${\it Cor}({\EuScript T}_{/c})={\it Cor}({\EuScript T})\backslash\{c\}$ and whose involution $\sigma_{/c}$ agrees with the involution  $\sigma$ of ${\EuScript T}$ everywhere, except on $\sigma(v)$, which is a fixpoint of $\sigma_{/c}$. Lemma \ref{kuku} guarantees that ${\EuScript T}_{/c}$ is well-defined.\\[0.1cm]
\indent We now establish a   non-inductive characterisation of the output of the algorithm ${\EuScript P}$. 
\begin{lem}\label{alg2}
Let ${\EuScript T}$ be an unrooted tree with involution $\sigma$ and let $a\in {\it Cor}({\EuScript T})$ and $v\in \linebreak FV(a)\backslash FV({\EuScript T})$. The following properties are equivalent for a subtree ${\EuScript T}'$ of ${\EuScript T}$:\\[-0.6cm]
\begin{enumerate}
\item ${\EuScript T}'={\EuScript P}({\EuScript T},a,v)$,\\[-0.65cm]
\item $\sigma(v)\in FV({\EuScript T}')$ and $FV({\EuScript T}')\backslash \{\sigma(v)\}\subseteq FV({\EuScript T})$.
\end{enumerate}
\end{lem}
\begin{proof}
That (1) implies (2) is clear.\\
\indent We prove that $(2)$ implies $(1)$ by induction on the number $n$ of corollas of ${\EuScript T}'$. If $n=1$, then, since $\sigma(v)\in FV({\EuScript T}')$, $S_{v}(a)$ is the only corolla of ${\EuScript T}'$ and the conclusion follows since, by the assumption, $FV({\EuScript T}')\backslash\{\sigma(v)\}=FV(S_{v}(a))\backslash\{\sigma(v)\}\subseteq X$, i.e. $FV(S_{v}(a))\backslash\{X\cup\{\sigma(v)\}\}=\emptyset$. 
\indent Suppose that   $n\geq 2$, and let, by Lemma \ref{kuku}, $c\in {\it Cor}({\EuScript T}')$ be such that $FV(c)\backslash FV({\EuScript T}')=\{u\}$.  If $c=S_{v}(a)$, then it follows easily that ${\EuScript T}'={\EuScript T}_v$. If not, by applying the induction hypothesis on ${\EuScript T}'_{/c}$, we get that  ${\EuScript T}'_{/c}={\EuScript P}({\EuScript T}_{/c},a,v)$. Observe that $(S_u(c),w)\in{\it Cor}({\EuScript T}'_{/c})^{+}$, for some $w\in FV(S_u(c))$  different from $\sigma(u)$. By instantiating ${\EuScript P}$ on $(S_u(c),w)$ and $\sigma(u)$, we get the pair $(c,u)$, and the claim follows since $FV(c)\backslash (FV({\EuScript T})\cup\{u\})=\emptyset$ (i.e. the algorithm stops) and since ${\EuScript T}'_{/c}$ and the single-corolla unrooted tree determined by $c$ make a decomposition of ${\EuScript T}'$.
\end{proof}
\indent For the following two lemmas, recall the definition of the simultaneous composition \eqref{simultaneous} for entries-only cyclic operads. We shall instantiate it on the cyclic operad of classes of unrooted trees, described in \S \ref{free-forgetiful}.   
\begin{lem}\label{jbt}
Let\label{uf} $a\in\underline{\EuScript C}(X)$, where $X=\{x_1,\dots,x_n\}$, and let, for all $x_i\in X$, $\gamma:x_i\mapsto([{\cal T}_{x_i}]_{\alpha},\overline{x_i})$ be an assignment for which the simultaneous composition $[\{a(x_1,\dots,x_n);{\it id}_X\}]_{\alpha}(\gamma)$ is well-defined. Then the following properties hold.
\begin{itemize}
\item[a)] The $\alpha$-equivalence class $[\{a(x_1,\dots,x_n);{\it id}_X\}]_{\alpha}(\gamma)$ admits a representative ${\EuScript T}$, such that $a\in {\it Cor}({\EuScript T})$.
\item[b)] If ${\EuScript T}$ is a representative of $[\{a(x_1,\dots,x_n);{\it id}_X\}]_{\alpha}(\gamma)$, such that $a\in{\it Cor}({\EuScript T})$, and if $\sigma$ is the involution of ${\EuScript T}$, then each class $[{\EuScript T}_{x_i}]_{\alpha}$ admits the unrooted tree ${\EuScript P}({\EuScript T},a,x_i)^{\rho_i}$, where $\rho_i$ renames $\sigma(x_i)$ to $\overline{x_i}$, as a representative. 
\end{itemize}
\end{lem}
\begin{proof} Observe that there are two stages of renaming involved in forming the simultaneous composition $[\{a(x_1,\dots,x_n);{\it id}_X\}]_{\alpha}(\gamma)$. By   \eqref{simultaneous}, we first rename the free variables of the corolla $a$, obtaining in this way the composition $$(\cdots([\{a^{\sigma}(x'_1,\dots,x'_n);{\it id}_{X'}\}]_{\alpha} \, {_{{x'}\!_1}}\! \!\bullet_{\,{\overline {x_1}}}\, [{\EuScript T}_{x_1}]_{\alpha})\cdots) \, {_{{x'}\!_n}}\! \!\bullet_{\,{\overline {x_n}}}\, [{\EuScript T}_{x_n}]_{\alpha},  $$ where $X'=\{x'_1,\dots, x'_n\}$ and $\sigma:X'\rightarrow X$ is derfined by $\sigma(x'_i)=x_i$,  which is then ``calculated'' by the definition of $_x\bullet_y$ from \S \ref{free-forgetiful}. This calculation involves the renaming of variables of all the trees from the above composition, in such a way that the resulting trees have mutually disjoint sets of variables, i.e. it goes though the simultaneous composition
$$(\cdots([\{a^{\sigma\circ\tau}(y_1,\dots,y_n);{\it id}_{Y}\}]_{\alpha} \, {_{{y'}\!_1}}\! \!\bullet_{\,{\overline {y_1}}}\, [{\EuScript T}^{\tau_1\cup {\it id}_{FV({\EuScript T}_{x_i})\backslash\{\overline{x_i}\}}}_{x_1}]_{\alpha})\cdots) \, {_{{y'}\!_n}}\! \!\bullet_{\,{\overline {y_n}}}\, [{\EuScript T}^{\tau_n\cup{\it id}_{FV({\EuScript T}_{x_n})\backslash\{\overline{x_n}\}}}_{x_n}]_{\alpha},  $$ 
where $Y=\{y_1,\dots,y_n\}$, $\tau:Y\rightarrow X'$ is defined by $\tau(y_i)=x'_i$ and each $\tau_i: V_i\rightarrow (V({{\EuScript T}_{x_i}})\backslash {\it FV}({\EuScript T}_{x_i}))\linebreak\cup \{\overline{x_i}\}$ is such that $\tau_i(\overline{y_i})=\overline{x_i}$. The resulting class now has as a representative the tree ${\EuScript T}'$, such that $${\it Cor}({\EuScript T}')=\{a^{\sigma\circ\tau}(y_1,\dots,y_n)\}\cup \bigcup_{1\leq i\leq n} {\it Cor}({\EuScript T}^{\tau_i}_{x_i})$$ and whose involution $\sigma'$ is defined in the obvious way.\\[0.1cm]
\indent The first claim holds, since, thanks to the equivariance axiom {\tt{(EQ)}} for $_x\bullet_y$, we can turn  ${\EuScript T}'$  into an unrooted tree ${\EuScript T}$ that has $a$ as a corolla, by ``undoing'' the renaming $\sigma\circ\tau$. Clearly, if some variable $x_i$ appears in ${\EuScript T}'$,   but   did not originally come from the corolla $a$, this variable has to be renamed too, in order to ensure that all the variables of ${\EuScript T}$ are distinct. Therefore, $${\it Cor}({\cal T})=\{a(x_1,\dots,x_n)\}\cup \bigcup_{1\leq i\leq n} {\it Cor}(({\EuScript T}^{\tau_i}_{x_i})^{\kappa_i\cup {\it id}_{FV({\EuScript T}_{x_i})\backslash\{\overline{x_i}\}}}),$$ where  $\kappa_i : U_i\cup \{\overline{z_i}\} \rightarrow V_i\cup\{\overline{y_i}\}$ is such that $\kappa_i(\overline{z_i})=\overline{y_i}$ and the distinctness requirement for the variables of ${\EuScript T}$ is satisfied. The involution $\sigma$ of ${\cal T}$ is defined from $\sigma'$ in the obvious way. \\[0.1cm]
\indent For the second claim, fix an $i\in\{1,\dots,n\}$. Observe that  we have that $$({\EuScript T}_{x_i}^{\tau_i})^{\nu_i}_{x_i}=_{\alpha}{\EuScript T}_{x_i},$$ where $\nu_i$ renames $\overline{y_i}$ to $\overline{x_i}$. Also, we have that $${\EuScript T}^{\tau_i}_{x_i}=_{\alpha} (({\EuScript T}^{\tau_i}_{x_i})^{\kappa_i\cup {\it id}_{FV({\EuScript T}_{x_i})\backslash\{\overline{x_i}\}}})^{\pi_i},$$ where $\pi_i$ renames $\overline{z_i}$ to $\overline{y_i}$. Therefore, $$((({\EuScript T}^{\tau_i}_{x_i})^{\kappa_i\cup {\it id}_{FV({\EuScript T}_{x_i})\backslash\{\overline{x_i}\}}})^{\pi_i})^{\nu_i}=_{\alpha}{\EuScript T}_{x_i},$$ i.e. each class $[{\EuScript T}_{x_i}]_{\alpha}$ admits as a representative $(({\EuScript T}^{\tau_i}_{x_i})^{\kappa_i\cup {\it id}_{FV({\EuScript T}_{x_i})\backslash\{\overline{x_i}\}}})^{\rho_i}$, where $\rho_i$ renames $\overline{z_i}=\sigma(x_i)$ to $\overline{x_i}$. Observe that $({\EuScript T}^{\tau_i}_{x_i})^{\kappa_i\cup {\it id}_{FV({\EuScript T}_{x_i})\backslash\{\overline{x_i}\}}}$ is a subtree of ${\cal T}$. That we indeed have that $$({\EuScript T}^{\tau_i}_{x_i})^{\kappa_i\cup {\it id}_{FV({\EuScript T}_{x_i})\backslash\{\overline{x_i}\}}}={\EuScript P}({\EuScript T},a,x_i)$$ is clear by  considering the non-inductive criterion from Lemma  \ref{alg2}. 
\end{proof}
\begin{lem}
Let\label{uf} $a\in\underline{\EuScript C}(X)$, where $X=\{x_1,\dots,x_n\}$, and let, for all $x_i\in X$, $\gamma:x_i\mapsto([{\cal T}_{x_i}]_{\alpha},\overline{x_i})$ and $\tau:x_i\mapsto ([{\cal T}'_{x_i}]_{\alpha},\tilde{x_i})$ be assignments for which the simultaneous compositions  $$[\{a(x_1,\dots,x_n);{\it id}_X\}]_{\alpha}(\gamma) \quad \mbox{ and }\quad [\{a(x_1,\dots,x_n);{\it id}_X\}]_{\alpha}(\tau)$$ are well-defined. Then, if $[\{a(x_1,\dots,x_n);{\it id}_X\}]_{\alpha}(\gamma)=[\{a(x_1,\dots,x_n);{\it id}_X\}]_{\alpha}(\tau)$, we have that $[{\cal T}_{x_i}]^{\kappa}_{\alpha}=[{\cal T}'_{x_i}]_{\alpha}$ for all $x_i\in X,$ where $\kappa$ renames $\overline{x_i}$ to ${\tilde{x_i}}$.
\end{lem}
\begin{proof}
By Lemma \ref{jbt}(a), for  $$[\{a(x_1,\dots,x_n);{\it id}_X\}]_{\alpha}(\gamma)=[\{a(x_1,\dots,x_n);{\it id}_X\}]_{\alpha}(\tau)=[{\cal T}]_{\alpha},$$ we can assume that the representative ${\cal T}$ is such that it has $a\in {\it Cor}({\EuScript T})$.  Let $\sigma$ be the involution of ${\EuScript T}$. By applying twice Lemma \ref{jbt}(b), we get that $$[{\EuScript T}_{x_i}]^{\kappa}_{\alpha}=[{\EuScript P}({\EuScript T},a,x_i)^{\rho_i}]^{\kappa}_{\alpha}=[{\cal T}'_{x_i}]_{\alpha},$$ 
where $\rho_i$ renames $\sigma(x_i)$ to $\overline{x_i}$,  which proves the claim.  
\end{proof}
\subsubsection{The equivalence relation $='$  on $\mu\tt{Comm}^{\it nf}_{{\underline{\EuScript C}}}$}
 Let $a\in{\underline{\EuScript C}}(X)$ and let  $\sigma:x\mapsto t_x$ be an association of terms to variables from  $X$, such that the command $\underline{a}\{\sigma\}$ is well-typed. The equivalence relation $='$ is the smallest equivalence relation  generated by  equalities\vspace{-0.1cm} \begin{center}\mybox{ 
 $\underline{a}\{\sigma\}='c[\mu x.\underline{a}\{\sigma[x/x]\}/y] $ }\end{center} where $\sigma(x)=\mu y.c$ and $\sigma[x/x]$ denotes the same association as $\sigma$, except for $x$, to which it associates $x$ itself. We, moreover, assume that $='$ is congruent  with respect to {\small{\tt(MU3)}},  {\small{\tt(MU4)}} and substitution.
\begin{rem}
Observe that, if  $\underline{a}\{\sigma\}='c[\mu x.\underline{a}\{\sigma[x/x]\}/y]$, and if $\underline{a}\{\sigma\}$  is a normal form, then this is also true for the command  $c[\mu x.\underline{a}\{\sigma[x/x]\}/y]$. Therefore, $='$ is well-defined on $\mu\tt{Comm}^{\it nf}_{{\underline{\EuScript C}}}$.
\end{rem}
\indent The intuition behind these equalities  is again about equating commands that reflect two ways to build the same unrooted tree. 
\begin{example}
Consider the  unrooted tree ${\EuScript T}=\{a(x_1,x_2,x_3,x_4), b(y_1,y_2,y_3,y_4,y_5);\sigma\}$, where $\sigma=(x_1\, y_2)$, represented pictorially as
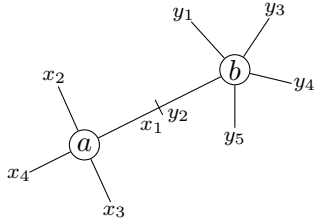
\vspace{-0.1cm}
 \begin{center}
\begin{tikzpicture}[level distance=10mm,
  every node/.style={minimum size=5mm, inner sep=0.5mm},
  level 1/.style={every child/.style={edge from parent/.style={draw,solid}},nodes={fill=red!45},sibling distance=7mm},
  level 2/.style={every child/.style={edge from parent/.style={draw,solid}},nodes={fill=red!30}},
  level 3/.style={every child/.style={edge from parent/.style={draw,solid}},nodes={fill=red!30}, sibling distance=4mm},
  level 4/.style={every child/.style={edge from parent/.style={draw,dashed}},sibling distance=3mm},
  normal/.style={circle,draw,solid},
  jov/.style={circle,draw=none, fill=white}
  acc/.style={circle,thick,solid,draw=green!50,fill=green!2},
  rej/.style={circle,thick,solid,draw=red!50,fill=red!20},
  semithick
   empt/.style={circle,draw=none,edge from parent/.style={draw,solid} }]

\node (f) [fill=none,draw=black,minimum size=4mm,inner sep=0.1mm,circle, normal]  at (-1,0) {$a$};
\node (g) [circle, normal, minimum size=4mm,fill=none,inner sep=0.1mm]  at (1,1) {$b$};
\node (2) [label={[xshift=-0.02cm, yshift=-0.28cm]\scriptsize{$x_2$}},circle, normal, minimum size=1mm,fill=none,draw=none]  at (-1.38,0.85) {};
\node (4) [label={[xshift=-0.07cm, yshift=-0.34
cm]\scriptsize{$x_4$}},circle, normal, minimum size=1mm,fill=none,draw=none]  at (-1.8,-0.4) {};
\node (6) [label={[xshift=0.02cm, yshift=-0.39cm]\scriptsize{$x_3$}},circle, normal, minimum size=1mm,fill=none,draw=none]  at (-0.62,-0.83) {};
\node (a) [label={[xshift=0cm, yshift=-0.42cm]{\scriptsize{$y_1$}}},circle, normal, minimum size=1mm,fill=none,draw=none,inner sep=1mm]  at (0.32,1.75) {};
\node (b) [label={[xshift=0.04cm, yshift=-0.28cm]\scriptsize{$y_3$}},circle, normal, minimum size=1mm,fill=none,draw=none]  at (1.5,1.75) {};
\node (c) [label={[xshift=0.1cm, yshift=-0.33cm]\scriptsize{$y_4$}},circle, normal, minimum size=1mm,fill=none,draw=none]  at (1.83,0.8) {};
\node (d) [label={[xshift=0cm, yshift=-0.4
cm]\scriptsize{$y_5$}},circle, normal, minimum size=1mm,fill=none,draw=none]  at (1,0.15) {};
\node (x) [label={[xshift=-0.1cm, yshift=-0.6cm]\scriptsize{$x_1$}},circle, normal, minimum size=1mm,fill=none,draw=none]  at (0,0.5) {};
\node (y) [label={[xshift=0.25cm, yshift=-0.44cm]\scriptsize{$y_2$}},circle, normal, minimum size=1mm,fill=none,draw=none]  at (0,0.5) {};
\draw (-0.06,0.6)--(0.04,0.4);
\draw (f)--(g);
\draw (f)--(2) ;\draw (f)--(4);\draw (f)--(6);
\draw (g)--(a);\draw (g)--(b);\draw (g)--(c);\draw (g)--(d);

\end{tikzpicture}
\end{center}\vspace{-0.3cm}
The commands equated by $='$   reflect  the two possible ways to build ${\EuScript T}$  by means of simultaneous grafting: we could pick either the corolla $a$ and graft to it the surrounding trees, or we can do the same by choosing first the corolla $b$. In the language of the $\mu$-syntax, the two constructions are  described by the left hand side and the right hand side of the equality \vspace{-0.1cm}$$\underline{a}\{\mu y_2.\underline{b}\{y_1,y_2,y_3,y_4,y_5\},a,b,c\}='\underline{b}\{y_1,\mu x_1.\underline{a}\{x_1,x_2,x_3,x_4\},y_3,y_4,y_5\},\vspace{-0.1cm}$$  respectively. Observe that, from the tree-wise perspective, $='$ enables us to ``move between two adjacent corollas'', i.e. it enables us to ``move along a path in a tree''. As we shall see, this feature will be crucial for in the proof of injectivity of Theorem \ref{equiv}. \hfill$\square$
\end{example}

\indent The proof of the following lemma shows that $='$ is   a ``macro" derivable from $=_{\mu}$.
\begin{lem}\label{bla} For any $c_1,c_2\in$ $\mu${\em\texttt{Comm}}$_{{\underline{\EuScript C}}}^{\it nf}$, if $c_1='c_2$, then $c_1=_{\mu}c_2$.
\end{lem}
\begin{proof}  If $\underline{a}\{\sigma\}='c[\mu x.\underline{a}\{\sigma[x/x]\}/y]$, then $\sigma(x)=\mu y.c$, which justifies the following sequence of  equalities: \vspace{-0.1cm}$$\underline{a}\{\sigma\}=_{\mu}\langle \mu x.\underline{a}\{\sigma[x/x]\} \,|\, \mu y.c\rangle=_{\mu}\langle  \mu y.c \,|\, \mu x.\underline{a}\{\sigma[x/x]\}\rangle=_{\mu} c[\mu x.\underline{a}\{\sigma[x/x]\}/y] .\vspace{-0.65cm}$$
\end{proof}

The equality $='$ (denoted differently) appears in the work \cite{lam} of Lamarche, where it is called Adjunction and used in the context of the so-called {\em reversible terms}. Although the Adjunction rule materialises the same intuition about unrooted trees, there, unlike in our work, it is not derived from a more primitive notion of equality.
\subsubsection{The proof of Theorem \ref{equiv}}
The correspondence $\Phi_X:\mu${\texttt{Comm}}$_{{\underline{\EuScript C}}}(X)_{/_{=_{\mu}}}\rightarrow  {\tt{T}}_{{\underline{\EuScript C}}}(X)$ is canonically induced from the correspondence
 $$\underline{\Phi}:\mu \tt{Exp}_{{\underline{\EuScript C}}}\rightarrow \tt{T}_{{\underline{\EuScript C}}},$$ defined as the composition of the translation function $[[\rule{.4em}{.4pt}]]:$ $\mu$\texttt{Exp}$_{{\underline{\EuScript C}}}\rightarrow$ \texttt{cTerm}$_{{\underline{\EuScript C}}}$ (see \S \ref{muint}) with the interpretation function $[\rule{.4em}{.4pt}]_{\mbox{\texttt{T}}_{{\underline{\EuScript C}}}} : $ \texttt{cTerm}$_{{\underline{\EuScript C}}}\rightarrow\tt{T}_{{\underline{\EuScript C}}}$ (that arises by considering the free cyclic operad defined in \S \ref{free-forgetiful} through Definition \ref{d2}). We  show explicitly the definition of  $\underline{\Phi}$ below, wherein the assignment of an $\alpha$-equivalence class of unrooted trees to a term $t\in \mu\tt{Term}_{{\underline{\EuScript C}}}$ will be indexed by a fresh variable $y$  involved in the corresponding interpretation $[[t]]_y$:\\[-0.5cm]
\begin{itemize}
\item $\underline{\Phi}_y(x)=[\{(x,y);{\it id}_{\{x,y\}}\}]_{\alpha}$,\\[-0.55cm]
\item if, for each $x_i\in \{x_1,\dots,x_n\}$, $\underline{\Phi}\,_{\overline{x}_i}(t_{x_i})=[{\cal T}_{x_i}]_{\alpha}$, then $$\underline{\Phi}(\underline{a}\{t_{x_1},\dots, t_{x_n} \})=[\{a(x_1,\dots,x_n);{\it id}_X\}]_{\alpha}(\varphi), $$
where  $\varphi : x_i\mapsto ([{\cal T}_{x_i}]_{\alpha}, \overline{x}_i)$ (see \eqref{simultaneous}),\\[-0.55cm]
\item  $\underline{\Phi}_y(\mu x.c)=(\Phi(c))^{\kappa}, $ where $\kappa$ renames $x$ to $y$, and\\[-0.55cm]
\item if $\underline{\Phi}_x(s)=[{\EuScript T}_s]_{\alpha}$ and $\underline{\Phi}_y(t)=[{\EuScript T}_t]_{\alpha}$, then $\underline{\Phi}(\langle s \,| \,t\rangle)=[{\EuScript T}_s]_{\alpha}\,{_{x}\bullet_{y}} \,[{\EuScript T}_t]_{\alpha}.$
\end{itemize} 

By Theorem \ref{wd},    $\underline{\Phi}$ is well-defined. We prove that it is both injective and surjective.\\[0.1cm]
\indent \textit{Surjectivity.} Suppose given an $\alpha$-equivalence class $[{\EuScript T}]_{\alpha}\in$ {\texttt{T}}$_{{\underline{\EuScript C}}}(X)$.  If ${\EuScript T}=\{(x,y);{\it id}_{\{x,y\}}\}$,  then it is easily seen that $\underline{\Phi}(\langle x\,|\, y\rangle)=[\{(x,y);{\it id}_{\{x,y\}}\}]_{\alpha} $.
 
\indent Suppose now that ${\EuScript T}$ is an ordinary unrooted tree. We proceed by induction on the number $k$ of corollas of $\EuScript T$. Let  $a\in{\it Cor}({\EuScript T})$ be such that ${\it FV}(a)=Y$, where $Y=\{y_1,\dots,y_n\}$.\\
\indent If $a$ is the only corolla of ${\EuScript T}$, then $\underline{\Phi}(\underline{a}\{y_1,\dots ,y_n\})=[\{a(y_1,\dots,y_n);{\it id}_Y\}]_{\alpha}.$ \\
\indent Suppose  that $a$ is not the only corolla of ${\EuScript T}$, i.e. that  $k\geq 2$, and let $\sigma$ be the involution of ${\EuScript T}$. 
Let $I=\{i\in\{1,\dots,n\}\,|\, y_i\in FV(a)\backslash X\}$ and $J=\{1,\dots,n\}\backslash I$. By the induction hypothesis for each ${\EuScript P}({\EuScript T},a,y_i)={\EuScript T}_{x_i}$ (recall from \S \ref{pru} that ${\EuScript P}$ is the ``pruning'' algorithm), for $i\in I$, we get a set\vspace{-0.1cm} $$\{c_{i}\in \mu{\tt{Comm}}_{{\underline{\EuScript C}}}\,|\, i\in I\mbox{ and }\underline{\Phi}(c_{i})=[{\EuScript T}_{y_i}]_{\alpha}\}.\vspace{-0.1cm}$$  We now set for all $i\in I$, $t_{y_i}=\mu \sigma(y_i).c_i$, and for all $j\in J$, $t_{y_j}=y_j$, and we claim that $\Phi(a\{t_y\,|\, y\in Y\})=[{\EuScript T}]_{\alpha}$. We have $\underline{\Phi}(a\{t_{y_k}\,|\, k\in\{1,\dots,n\}\})=a(\varphi)$, where 
$$\varphi : y_k \mapsto \left\{ 
	\begin{array}{ll}
		([{\EuScript T}_{y_i}]_{\alpha}^{\kappa_i},z_i) & \mbox{if }  k=i \mbox{ for some } i\in I\\[0.1cm]
		([\{(y_j,\underline{y_j});id_{\{y_j,\underline{y_j}\}}\}]_{\alpha},\underline{x_j}) & \mbox{if }  k=j \mbox{ for some } j\in J\\[0.1cm]
	\end{array}
\right.$$
 with $[{\EuScript T}_{y_i}]_{\alpha}^{\kappa_i}=\underline{\Phi}_{z_i}(\mu\sigma(y_i).c_i)$ being the class associated to the term $\mu\sigma(y_i).c_i$ with respect to the interpretation under the fresh variable $z_i$. 
 Therefore, if $I=\{i_1,\dots,i_{m_I}\}$ and $J=\{j_1,\dots,j_{m_J}\}$,  by the axiom {\small{\texttt{(U1)}}}, $\underline{\Phi}(a\{t_{y_k}\,|\, k\in\{1,\dots,n\}\})$ is equal to   $$(\cdots([\{a(y_1,\dots,y_n);id_Y\}]_{\alpha}^{\kappa_{j_1}\kappa_{j_2}\cdots\kappa_{j_{m_J}}}\,{_{y_{i_1}}\bullet_{z_{i_1}}} \,[{\EuScript T}_{y_{i_1}}]^{\kappa_{{i_1}}}_{\alpha})\cdots )\,{_{y_{i_{m_I}}}\bullet_{z_{i_{m_I}}}} \,[{\EuScript T}_{y_{i_{m_I}}}]^{\kappa_{{i_{m_I}}}}_{\alpha}$$ 
where each $\kappa_{j_m}$, $1\leq m\leq {m_J}$ is the renaming of $y_{j_k}$ to $y_{j_k}$, i.e. the identity on $Y$, and each $\kappa_{i_m}$, $1\leq m\leq m_{I}$, is the renaming of $z_{i_k}$ to $\sigma(x_{i_k})$.
Finally, by {\small{\texttt{(EQ)}}}, we have  $$\underline{\Phi}(a\{t_{y_k}\,|\, k\in\{1,\dots,n\}\})=(([\{a(y_1,\dots,y_n);{\it id}_Y\}]_{\alpha}\,{_{y_{i_1}}\bullet_{\sigma(y_{i_1})}} \,[{\EuScript T}_{y_{i_1}}]_{\alpha})\cdots) \,{_{y_{i_{m_I}}}\bullet_{\sigma(y_{i_{m_I}})}} \,[{\EuScript T}_{y_{i_{m_I}}}]_{\alpha},\vspace{-0.1cm}$$ and, consequently, by Lemma \ref{decomp}, that $\underline{\Phi}(a\{t_{y_k}\,|\, k\in\{1,\dots,n\}\})=[{\EuScript T}]_{\alpha}$.\\[0.1cm]
\indent \textit{Injectivity.} Notice that, in order to establish the injectivity of $\underline{\Phi}$, it suffices to prove it for commands $c_1,c_2\in \mu$\texttt{Comm}$_{{\underline{\EuScript C}}}^{\it nf}$.   By Lemma \ref{bla}, the injectivity for normal forms follows if we show that,  if $\underline{\Phi}(c_1)=\underline{\Phi}(c_2)$, then $c_1=' c_2$. \\
\indent If $c_1$ and $c_2$ have the same head symbol, we proceed by induction on the structure of $c_1$ and $c_2$.  Suppose that $c_1=\underline{a}\{s_x|x\in X\}=\underline{a}\{\sigma\}$ and $c_2=\underline{a}\{t_x|x\in X\}=\underline{a}\{\sigma'\}$. The assumption $\underline{\Phi}(c_1)=\underline{\Phi}(c_2)$ means that  $$[\{a(x_1,\dots,x_n);{\it id}_X\}]_{\alpha}(\varphi)=[\{a(x_1,\dots,x_n);{\it id}_X\}]_{\alpha}(\psi),$$ where $\varphi:x\mapsto(\underline{\Phi}_{\tilde{x}}(s_x),\tilde{x})$ and $\psi:x\mapsto(\underline{\Phi}_{\overline{x}}(t_x),\overline{x})$, and consequently, by Lemma \ref{uf},  that for all $x\in X$, $\underline{\Phi}_{\tilde{x}}(s_x)^{\kappa}=\underline{\Phi}_{\overline{x}}(t_x)$, where $\kappa$ renames $\tilde{x}$ to $\overline{x}$. The claim holds by the reflexivity of $='$ if all $s_x$ and $t_x$ are variables: if $s_x=u$ and $t_x=v$, then\vspace{-0.1cm} $$[\{(u,\overline{x});{\it id}_{\{u,\overline{x}\}}\}]_{\alpha}=(\underline{\Phi}_{\tilde{x}}(u))^{\kappa}=\underline{\Phi}_{\overline{x}}(v)=[\{(v,\overline{x});{\it id}_{\{v,\overline{x}\}}\}]_{\alpha},\vspace{-0.1cm}$$ and, therefore, it must be the case that $u=v$.

\indent Suppose, therefore, that $s_x=\mu u.c_x$ and $t_x=\mu v.c'_x$. We then have $$[[c^{\tau_1}_x]]=[[c_x]]^{\tau_1}=[[s_x]]_{\tilde{x}}^{\kappa}=[[t_x]]_{\overline{x}}=[[c'_x]]^{\tau_2}=[[c'^{\tau_2}_x]],$$ and, consequently, that $\underline{\Phi}(c^{\tau_1}_x)=\underline{\Phi}(c'^{\tau_2}_x)$, where ${\tau_1}$ renames $u$ to $\overline{x}$ and $\tau_2$ renames $v$ to $\overline{x}$. By the induction hypothesis we now have $c^{\tau_1}_x='c'^{\tau_2}_x$ and, consequently, we get that $$
\begin{array}{rllllll}
\underline{a}\{\sigma\}&='& c_x[\mu x.\underline{a}\{\sigma[x/x]\}/u]
&=& c^{\tau_1}_x[\mu x.\underline{a}\{\sigma[x/x]\}/\overline{x}]\\[0.1cm]
&='&c'^{\tau_2}_x[\mu x.\underline{a}\{\sigma[x/x]\}/\overline{x}]
&=& c'_x[\mu x.\underline{a}\{\sigma[x/x]\}/v]
&='& \underline{a}\{\sigma'\}.  
\end{array}$$
\indent Suppose now that $c_1$ and $c_2$ do not have the same head symbol, i.e. that  $c_1=\underline{a}\{s_x|x\in X\}=\underline{a}\{\sigma_1\}$ and $c_2=\underline{b}\{t_y|y\in Y\}=\underline{b}\{\sigma_2\}$, and let $\underline{\Phi}(c_1)=[{\EuScript T}_{c_1}]_{\alpha}$ and $\underline{\Phi}(c_2)=[{\EuScript T}_{c_2}]_{\alpha}$. Let ${\EuScript T}$ be a representative of $[{\EuScript T}_{c_1}]_{\alpha}=[{\EuScript T}_{c_2}]_{\alpha}$. Observe that two groups of renamings  feature in the transitions from $c_1$ and $c_2$ to ${\EuScript T}$: the first one contains the renamings specified by the definitions  of the simultaneous compositions $\underline{\Phi}(c_1)$ and $\underline{\Phi}(c_2)$, and the second one contains the renamings given by the $\alpha$-equivalence of ${\EuScript T}_{c_1}$ and  ${\EuScript T}$, and ${\EuScript T}_{c_2}$ and ${\EuScript T}$. However,  by {\small{\tt{(MU4)}}}, all the  renamings of parameters and variables of $c_1$ and $c_2$ made in defining  ${\EuScript T}$, can be also performed on $c_1$ and $c_2$ themselves, leading to commands $c'_1=_{\mu}c_1$ and  $c'_2=_{\mu}c_2$, such that $\underline{\Phi}(c'_1)=\underline{\Phi}(c'_2)=[{\cal T}]_{\alpha}$ and such that ${\EuScript T}$   shares the same sets of parameters and variables with both $c'_1$ and $c'_2$. Hence, we can assume that  ${\EuScript T}$ already  shares the same sets of parameters and variables with $c_1$ and $c_2$. This, in particular, means that $a,b\in {\it Cor}({\EuScript T})$.

  Let $x\in X$   be such that $b\in {\it Cor}({\EuScript P}({\EuScript T},a,x))$.  By the construction of ${\EuScript T}$, the parameter $\underline{b}$ appears in $\sigma_1(x)=\mu u.c$.
We define the {\em distance between} $\underline{a}$ {\em and} $\underline{b}$ {\em in} $c_1$ as the natural number $d_{c_1}(a,b)$ determined as follows.\\[-0.5cm] \begin{itemize}
\item If $\underline{b}$ is the head symbol of $c$,  then $d_{c_1}(a,b)=1$.\\[-0.5cm]
\item If $\underline{h}$ is the head symbol of $c$, $h\neq b$, then $d_{c_1}(a,b)=d_c(h,b)+1$.
\end{itemize}
We prove that $c_1='c_2$ by induction on  $d_{c_1}(a,b)$. If $d_{c_1}(a,b)=1$, then, for some $y\in FV(b)$,  we have that $\sigma_1(x)=\mu y.\underline{b}\{\sigma_2[y/y]\}$. Therefore,\vspace{-0.1cm} 
$$\begin{array}{rcl}
\underline{a}\{\sigma_1\}&='&\underline{b}\{\sigma_2[y/y]\}[\mu x.\underline{a}\{\sigma_1[x/x]\}/y]\\[0.1cm]
&=&\underline{b}\{\sigma_2[\mu x.\underline{a}\{\sigma_1[x/x]\}/y]\}\\[0.1cm]
&='&\underline{b}\{\sigma_2\}.
\end{array}$$

\indent If $d_{c_1}(a,b)\geq 2$, then, since $d_{c_1}(a,h)=1$ (where $h$ is as above), we have that $c_1='\linebreak c[\mu x.\underline{a}\{\sigma_1[x/x]\}/u]$. On the other hand, by the induction hypothesis for $d_c(h,b)<n$, we have that $c_2='c[\mu x.\underline{a}\{\sigma_1[x/x]\}/u]$, and
 the conclusion follows by the transitivity of $='$. The iterative application of the equality $='$,  implicit in the induction argument, which reduces the distance between $\underline{a}$ and $\underline{b}$,   can be illustrated as follows
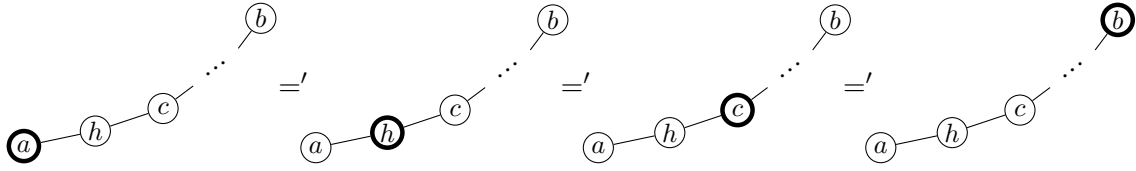
\begin{center}
\begin{tabular}{c @{\hspace{-0.9ex}} c @{\hspace{-0.9ex}} c @{\hspace{-0.9ex}} c}
\begin{tikzpicture}
 \node (f) [circle,fill=none,draw=black,line width=0.6mm,minimum size=4mm,inner sep=0.1mm]  at (-1.15,0) {\small $a$};
\node (g) [circle,fill=none,draw=black,minimum size=4mm,inner sep=0.1mm]  at (2,1.7) {\small $b$};
\node (h) [circle,fill=none,draw=black,minimum size=4mm,inner sep=0.1mm]  at (-0.2,0.2) {\small $h$};
\node (h1) [circle,fill=none,draw=black,minimum size=4mm,inner sep=0.1mm]  at (0.7,0.5) {\small $c$};
\node (d1) [circle,fill=none,draw=none,minimum size=4mm,inner sep=0.1mm]  at (1.3,0.95) { $\cdot$};
\node (d2) [circle,fill=none,draw=none,minimum size=4mm,inner sep=0.1mm]  at (1.4,1.025) { $\cdot$};
\node (d3) [circle,fill=none,draw=none,minimum size=4mm,inner sep=0.1mm]  at (1.5,1.1) { $\cdot$};
\draw (f)--(h);
\draw (h)--(h1);
\draw (h1)--(1.1,0.8);
\draw (g)--(1.7,1.3);
\end{tikzpicture} \ & \begin{tikzpicture}
 \node (f) [circle,fill=none,draw=black,minimum size=4mm,inner sep=0.1mm]  at (-1.15,0) {\small $a$};
 \node (fd) [circle,fill=none,draw=none,minimum size=4mm,inner sep=0mm]  at (-1.45,0.85) { $='$};
\node (g) [circle,fill=none,draw=black,minimum size=4mm,inner sep=0.1mm]  at (2,1.7) {\small $b$};
\node (h) [circle,fill=none,draw=black,line width=0.6mm,minimum size=4mm,inner sep=0.1mm]  at (-0.2,0.2) {\small $h$};
\node (h1) [circle,fill=none,draw=black,minimum size=4mm,inner sep=0.1mm]  at (0.7,0.5) {\small $c$};
\node (d1) [circle,fill=none,draw=none,minimum size=4mm,inner sep=0.1mm]  at (1.3,0.95) { $\cdot$};
\node (d2) [circle,fill=none,draw=none,minimum size=4mm,inner sep=0.1mm]  at (1.4,1.025) { $\cdot$};
\node (d3) [circle,fill=none,draw=none,minimum size=4mm,inner sep=0.1mm]  at (1.5,1.1) { $\cdot$};
\draw (f)--(h);
\draw (h)--(h1);
\draw (h1)--(1.1,0.8);
\draw (g)--(1.7,1.3);
\end{tikzpicture}  &  \begin{tikzpicture}
 \node (f) [circle,fill=none,draw=black,minimum size=4mm,inner sep=0.1mm]  at (-1.15,0) {\small $a$};
 \node (fd) [circle,fill=none,draw=none,minimum size=4mm,inner sep=0mm]  at (-1.45,0.85) { $='$};
\node (g) [circle,fill=none,draw=black,minimum size=4mm,inner sep=0.1mm]  at (2,1.7) {\small $b$};
\node (h) [circle,fill=none,draw=black,minimum size=4mm,inner sep=0.1mm]  at (-0.2,0.2) {\small $h$};
\node (h1) [circle,fill=none,draw=black,line width=0.6mm,minimum size=4mm,inner sep=0.1mm]  at (0.7,0.5) {\small $c$};
\node (d1) [circle,fill=none,draw=none,minimum size=4mm,inner sep=0.1mm]  at (1.3,0.95) { $\cdot$};
\node (d2) [circle,fill=none,draw=none,minimum size=4mm,inner sep=0.1mm]  at (1.4,1.025) { $\cdot$};
\node (d3) [circle,fill=none,draw=none,minimum size=4mm,inner sep=0.1mm]  at (1.5,1.1) { $\cdot$};
\draw (f)--(h);
\draw (h)--(h1);
\draw (h1)--(1.1,0.8);
\draw (g)--(1.7,1.3);
\end{tikzpicture}   &  \begin{tikzpicture}
 \node (f) [circle,fill=none,draw=black,minimum size=4mm,inner sep=0.1mm]  at (-1.15,0) {\small $a$};
 \node (fd) [circle,fill=none,draw=none,minimum size=4mm,inner sep=0mm]  at (-1.45,0.85) { $='$};
\node (g) [circle,fill=none,draw=black,line width=0.6mm,minimum size=4mm,inner sep=0.1mm]  at (2,1.7) {\small $b$};
\node (h) [circle,fill=none,draw=black,minimum size=4mm,inner sep=0.1mm]  at (-0.2,0.2) {\small $h$};
\node (h1) [circle,fill=none,draw=black,minimum size=4mm,inner sep=0.1mm]  at (0.7,0.5) {\small $c$};
\node (d1) [circle,fill=none,draw=none,minimum size=4mm,inner sep=0.1mm]  at (1.3,0.95) { $\cdot$};
\node (d2) [circle,fill=none,draw=none,minimum size=4mm,inner sep=0.1mm]  at (1.4,1.025) { $\cdot$};
\node (d3) [circle,fill=none,draw=none,minimum size=4mm,inner sep=0.1mm]  at (1.5,1.1) { $\cdot$};
\draw (f)--(h);
\draw (h)--(h1);
\draw (h1)--(1.1,0.8);
\draw (g)--(1.7,1.3);
\end{tikzpicture}
\end{tabular}
\end{center}
This completes the proof of Theorem \ref{equiv}.\\[0.15cm]
\indent  Note that we have in fact    {\em two} bijections: $\mu${\texttt{Comm}}$_{{\underline{\EuScript C}}}(X)/_{=_{\mu}}\enspace{\simeq}\enspace\mu$\texttt{Comm}$_{{\underline{\EuScript C}}}^{\it nf}(X)/_{='}\enspace{\simeq}\enspace ${\texttt{T}}$_{{\underline{\EuScript C}}}(X)$,
the first one being induced via normal forms of $\leadsto$:   we have that ${\it nf}(c_1) =' {\it nf}(c_2)$ implies $c_1=_{\mu}c_2$, and conversely, if $c_1=_{\mu}c_2$, then $\Phi({\it nf}(c_1)=\Phi({\it nf}(c_2))$ implies ${\it nf}(c_1) =' {\it nf}(c_2)$.

\section{The equivalence established}

We finally show how the $\mu$-syntax, together with the syntactic formalism of unrooted trees suited to it, allows us to prove   Theorem \ref{th1} in a genuinely constructive, and, thereby, algorithmic way. In both directions, the proof we give elaborates calculations to be made at each step of the transition, relying on the constructions made in the proof of Theorem \ref{equiv}. Let ${\underline{\EuScript C}}:{\bf Bij}^{\it op}\rightarrow {\bf Set}$ be a functor.\\[0.1cm]
\indent Suppose that $({\underline{\EuScript C}},\delta)$ is an ${\EuScript M}$-algebra.  We build a  cyclic operad, as described by Definition \ref{d2},   as follows.

We  distinguish the identities, by setting ${\it id}_{x,y}=\delta_{\{x,y\}}([\{(x,y);{\it id}_{\{x,y\}}\}]_{\alpha})$. 
The definition of the partial composition operation $_x\circ_y$ is derived  by considering restrictions of  $\delta$   to unrooted trees with two corollas:
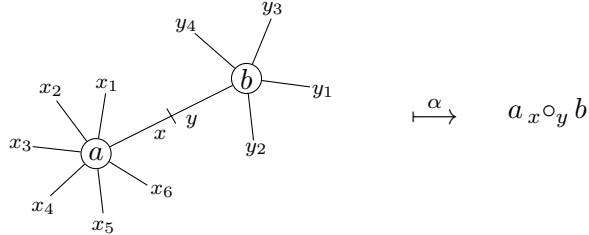
\vspace{-0.1cm}
\begin{center}
\begin{tikzpicture}
\node (f) [circle, inner sep=0.1mm, minimum size=4mm,fill=none,draw=black]  at (-1,0) {$a$};
\node (g) [circle, inner sep=0.1mm, minimum size=4mm,fill=none,draw=black]  at (1,1) {$b$};
\node (1) [label={[xshift=0cm, yshift=-0.4cm]\scriptsize{$x_1$}},circle,inner sep=0mm, minimum size=1mm,fill=none,draw=none,inner sep=1mm]  at (-0.85,0.95) {};
\node (2) [label={[xshift=-0.04cm, yshift=-0.17cm]\scriptsize{$x_2$}},circle,inner sep=0mm,   minimum size=1mm,fill=none,draw=none]  at (-1.56,0.75) {};
\node (3) [label={[xshift=-0.1cm, yshift=-0.26cm]\scriptsize{$x_3$}},circle,inner sep=0mm,  minimum size=1mm,fill=none,draw=none]  at (-1.9,0.1) {};
\node (4) [label={[xshift=-0.05cm, yshift=-0.39
cm]\scriptsize{$x_4$}},circle, inner sep=0mm,  minimum size=1mm,fill=none,draw=none]  at (-1.65,-0.6) {};
\node (5) [label={[xshift=0.0cm, yshift=-0.4cm]\scriptsize{$x_5$}},circle, inner sep=0mm, minimum size=1mm,fill=none,draw=none]  at (-0.9,-0.85) {};
\node (6) [label={[xshift=0.15cm, yshift=-0.32cm]\scriptsize{$x_6$}},circle,inner sep=0mm, minimum size=1mm,fill=none,draw=none]  at (-0.27,-0.45) {};
\node (a) [label={[xshift=0.0cm, yshift=-0.37cm]\scriptsize{$y_4$}},circle,  inner sep=0mm, minimum size=1mm,fill=none,draw=none,inner sep=1mm]  at (0.2,1.7) {};
\node (b) [label={[xshift=0.0cm, yshift=-0.2cm]\scriptsize{$y_3$}},circle,inner sep=0mm,   minimum size=1mm,fill=none,draw=none]  at (1.35,1.85) {};
\node (c) [label={[xshift=0.12cm, yshift=-0.33cm]\scriptsize{$y_1$}},circle,inner sep=0mm,  minimum size=1mm,fill=none,draw=none]  at (1.9,0.88) {};
\node (d) [label={[xshift=0.03cm, yshift=-0.34
cm]\scriptsize{$y_2$}},circle,inner sep=0mm,   minimum size=1mm,fill=none,draw=none]  at (1.1,0.12) {};
\node (x) [label={[xshift=-0.15cm, yshift=-0.5cm]\scriptsize{$x$}},circle,inner sep=0mm ,minimum size=1mm,fill=none,draw=none]  at (0,0.5) {};
\node (y) [label={[xshift=0.27cm, yshift=-0.35cm]\scriptsize{$y$}},circle, inner sep=0mm,  minimum size=1mm,fill=none,draw=none]  at (0,0.5) {};
\node (alpha) [circle, inner sep=0mm, minimum size=2cm,fill=none,draw=none]  at (3.5,0.58) {$\overset{\alpha}{\longmapsto}$};
\node (com) [circle, inner sep=0mm, minimum size=2cm,fill=none,draw=none]  at (5,0.55) {$a \,{_{x}\circ_y }\,b$};
\draw (-0.06,0.6)--(0.06,0.4);
\draw (f)--(g);
\draw (f)--(1);
\draw (f)--(2) ;\draw (f)--(3);\draw (f)--(4);\draw (f)--(5);\draw (f)--(6);
\draw (g)--(a);\draw (g)--(b);\draw (g)--(c);\draw (g)--(d);
\end{tikzpicture}
\vspace{-0.1cm}
 \end{center} 

Formally, for $a\in {\underline{\EuScript C}}(X)$ and $b\in{\underline{\EuScript C}}(Y)$  different then units, the partial composition operation \vspace{-0.1cm}$${{_{x}\circ_{y}}}:{\underline{\EuScript C}}(X)\times {\underline{\EuScript C}}(Y)\rightarrow {\underline{\EuScript C}}(X\backslash\{x\}\cup Y\backslash\{y\})\vspace{-0.1cm}$$ is characterised via $\delta_{X\backslash\{x\}\cup Y\backslash\{y\}}:{\EuScript M}({\underline{\EuScript C}})(X\backslash\{x\}\cup Y\backslash\{y\})\rightarrow {\underline{\EuScript C}}(X\backslash\{x\}\cup Y\backslash\{y\})$ as \vspace{-0.1cm}$$a\, {{_{x}\circ_{y}}} \, b=\delta_{X\backslash\{x\}\cup Y\backslash\{y\}}([\{a(x,\dots);{\it id}_X\}]_{\alpha}\, {{_{x}\bullet_{y}}} \, [\{b(y,\dots);{\it id}_Y\}]_{\alpha}),\vspace{-0.1cm}$$
where ${{_{x}\bullet_{y}}}$ is the operation on (classes of) unrooted trees defined in \S \ref{free-forgetiful}.  If, say, $b={\it id}_{y,z}$, we set   $$a\, {{_{x}\circ_{y}}} \, {\it id}_{\{y,z\}}=\delta_{X\backslash\{x\}\cup \{z\}}([\{a(x,\dots);{\it id}_X\}]_{\alpha}\, {{_{x}\bullet_{y}}} \, [\{(y,z);{\it id}_{\{y,z\}}\}]_{\alpha}).$$

As a structure morphism of ${\EuScript M}$-algebra $({\underline{\EuScript C}},\delta)$, $\delta$ satisfies the coherence conditions given by commutations of the following two diagrams:
\vspace{-0.1cm}
\begin{center}
\begin{tikzpicture}[scale=1.5]
\node (A)  at (0,1) {${\EuScript M}{\EuScript M}({\underline{\EuScript C}})$};
\node (B) at (2.1,1) {${\EuScript M}({\underline{\EuScript C}})$};
\node (C) at (0,-0.2) {${\EuScript M}({\underline{\EuScript C}})$};
\node (D) at (2.1,-0.2) {${\underline{\EuScript C}}$};
\path[->,font=\footnotesize]
(A) edge node[above]{${\EuScript M}\delta$} (B)
(A) edge node[left]{$\mu_{{\underline{\EuScript C}}}$} (C)
(B) edge node[right]{$\delta$} (D)
(C) edge node[above]{$\delta$} (D);
\end{tikzpicture}
\enspace\enspace\enspace\enspace
\begin{tikzpicture}
\node (A)  at (0,0.8) {${\underline{\EuScript C}}$};
\node (B) at (2.5,0.8) {${\EuScript M}({\underline{\EuScript C}})$};
\node (C) at (1.25,-0.5) {${\underline{\EuScript C}}$};
\node (D) at (1.25,-1.1) {};
\path[->,font=\footnotesize]
(A) edge node[above]{$\eta_{{\underline{\EuScript C}}}$} (B)
(A) edge node[left]{$id_{{\underline{\EuScript C}}}$} (C)
(B) edge node[right]{$\delta$} (C);
\end{tikzpicture}
\end{center}
\vspace{-0.1cm}
called the multiplication and the unit law for $\delta$, which allows us to verify the axioms from  Definition \ref{entriesonly}  as follows.\\
\indent  For the proof of  {\small \texttt{(A1)}}, let $a$ and $b$ be as above, let $c\in {\underline{\EuScript C}}(Z)$, $z\in Z$ and $u\in Y$.  Suppose   that  $a$, $b$ and $c$ are all different from identity and that $X$, $Y$ and $Z$ are mutually disjoint (only to avoid the renaming technicalities).  We will chase the multiplication diagram  above two times,  starting with   two-level unrooted trees $${\cal T}_1=\{[\{a(x,\dots),b(y,u,\dots);\sigma_1'\}]_{\alpha},[\{c(z,\dots);{\it id}_Z\}]_{\alpha};\sigma_1\}$$  and $${\cal T}_2=\{[\{a(x,\dots);{\it id}_X\}]_{\alpha}, [\{ b(y,u,\dots),c(z,\dots);\sigma_2'\}]_{\alpha};\sigma_2\},$$ where $\sigma_1'=(x\,\,y)$, $\sigma_1=(u\,\,z)$, $\sigma_2'=(u\,\,z)$ and $\sigma_2=(x\,\,y)$.  If we start with   ${\cal T}_1$, then, by chasing the diagram to the right-down, the action of ${\EuScript M}\delta$ corresponds to the action of $\delta$ on $[\{a(x,\dots),b(y,u,\dots);\sigma_1'\}]_{\alpha}$ and $[\{c(z,\dots);{\it id}_Z\}]_{\alpha}$ separately. Followed by the action of $\delta$ again, we get the following sequence   
$$
{\cal T}_1 \overset{{\EuScript M}\delta}{\longmapsto}\{(a\, {_{x}\circ_{y}}\,\, b)(u,\dots),c(z,\dots);\sigma\}\overset{\delta}{\longmapsto} (a\, {_{x}\circ_{y}}\,\, b)\,\,{_{u}\circ_z}\, c .$$
In the other direction, the action of the monad multiplication flattens ${\cal T}_1$, the resulting tree already being in normal form.  Followed by the action of $\delta$, we obtain the  sequence: 
$${\cal T}_1  \overset{\mu_{\underline{\EuScript C}}}{\longmapsto} \{a(x,\dots),b(y,u,\dots),c(z,\dots);\underline{\sigma}\} \overset{\delta}{\longmapsto} \delta(\{a(x,\dots),b(y,u,\dots),c(z,\dots);\underline{\sigma}\}).
$$
Hence, $$(a\, {_{x}\circ_{y}}\,\, b)\,\,{_{u}\circ_z}\, c=\delta(\{a(x,\dots),b(y,u,\dots),c(z,\dots);\underline{\sigma}\}).$$ The diagram chasing with respect to ${\cal T}_2$ gives us  that $$a\, {_{x}\circ_{y}}\, (b\,\,{_{u}\circ_z}\, c)= \delta(\{a(x,\dots),b(y,u,\dots),c(z,\dots);\underline{\sigma}\}).$$  Therefore,  $(a\, {_{x}\circ_{y}}\,\, b)\,\,{_{u}\circ_z}\, c=a\, {_{x}\circ_{y}}\, (b\,\,{_{u}\circ_z}\, c)$.\\[0.1cm]
\indent The axiom {\small{\texttt{(CO)}}} follows directly by the commutativity of $_x\bullet_y$. \\[0.1cm]
\indent The axiom {\small{\texttt{(EQ)}}} holds by the equivariance of ${_{x}\bullet_{y}}$ and the naturality of $\eta$ and $\delta$. For $\sigma_1$, $\sigma_2$ and $\sigma$   as in {\small{\tt{(EQ)}}}, and denoting $Z=X'\backslash\{\sigma^{-1}_1(x)\}\cup  Y'\backslash\{\sigma^{-1}_2(y)\}$, we have     $$\begin{array}{rcl}
a^{\sigma_1}\,\,{_{{ {\sigma_1^{-1}}(x)}}\circ_{\sigma_2^{-1}(y)}}\,\, b^{\sigma_2}&=&\delta_{Z}({\eta_{\underline{\EuScript C}}}_{X'}(a^{\sigma_1})\,\,{_{{ {\sigma_1^{-1}}(x)}}\bullet_{\sigma_2^{-1}(y)}}\,\, {\eta_{\underline{\EuScript C}}}_{Y'}(b^{\sigma_2}))\\[0.1cm]
&=&\delta_Z(\eta_X(a)^{\sigma_1}\,\,{_{{ {\sigma_1^{-1}}(x)}}\bullet_{\sigma_2^{-1}(y)}}\,\, \eta_Y(b)^{\sigma_2})\\[0.1cm]
&=&\delta_Z((\eta_X(a)\, {_{x}\bullet_{y}}\, \eta_Y(b))^{\sigma}) \\[0.1cm]
&=&\delta_Z(\eta_X(a)\, {_{x}\bullet_{y}}\, \eta_Y(b))^{\sigma}\\[0.1cm]
&=&(a\, {_{x}\circ_{y}}\, b)^{\sigma}.
\end{array}\vspace{-0.15cm}$$

\indent The unit axioms {\small{\texttt{(U1)}}} and {\small{\texttt{(U3)}}} are verified by the corresponding  axioms for $_x\bullet_y$, the commuting triangle and  the naturality of $\sigma$ and $\eta$: for the bijection $\kappa$ that renames $x$ to $z$, and denoting $X'=X\backslash\{x\}\cup \{z\}$, we have\vspace{-0.1cm} 
$$\begin{array}{rcl}
a\, {_{x}\circ_{y}}\, {\it id}_{y,z}&=&\delta_{X'}([\{a(x,\dots);{\it id}_X\}]_{\alpha} \, {_{x}\bullet_{y}}\,  [\{(y,z);{\it id}_{\{y,z\}}\}]_{\alpha} )\\[0.1cm]
&=&\delta_{X'}({\eta_{\underline{\EuScript C}}}_X(a)^{\kappa})\\[0.1cm]
&=&\delta_{X'}({\eta_{\underline{\EuScript C}}}_{X'}(a^{\kappa}))\\[0.1cm]
&=&a^{\kappa},\vspace{-0.1cm}
\end{array}$$
and, for   $\sigma:\{u,v\}\rightarrow \{x,y\}$, we have 
$$\begin{array}{rcl}
{\it id}_{x,y}^{\sigma}&=&\delta_{\{x,y\}}([\{(x,y);{\it id}_{\{x,y\}}\}]_{\alpha})^{\sigma}\\[0.1cm]
&=&\delta_{\{u,v\}}([\{(x,y);{\it id}_{\{x,y\}}\}]^{\sigma}_{\alpha})\\[0.1cm]
&=&\delta_{\{u,v\}}([\{(u,v);{\it id}_{\{u,v\}}\}]_{\alpha})\\[0.1cm]
&=&{\it id}_{u,v}.
\end{array}$$

\indent In the other direction, we define $\delta:{\EuScript M}({\underline{\EuScript C}})\rightarrow{{\underline{\EuScript C}}}$ as the map induced by the interpretation of the $\mu$-syntax in the cyclic operad $\underline{\EuScript C}$, i.e. by the   composition  
of $[[\rule{.4em}{.4pt}]]:\mu$\texttt{Exp}$_{{\underline{\EuScript C}}}\rightarrow$ \texttt{cTerm}$_{{\underline{\EuScript C}}}$ and $[\rule{.4em}{.4pt}]_{{{\EuScript C}}}:\mbox{\texttt{cTerm}}_{{\underline{\EuScript C}}}\rightarrow {{\underline{\EuScript C}}}$. Therefore, with  $\Phi$ being defined as in the proof of Theorem \ref{equiv}, we set
$$\delta({\EuScript T})= [\,[[c]]\,]_{{\underline{\EuScript C}}},\enspace\enspace\mbox{where}\; c\;\mbox{is any command of $\mu$\texttt{Exp}$_{{\underline{\EuScript C}}}$ such that}\;\underline{\Phi}(c)=[{\EuScript T}]_{\alpha}.$$
Note that  this definition is valid by Theorem \ref{equiv}.
We verify that $\delta$   satisfies the equations of an ${\EuScript M}$-algebra  on simple examples. The general case follows naturally. 
Let $${\cal T}=\{[\{a(x_1,\dots,x_n),b(y_1,\dots,y_m);\sigma_1\}]_{\alpha},[\{d(z_1,\dots,z_p);{\it id}_Z\}]_{\alpha};\sigma \}$$ be a two-level unrooted tree such that $\sigma_1(x_i)=y_j$,  and $\sigma(y_k)=z_l$,   and suppose, say,  that $$\underline{\Phi}(\underline{a}\{t_1,\dots,t_n\})=[\{a(x_1,\dots,x_n),b(y_1,\dots,y_m);\sigma_1\}]_{\alpha}$$ and $$\underline{\Phi}(\underline{d}\{s_1,\dots,s_p\})=[\{d(z_1,\dots,z_p);{\it id}_Z\}]_{\alpha}.$$ 
\indent By chasing the multiplication diagram to the right-down, the action of ${\EuScript M}\delta$ provides the interpretations of the commands that correspond to each of the corollas of ${\cal T}$. Thus, setting $[[\underline{a}\{t_1,\dots,t_n\}]]=a(\varphi)$ and $[[\underline{d}\{s_1,\dots,s_p\}]]=d(\tau)$,  we get that  $${\EuScript M}\delta([{\cal T}]_{\alpha})=\{[a(\varphi)]_{\underline{\EuScript C}}(x_1,\ldots,x_{i-1},x_{i+1},\dots,x_n,y_1,\dots,y_{j-1},y_{j+1},\dots,y_m),[d(\tau)]_{\underline{\EuScript C}}(z_1,\dots,z_p);\sigma\}.$$  If now
  $$\underline{\Phi}(\underline{[a(\varphi)]_{\underline{\EuScript C}}}\{k_1,\dots,k_{n+m-2}\})={\EuScript M}\delta([{\cal T}]_{\alpha}),$$ then, by setting
  $[[\underline{[a(\varphi)]_{\underline{\EuScript C}}}\{k_1,\dots,k_{n+m-2}\}]]=[a(\varphi)]_{\underline{\EuScript C}}(\psi)$, we get
  $$\delta({\EuScript M}\delta([{\cal T}]_{\alpha}))={[a(\varphi)}(\psi)]_{\underline{\EuScript C}}.$$ 
\indent By chasing the multiplicaiton diagram to the down-left, we first get  $$\mu_{{\underline{\EuScript C}}}([{\cal T}]_{\alpha})=\{a(x_1,\dots,x_n),b(y_1,\dots,y_m),d(z_1,\dots,z_p);\underline{\sigma}\}\, $$
We shall construct a  command $c$, such that $\underline{\Phi}(c)=\mu_{\underline{\EuScript C}}({\cal T})$, in the way guided by the choices we made in chasing the diagram to the right-down. More precisely, in that direction, $a$ was the corolla of $\{a(x_1,\dots,x_n),b(y_1,\dots,y_m);\sigma_1\}$ chosen in constructing the corresponding command, and $d$ was the one for $\{d(z_1,\dots,z_p);{\it id}_Z\}$, and then, in the next step, $[a(\varphi)]_{\underline{\EuScript C}}$ was the  chosen corolla of  ${\EuScript M}\delta([{\cal T}]_{\alpha})$. Therefore, we set   $c=\underline{a}\{\sigma\}$, where $$\sigma(x_i)=\mu y_j.\underline{b}\{y_1,\dots,y_{k-1},\mu z_l.\underline{d}\{z_1,\dots,z_p\},y_{k+1},\dots ,y_m\}.$$  Thus, setting $[[\underline{a}\{\sigma\}]]=a(\xi)$, we get $$\delta(\mu_{{\underline{\EuScript C}}}({\cal T}))=[a(\xi)]_{\underline{\EuScript C}}$$ as a result of chasing the diagram to the down-left. The equality ${a(\varphi)}(\psi)=a(\xi)$  follows   by Lemma \ref{geneq}.(b). \\[0.1cm]
\indent As for the unit diagram, if $a\in\underline{\EuScript C}(X)$, where $X=\{x_1,\dots,x_n\}$, then ${\eta_{\underline{\EuScript C}}}_X(a)=\linebreak\{a(x_1,\dots,x_n);  {\it id}_X\}$, and, since $[\{a(x_1,\dots,x_n);{\it id}_X\}]_{\alpha}=\underline{\Phi}(\underline{a}\{x_1,\dots,x_n\})$, we have that $$\delta_X({\eta_{\underline{\EuScript C}}}_X(f))=[\,[[\underline{a}\{x_1,\dots,x_n\}]]\,]_{\underline{\EuScript C}}=a.\vspace{-0.1cm}$$
This completes the proof.


\begin{thebibliography}{000000}
\bibitem[BN99]{rewriting} F. Baader, T. Nipkow, {\em Term Rewriting and All That}, Cambridge University Press, 1999.
\bibitem[CH00]{doc}  P. -L. Curien, H. Herbelin, {\em The duality of computation}, ACM SIGPLAN Notices,
Volume 35 Issue 9, 233-243, September 2000. \\[-0.55cm]
\bibitem[G09]{ezra}
   E. Getzler,  {\em Operads revisited}, Algebra, arithmetics, and geometry: in honor of Yu. I. Manin, Vol. I, volume 269 of Progr. Math. p. 675-698. Birkh\" auser Boston Inc., Boston MA, 2009. \\[-0.55cm]
\bibitem[GK95]{Getzler:1994pn}
   E. Getzler, M. Kapranov, {\em Cyclic operads and cyclic homology}, Geom., Top., and Phys. for Raoul Bott, International Press, Cambridge, MA, 167-201, 1995. \\[-0.55cm]
\bibitem[J81]{joyal} A. Joyal, Une th\' eorie combinatoire des s\' eries formelles, {\em Advances in Mathematics}, {\bf 42} 1-82, 1981. 
\bibitem[JK11]{kock} A. Joyal, J. Kock, {\em Feynman Graphs, and Nerve Theorem for
Compact Symmetric Multicategories}
(Extended Abstract), Electronic Notes in Theoretical Computer Science 270 (2)  105-113, 2011.
\bibitem[KW17]{Kaufmann} R. M. Kaufmann, B. C. Ward, {\em Feynman categories}, \href{https://arxiv.org/abs/1312.1269}{ 	arXiv:1312.1269v3} \\[-0.55cm]
\bibitem[KM94]{gra} M. Kontsevich, Y. Manin. {\em Gromov-Witten Classes, Quantum Cohomology, and Enumerative Geometry},
Comm. Math. Phys., 164:525-562, February 1994. Preprint, hep-th/9402147.\\[-0.55cm]
\bibitem[L07]{lam} F. Lamarche, {\em On the algebra of structural contexts},  Mathematical Structures in Computer Science, Cambridge University Press,  51 p, 2003.\\[-0.55cm]
\bibitem[Man99]{manin} Y. I. Manin, {\em Frobenius manifolds, quantum cohomology, and moduli spaces}, volume 47 of AMS Colloquium Publications, American mathematical Society, Providence, RI, 1999.
\bibitem[Mar96]{mmm}  M. Markl, {\em Models for operads}, Comm. Algebra, 24(4):1471–1500, 1996.\\[-0.55cm]
\bibitem[Mar16]{mm} M. Markl,  {\em Modular envelopes, OSFT and nonsymmetric (non-$\Sigma$) modular operads},  J. Noncommut. Geom. 10, 775-809, 2016.\\[-0.55cm]
\bibitem[Mar08]{opsprops} M. Markl, {\em Operads and PROPs},  Elsevier, Handbook for Algebra,
Vol. 5, 87-140, 2008.\\[-0.55cm]
\bibitem[MSS02]{mss} M. Markl, S. Schnider, J. Stasheff, {Operads in Algebra, Topology and Physics},  American  Mathematical  Society, Providence, 2002. 
\bibitem[M72]{GILS} J. P. May, {\em The geometry of iterated loop spaces}, volume 271 of Lectures Notes in Mathematics.
Springer-Verlag, Berlin, 1972. 
\bibitem[O17]{mo}  J. Obradovi\' c, {\em Monoid-like definitions of cyclic operads},   Theory and Applications of Categories,
Vol.  32, No.  12,  pp.  396-436, 2017.  
\bibitem[P02]{types} B. C. Pierce, {\em Types and Programming Languages}, The MIT Press, 2002.
\end{thebibliography}
\end{document}